\documentclass[a4paper]{amsart}
\usepackage[ps2pdf]{hyperref}
\usepackage{amsmath,amssymb}
\usepackage{bbold}
\usepackage{psfrag}
\usepackage{graphicx}
\usepackage{subfigure}
\usepackage{leftidx}
\usepackage[all,poly]{xy}
\usepackage[latin1]{inputenc}

\usepackage[dvipsnames]{xcolor}

\newtheorem{thm}{Theorem}[section]
\newtheorem{cor}[thm]{Corollary}
\newtheorem{lem}[thm]{Lemma}

\theoremstyle{definition}

\def\1{\hbox{\rm\rlap {1}\hskip .03in{\rm I}}}

\newcommand{\co}{\colon}
\newcommand{\id}{\mathrm{id}}

\newcommand{\cyl}{\mathrm{cyl}}

\newcommand{\cc}{\mathcal{C}}
\newcommand{\cch}{\mathcal{C}_{\mathrm{hom}}}
\newcommand{\bb}{\mathcal{B}}

\newcommand{\Cob}{\mathrm{Cob}}
\newcommand{\Fac}{\mathrm{Fac}}

\newcommand{\ee}{\mathcal{E}}

\newcommand{\uu}{\mathcal{U}}

\newcommand{\zz}{\mathcal{Z}}

\newcommand{\CC}{\mathbb{C}}

                     \newcommand{\RR}{\mathbb{R}}

\newcommand{\Mod}{\mathrm{Mod}}

\newcommand{\iso}{\stackrel{\sim}{\longrightarrow}}

\newcommand{\kk}{\Bbbk}
\newcommand{\kt}{$\Bbbk$\nobreakdash-\hspace{0pt}}

\newcommand{\ti}{\mbox{-}\,}

\newcommand{\R}{\mathbb{R}}
\newcommand{\un}{\mathbb{1}}
\newcommand{\T}{\mathbb{T}}

\newcommand{\Ob}{\mathrm{Ob}}
\newcommand{\Int}{\mathrm{Int}}

\newcommand{\Aut}{\mathrm{Aut}}

\newcommand{\End}{\mathrm{End}}
\newcommand{\Hom}{\mathrm{Hom}}

\newcommand{\tr}{\mathrm{tr}}

\newcommand{\inv}{\mathbb{F}}

\newcommand{\lev}{\mathrm{ev}}

\newcommand{\rev}{\widetilde{\mathrm{ev}}}
\newcommand{\lcoev}{\mathrm{coev}}
\newcommand{\rcoev}{\widetilde{\mathrm{coev}}}
\newcommand{\ldual}[1]{#1^{*}}

\newcommand{\scaledraw}[1]{A}
\newcommand{\scaleraisedraw}[2]{A}
\newcommand{\rsdraw}[3]{\raisebox{-#1\height}{\scalebox{#2}{\includegraphics{#3.eps}}}}
\newcommand{\labela}{\renewcommand{\labelenumi}{{\rm (\alph{enumi})}}}
\newcommand{\labeli}{\renewcommand{\labelenumi}{{\rm (\roman{enumi})}}}

\newcommand{\X}{\mathbf{X}}
\newcommand{\x}{\mathbf{x}}
\newcommand{\cl}{\mathbf{c}}

\begin{document}

\title[On 3-dimensional homotopy quantum field theory  III]{On   3-dimensional homotopy quantum  field theory~III:   comparison of two approaches}
\author[V. Turaev]{Vladimir Turaev}
 \address{%
 Vladimir Turaev\newline
  \indent Department of Mathematics, Indiana University \newline
  \indent Bloomington IN47405 USA \newline
  \indent  and \newline
  \indent IRMA, 7 rue Rene Descartes \newline
  \indent 67084 Strasbourg France}
\author[A. Virelizier]{Alexis Virelizier}
\address{%
 Alexis Virelizier\newline
     \indent Universit\'e de Lille\newline
     \indent Laboratoire Paul Painlev\'e (CNRS, UMR 8524)\newline
     \indent F-59000 Lille  France\newline
     \indent e-mail:   alexis.virelizier@univ-lille.fr}
\date{\today}

\begin{abstract}
Let $G$ be a discrete group and $\cc$ be an  additive spherical $G$-fusion category. We prove that the state sum 3-dimensional  HQFT derived from $\cc$ is isomorphic to the  surgery 3-dimensional HQFT derived from the $G$-center of~$\cc$.
\end{abstract}
\maketitle

\section{Introduction}\label{sect-Intro}

Homotopy Quantum Field Theories (HQFTs) introduced in \cite{Tu1} generalize   Topological Quantum Field Theories to  manifolds and cobordisms endowed with  homotopy classes of maps   to a fixed target space. We focus here on  3-dimensional HQFTs with   target space   the   (pointed)  Eilenberg-MacLane space $ K(G,1)$ where~$G  $ is a discrete group. Note that the homotopy classes of maps from a manifold $M$ to  $K(G,1)$  bijectively correspond to isomorphism classes of principal $G$-bundles over~$M$. In \cite{TVi2} we defined    spherical $G$-fusion categories over a commutative ring~$\kk$ and  showed that  such   a  category $\cc$   satisfying a non-degeneracy condition  determines       a 3-dimensional HQFT $\vert \cdot \vert_\cc$ over~$\kk$.   The non-degeneracy  condition in question requires the dimension $\dim(\cc_1) \in \kk$ of the neutral component  $\cc_1$ of~$\cc$ to be   invertible in~$\kk$. The construction of $\vert \cdot \vert_\cc$ uses    Turaev-Viro-Barrett-Westbury-type   state sums on skeletons of 3-manifolds.   In~\cite{TVi4}, we  defined $G$-modular categories over~$\kk$ and showed that any   such category $\bb$
endowed with a square root   of  $\dim(\bb_1)$ in~$\kk$  determines   a 3-dimensional HQFT $\tau_\bb$  over~$\kk$.  The construction of $\tau_\bb$ uses Reshetikhin-Turaev-type surgery formulas. The main result of the present paper is an isomorphism   between the   HQFTs $\vert \cdot \vert_\cc$ and $\tau_\bb$ provided $ \bb=\zz_G(\cc)$  is the $G$-center of~$\cc$ and~$\kk$ is an algebraically closed field.

To be more  specific,   consider the cobordism category   $\Cob^G $ whose objects are $G$-surfaces, that is, pointed closed oriented surfaces endowed with homotopy classes of maps to $ K(G,1)$. Morphisms in $\Cob^G$ are represented by $G$-cobordisms defined as   compact oriented 3-dimensional cobordisms  endowed with homotopy classes of maps to $K(G,1) $ (for a precise definition of $\Cob^G $, see Section~\ref{sect-CobG-v0} below).  The category $\Cob^G$ has a natural structure of  a symmetric monoidal category.
For a commutative ring~$\kk$, the category $\Mod_\kk$ of $\kk$-modules and $\kk$-homomorphisms   is a symmetric monoidal category with   monoidal   product $\otimes_\kk$  and unit object $\kk$. A 3-dimensional HQFT   over~$\kk$ is a symmetric strong monoidal functor  $\Cob^G \to \Mod_\kk$.
A  precise definition of  a 3-dimensional  HQFT  involves Lagrangian spaces in homology of surfaces and $p_1$-structures in  cobordisms (see \cite{Tu1}).  However, the HQFTs studied in this paper do not depend on this data and we ignore it from now on.

We now state our main results.
Consider    an additive  spherical $G$-fusion category~$\cc=\oplus_{g\in G} \, \cc_g$  over an algebraically closed field~$\kk$  such that $\dim(\cc_1) \neq 0$.   Consider the  center $\zz_G(\cc)=\oplus_{g\in G} \, \zz_g(\cc)$ of $\cc$ relative to  $\cc_1$ (see \cite{GNN} for   finite~$G$ and     \cite{TVi3} for   all~$G$). By    \cite{TVi3},   $\zz_G(\cc)$ is an additive $G$-modular category.  Thus~$\cc$ gives rise to two 3-dimensional HQFTs over~$\kk$: the state sum HQFT  $\vert \cdot \vert_\cc $ and the surgery HQFT $\tau_{\zz_G(\cc)}$ determined by the square root  $ \dim(\cc_1)$ of $ \dim(\zz_1(\cc)) =(\dim(\cc_1))^2 $.

\begin{thm}\label{thm-comparison-intro}
The HQFTs $\vert \cdot \vert_\cc$ and  $\tau_{\zz_G(\cc)}$   are  isomorphic.
\end{thm}

Theorem~\ref{thm-comparison-intro} means that one can pick  for each $G$-surface~$\Sigma$ a $\kk$-linear isomorphism of modules
 $ \vert \Sigma
\vert_\cc  \simeq \tau_{\zz_G(\cc)}(\Sigma) $  so that the resulting family of isomorphisms is  compatible with disjoint unions of $G$-surfaces and  the action of $G$-cobordisms.
For a  closed   oriented 3-manifold~$M$ endowed with a homotopy class of maps to $K(G,1)$,  Theorem~\ref{thm-comparison-intro} gives
$$
\vert M \vert_\cc=\tau_{\zz_G(\cc)}(M)\in \kk.
$$
For   $G=\{1\}$, Theorem~\ref{thm-comparison-intro}   was    established in \cite{TVi1} and independently in \cite{Ba2}.

We prove Theorem~\ref{thm-comparison-intro} by first introducing the notion of a graph HQFT over a crossed $G$-graded category (see Section~\ref{sect-def-graph-HQFT}) and then by deducing Theorem~\ref{thm-comparison-intro} from the following stronger claim:

\begin{thm}\label{thm-comparison-intro-general}
Both HQFTs $\vert \cdot \vert_\cc$ and  $\tau_{\zz_G(\cc)}$ extend to graph HQFTs over $\zz_G(\cc)$. The resulting graph HQFTs are  isomorphic to each other.
\end{thm}

Theorem~\ref{thm-comparison-intro-general} is the main result of this paper and is a combination of several claims established in Section~\ref{sect-surg-HQFT} and of Theorem~\ref{thm-extension-state-sum-graph-HQFT} and Theorem~\ref{thm-comparison}.
The extension of the state sum HQFT to a graph HQFT (Theorem~\ref{thm-extension-state-sum-graph-HQFT})
uses an invariant of colored knotted nets in the 2-dimensional sphere (see Section~\ref{sect-prelimoncolorednets})
which plays the role of the familiar $6j$-symbols. The comparison of the surgery and state sum graph HQFTs (Theorem~\ref{thm-comparison}) is based on a surgery formula  for  graph HQFTs
via so-called torus vectors
(see Section~\ref{sect-sugery-computation}).
In the appendix, we recall the braided structure of~$\zz_G(\cc)$.\\

We fix throughout the paper a non-zero commutative ring~$\kk$, a discrete  group~$G$,  and an Eilenberg-MacLane  space  $\X=K(G,1)$ with base point $\x\in \X$.  Thus, $\X$ is a   connected aspherical CW-space and   $\pi_1(\X,\x)=G$.\\

\subsection*{Acknowledgments}
V.\@ Turaev acknowledges support from the NSF grant DMS-1664358. A.\@ Virelizier 
acknowledges support from the Labex CEMPI (ANR-11-LABX-0007-01) and from the FNS-ANR grant OChoTop (ANR-18-CE93-0002).

\section{Graded monoidal categories}\label{sect-G-categories}
We   recall      the  basic notions of the theory of graded and   crossed   categories referring to \cite{ML1}, \cite{EGNO}, \cite{Tu1}, \cite{TVi3}-\cite{TVi5} for details and   proofs.

\subsection{Pivotal categories}\label{sect-pivotal-cat}\label{sect-spherical-cat}\label{sect-lin-pivotal-cat}
A \emph{pivotal} category is a monoidal category $\cc=(\cc,\otimes,\un)$ such that  each
object $X$ of $\cc$ has a \emph{dual
object}~$X^*\in \cc$ and four morphisms
\begin{align*}
& \lev_X \co X^*\otimes X \to\un,  \qquad \lcoev_X\co \un  \to X \otimes X^*,\\
&   \rev_X \co X\otimes X^* \to\un, \qquad   \rcoev_X\co \un  \to X^* \otimes X,
\end{align*}
satisfying   several   conditions which say, in summary,  that the associated  left/right dual functors   coincide as   monoidal functors. In particular, each morphism $f\co X \to Y$ in~$\cc$ has a dual morphism $f^*\co Y^* \to X^*$  computed by
\begin{align*}
f^*&=(\lev_Y \otimes  \id_{X^*})(\id_{Y^*}  \otimes f \otimes \id_{X^*})(\id_{Y^*}\otimes \lcoev_X)\\
 &= (\id_{X^*} \otimes \rev_Y)(\id_{X^*} \otimes f \otimes \id_{Y^*})(\rcoev_X \otimes \id_{Y^*}).
\end{align*}

We shall  omit brackets in  monoidal products   and
suppress the associativity constraints $(X\otimes Y)\otimes Z\cong
X\otimes (Y\otimes Z)$, the unitality constraints $X\otimes \un
\cong X\cong \un \otimes X$, and the duality constraints
$X^* \otimes Y^*\cong (Y\otimes X)^* $ and $\un^* \cong \un$.
This does not lead  to   ambiguity because, by the Mac Lane coherence theorem,  all legitimate ways of inserting these constraints   give the same results.

The \emph{left} and  \emph{right traces} of an endomorphism $f$ of an object
$X$ of a pivotal category $\cc$ are defined by
$$
\tr_l(f)=\lev_X(\id_{\ldual{X}} \otimes f) \rcoev_X  \quad {\text {and}}\quad \tr_r(f)=  \rev_X( f \otimes
\id_{\ldual{X}}) \lcoev_X .
$$
Both traces take values in   the commutative monoid   $\End_\cc(\un)$. The \emph{left} and  \emph{right dimensions} of an object $X\in \cc$   are defined by
$$
\dim_l(X)=\tr_l(\id_X) \in\End_\cc(\un)  \quad \text{and} \quad \dim_r(X)=\tr_r(\id_X) \in\End_\cc(\un).
$$

A \emph{spherical category} is a pivotal category such that the left and
right traces of any endomorphism of any object  are equal. In a spherical category, the \emph{trace} of an endomorphism~$f$ and the \emph{dimension} of an object~$X$ are  defined by
$$
\tr(f)=\tr_l(f)=\tr_r(f)
\quad\text{and}\quad
\dim(X)=\dim_l(X)=\dim_r(X).
$$

\subsection{Linear categories} A monoidal category~$\cc$ is \emph{\kt linear}  if for all $X,Y\in \cc$, the set $\Hom_\cc(X,Y)$ carries a structure of a left  \kt module so that  both  the composition and the monoidal product of morphisms    are $\kk$-bilinear.
Note that then the monoid $\End_\cc(\un)$ is a commutative \kt algebra.

 In the rest of this section we focus on \kt  linear pivotal  categories.
 For  such a category $\cc$,   we  let
$\Aut(\cc)$ be the category of  pivotal strong monoidal \kt linear auto-equivalences of~$\cc$.  The objects of the category $\Aut(\cc)$  are  pivotal  strong monoidal functors
$\cc\to \cc$   which   are   $\kk$-linear on the $\Hom$-sets and are equivalences of categories. The  morphisms in  $\Aut(\cc)$ are
 monoidal natural isomorphisms   of such functors.   Then $\Aut(\cc)$ is a strict monoidal category with
 monoidal product being the composition of functors and
monoidal unit being the identity endofunctor of~$\cc$.

\subsection{Graded categories}\label{sect-graded-pivotal-cat}
In this paper, by a  \emph{$G$-graded  category} $\cc$ (over $\kk$), we mean a  \kt  linear pivotal  category   endowed with pairwise disjoint full   subcategories $\{\cc_{\alpha}\}_{{\alpha}\in G}$ such that:
\begin{enumerate}
\labeli
\item if $X\in \cc_{\alpha}$ and $Y\in \cc_{\beta}$ with ${\alpha}\neq {\beta}$, then $\Hom_{\cc} (X,Y)=0$;
\item if $X\in \cc_{\alpha}$ and  $Y\in \cc_{\beta}$, then $X\otimes Y\in \cc_{{\alpha}{\beta}}$;
\item the unit object $\un$ of~$\cc$ belongs to $ {\cc}_1$  where $1\in G$ is the group unit;
\item if $X \in \cc_\alpha$, then  $X^* \in \cc_{\alpha^{-1}}$.
\end{enumerate}
The category $\cc_1$ corresponding to $1\in G$ is called the {\it neutral component}
of~$\cc$.

An object $X$ of  a $G$-graded  category ${\cc} $ is   {\it homogeneous} if   $X\in \cc_\alpha$ for some $\alpha\in
G$. Such an $\alpha$  is then uniquely determined by $X$ and   is   denoted by  $|X|$.  We let $\cch=\amalg_{\alpha\in G}\,  \cc_\alpha$ be
the full subcategory of homogeneous objects of $\cc$.   Clearly,  $\cch$ is itself a $G$-graded  category  and   $(\cch)_\mathrm{hom}=\cch$.
 Note  that two objects $X \in \cc_\alpha$, $ Y\in \cc_\beta $ with $\alpha\neq \beta$ may be  isomorphic but then both are zero objects in the sense that  $\id_X=0$ and $\id_Y=0$.

 A  $G$-graded  category $\cc$ is \emph{additive} if any finite (possibly empty) family of objects of~$\cc$ has a direct sum in~$\cc$. In this case, we write  $\cc=\oplus_{\alpha\in G}\,  \cc_\alpha$.

\subsection{Crossed  categories}\label{sect-crossed-graded-cat}
  Let  $\overline{G}$ be  the category whose objects are elements of
the group $G$ and morphisms are identities. We view $\overline{G}$ as a strict monoidal category with
monoidal product
$\alpha\otimes \beta= \beta \alpha$ for all $\alpha, \beta \in G$.

A \emph{crossing}\footnote{In this paper, crossings correspond to pivotal crossings in \cite{TVi3,TVi4}.}  of a $G$-graded  category~$\cc$ is a
strong monoidal functor  $\varphi\co \overline{G} \to
\Aut(\cc)$  such that $ \varphi_\alpha(\cc_\beta) \subset
\cc_{\alpha^{-1} \beta \alpha} $ for all $\alpha,\beta \in G$.  The condition  that $\varphi_\alpha\co \cc \to \cc$ is a  pivotal  strong monoidal functor means that it comes equipped with
      natural isomorphisms
\begin{align*}
&(\varphi_\alpha)_0 \co \un \iso \varphi_\alpha(\un),\\
&(\varphi_\alpha)_2=\{(\varphi_\alpha)_2(X,Y) \co \varphi_\alpha(X) \otimes \varphi_\alpha(Y) \iso \varphi_\alpha(X \otimes Y)\}_{X,Y \in \cc},\\
&\varphi_\alpha^1=\{\varphi_\alpha^1(X) \co \varphi_\alpha(X^*) \iso (\varphi_\alpha(X))^*\}_{X \in \cc}.
\end{align*}
 The condition that the functor  $\varphi\co \overline{G} \to
\Aut(\cc)$   is strong monoidal means that it comes equipped with
 natural isomorphisms
\begin{align*}
&\varphi_2=\bigl\{\varphi_2(\alpha,\beta)=\{\varphi_2(\alpha,\beta)_X \co \varphi_\alpha\varphi_\beta(X)\iso \varphi_{\beta\alpha}(X)\}_{X \in \cc}\bigr\}_{\alpha,\beta \in G},\\
&\varphi_0=\{(\varphi_0)_X \co X \iso \varphi_1(X)\}_{X \in \cc}.
\end{align*}
These isomorphisms    should   satisfy appropriate compatibility conditions, see \cite{TVi4}.

 A \emph{$G$-crossed category} $(\cc, \varphi)$  is a pair consisting of a $G$-graded category $\cc$ and a crossing~$\varphi$ of~$\cc$.
Then for any $\alpha \in G$ and any integer $n \geq 3$,  we have a
 natural transformation
$$
(\varphi_\alpha)_n=\{(\varphi_\alpha)_n(X_1,\dots,X_n) \co \varphi_\alpha(X_1) \otimes \cdots \otimes \varphi_\alpha(X_n) \to \varphi_\alpha(X_1 \otimes \cdots \otimes X_n)\}
$$
where $X_1,\dots,X_n \in \cc$.  It is obtained by composing   monoidal    products of the  transformations $(\varphi_\alpha)_2$ and the identity
morphisms.   For instance, for $n=3$,
$$
(\varphi_\alpha)_3(X_1,X_2 ,X_3)=(\varphi_\alpha)_2(X_1,X_2 \otimes X_3) (\id_{\varphi_\alpha(X_1)} \otimes (\varphi_\alpha)_2(X_2,X_3)).
$$
 We can also use~$\varphi$ to       transform  certain isomorphisms in $\cc$.  Namely, for     any isomorphism $\psi\colon X \to \varphi_\alpha (Y)$ with $  X, Y\in \cc$ and $\alpha\in G$,   we
let    $  \psi^- \colon X^*\to
\varphi_{\alpha} (Y^*)$ be the   composition of the
isomorphisms
\begin{equation*}\label{psi-}
\xymatrix@R=1cm @C=1.7cm {  X^* \ar[r]^-{  (\psi^{-1})^* }     & (\varphi_\alpha (Y))^*    \ar[r]^-{ (\varphi_\alpha^1 (Y))^{-1} }     &
 \varphi_\alpha (Y^*). }
\end{equation*}
In this context, we will    sometimes write $\psi^+$ for $\psi$.

\subsection{Braided and  ribbon graded  categories}\label{sect-braided-graded-cat}\label{sect-ribbon-graded-cat}
A \emph{$G$-braiding}\footnote{In this paper, $G$-braidings correspond to pivotal $G$-braidings in \cite{TVi3,TVi4}.}  of a  $G$-crossed category $(\cc,\varphi)$ is a family
of isomorphisms
$$
\tau=\{\tau_{X,Y} \co X \otimes Y \to Y \otimes
\varphi_{|Y|}(X)\}_{X \in \cc, Y \in \cch}
$$
which is natural in   $X$, $Y$   and  satisfies  three conditions: two of them  generalizing  the usual braiding relations in braided monoidal categories and the third condition relating~$\tau$  and~$\varphi$, see \cite{TVi4} for details.
A \emph{$G$-braided  category} is a $G$-crossed  category
 endowed with a $G$-braiding.

The \emph{twist} of a $G$-braided category $(\cc,\varphi,\tau)$ is the natural isomorphism $\theta=\{\theta_X\}_{X \in
\cch}$ defined by
$$
\theta_X =(\lev_X \otimes \id_{\varphi_{|X|}(X)})(\id_{X^*} \otimes \tau_{X,X})(\rcoev_X \otimes \id_X)\co X\to \varphi_{|X|}(X).
$$
A \emph{$G$-ribbon category} is a  $G$-braided  category $(\cc,\varphi,\tau)$ whose twist is \emph{self-dual} in the sense that  for all $\alpha\in G$ and all $X \in \cc_\alpha$,
$$
(\theta_X)^*= (\varphi_0)_X^* \, (\varphi_2(\alpha^{-1},\alpha)_X^{-1})^* \, \varphi_{\alpha^{-1}}^1(\varphi_{\alpha}(X)) \, \theta_{\varphi_{\alpha}(X)^*}.
$$
Note that then  the neutral component $\cc_1$  of  $\cc$  is   a ribbon  category   in the usual sense  with braiding
$$
\{c_{X,Y}=(\id_Y \otimes (\varphi_0)^{-1}_X)\tau_{X,Y} \co X \otimes Y \to Y \otimes X\}_{X,Y \in \cc_1}
$$
and   twist $$ \{v_X=(\varphi_0)^{-1}_X \theta_X\co X \to X\}_{X \in \cc_1}.$$  All $G$-ribbon categories are spherical as pivotal categories (see \cite[Section 6.3]{TVi4}).

\subsection{Fusion graded categories}\label{sect-fusion-graded-cat}
An object $X$ of a \kt linear category is \emph{simple} if the \kt module of the endomorphisms of~$X$  is a free \kt module of rank 1 with  basis $\{\id_X\}$.     It is clear that    all  objects
isomorphic to a simple object are  simple  and (in a pivotal category)  the dual of a   simple object is  simple.

A \emph{$G$-fusion category} (over $\kk$) is a $G$-graded category $\cc$  (over $\kk$) such that
there is a set~$I$ of homogeneous simple objects of $\cc$ satisfying the following four conditions:
\begin{enumerate}
\labela
\item for each $\alpha\in G$, the set $I_\alpha\subset I$ of  elements of $I$ belonging to $\cc_\alpha$ is finite and non-empty;
\item   the unit object  $\un$ of $\cc$ belongs to $I_1\subset I$;
\item $\Hom_\cc(i,j)=0$ for any distinct  $i,j \in I $;
\item  every object of $\cc$ is  a   direct sum of a finite family of elements of~$I$.
\end{enumerate}
The set~$I$   is then called a \emph{representative set of simple objects}. Clearly, $I=\amalg_{\alpha\in G}\, I_\alpha$ and  every simple object of~$\cc$ is isomorphic to precisely one   object belonging to~$I$.

Let $\cc$ be a $G$-fusion category with representative set of simple objects~$I$. Condition (b)   above  implies that the map $\kk \to
\End_\cc(\un)$, $k \mapsto k \, \id_\un$  is a \kt algebra isomorphism
which we use  to identify $\End_\cc(\un)=\kk$.
Conditions (c) and~(d) imply that the Hom-sets in~$\cc$ are free of finite rank. The neutral component~$\cc_1$ of~$\cc$ is a  fusion (in the usual sense) \kt linear pivotal category  of  dimension
$$
\dim(\cc_1)=\sum_{i\in I_1} \dim_l(i)\dim_r(i) \in \End_\cc(\un)=\kk.
$$

\subsection{Modular graded categories}\label{sect-modular-graded-cat}
Consider a  $G$-ribbon $G$-fusion  category $\cc$ (over~$\kk$) with representative set  of simple objects~$I $. For $i,j \in I_1\subset I$,   set
$$
S_{i,j}= \tr(c_{j,i}\circ c_{i,j}\colon  i\otimes  j\to
i\otimes j)\in \End_{\cc}(\un)=\kk
$$
where $c_{i,j}\co i\otimes j\to j\otimes i$ is the braiding in $\cc_1$  defined in  Section~\ref{sect-ribbon-graded-cat}.
The symmetric matrix  $S=[S_{i,j}]_{i,j\in I_1} $ is called the {\it $S$-matrix} of $\cc$.  Since each object $ i\in I_1$ is
simple, the   twist  $ v_i\co i \to   i$ in $\cc_1$ (see Section~\ref{sect-ribbon-graded-cat}) expands as   $v_i=\nu_i \,\id_{ i}$ with  $\nu_i\in \kk$.  Since $v_i$ is an isomorphism, $\nu_i$ in invertible in $\kk$.
Set
$$
\Delta_{\pm} =\sum \limits_{i\in I_1}{\nu_i^{\pm 1}(\dim  (i))^2}\in
\kk
$$
where $\dim  (i)=\dim_l  (i)=\dim_r  (i)$ since such a category~$\cc$ is  spherical.

A \emph{$G$-modular category} is a   $G$-ribbon   $G$-fusion  category whose $S$-matrix is invertible over the ground ring~$\kk$.
In other words, a $G$-modular category  is a   $G$-ribbon   $G$-fusion
 category $\cc$ whose neutral component  $\cc_1$  is
modular  in the  sense of \cite{Tu0}.
The  properties of modular categories imply  that then the elements  $\Delta_{\pm} \in \kk$    above  are invertible in $\kk$    and   $\dim(\cc_1) =\Delta_+  \Delta_-$, see \cite[Formula II.2.4.a]{Tu0}. As a consequence, $\dim(\cc_1)$ is invertible in $\kk$.

A $G$-modular category $\cc$ is \emph{anomaly free} if $\Delta_+=\Delta_-$. Then $\Delta=\Delta_+=\Delta_-$ is  a square root   of $\dim(\cc_1)$ called the \emph{canonical rank} of $\cc$.

\subsection{The graded center}\label{sect-center-graded-cat}
By \cite{GNN},  the \emph{$G$-center} $\zz_G(\cc)$ of a  $G$-graded  category~$\cc$ over $\kk$ is the   category obtained as  the (left) center of~$\cc$
 relative to  its neutral component  $\cc_1 \subset \cc$.  The objects of $\zz_G(\cc)$ are  (left) half braidings of $\cc$ relative to~$\cc_1$, that is,   pairs $(A,\sigma)$ where $A \in \cc$ and
$$
\sigma=\{\sigma_X \co A \otimes
X\to X \otimes A\}_{X \in \cc_1}
$$
is a natural family of isomorphisms satisfying
 \begin{equation*}
 \sigma_{X \otimes Y}=(\id_X \otimes
\sigma_Y)(\sigma_X \otimes \id_Y)
\end{equation*} for all $X,Y \in \cc_1$.
A morphism $(A,\sigma)\to (A',\sigma')$ in $\zz_G(\cc)$ is a
morphism $f \co A \to A'$ in $\cc$ such that $(\id_X \otimes f)\sigma_X=\sigma'_X(f \otimes \id_X)$ for all
$X\in\cc_1$. In particular,
$$
\Hom_{\zz_G(\cc)}\bigl((A,\sigma), (A',\sigma')\bigr) \subset \Hom_\cc(A,A').
$$
The monoidal product of $\zz_G(\cc)$ is defined by
\begin{equation*}
(A,\sigma) \otimes (B, \rho)=\bigl(A \otimes B,(\sigma \otimes \id_B)(\id_A \otimes \rho) \bigr)
\end{equation*}
and the unit object of $\zz_G(\cc)$ is the pair $\un_{\zz_G(\cc)}=(\un,\{\id_X\}_{X \in \cc_1})$.

The category $\zz_G(\cc)$  inherits most of the properties of $\cc$.
The given \kt linear structure in~$\cc$ induces a  linear structure in $\zz_G(\cc)$ in the obvious way. The category $\zz_G(\cc)$ is pivotal: the dual of $(A,\sigma) \in \zz_G(\cc)$
is  $(A,\sigma)^*=(A^*,\sigma^\dagger)$, where
\begin{equation*}
 \psfrag{M}[Bc][Bc]{\scalebox{.9}{${{A}}$}}
 \psfrag{X}[Bc][Bc]{\scalebox{.9}{$X$}}
 \psfrag{s}[Bc][Bc]{\scalebox{.9}{$\sigma_{X^*}$}}
\sigma^\dagger_X= \rsdraw{.45}{1}{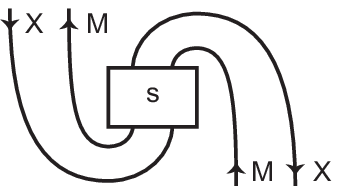}= \;\,\psfrag{s}[Bc][Bc]{\scalebox{.9}{$\sigma^{-1}_{X}$}} \rsdraw{.45}{1}{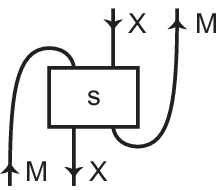} \co {{A}}^* \otimes X \to X \otimes {{A}}^*,
\end{equation*}
with $\lev_{({{A}},\sigma)}= \lev_{{A}}$, $\lcoev_{({{A}},\sigma)}= \lcoev_{{A}}$,
$\rev_{({{A}},\sigma)}= \rev_{{A}}$, $\rcoev_{({{A}},\sigma)}= \rcoev_{{A}}$. Note that the
traces of morphisms and   dimensions of objects
in~$\zz_G(\cc)$ are the same as in~$\cc$ through the inclusion $\End_{\zz_G(\cc)}(\un_{\zz_G(\cc)}) \subset \End_\cc(\un)$.
The category $\zz_G(\cc)$ is $G$-graded:
for $\alpha \in G$, the  component $\zz_\alpha(\cc) \subset \zz_G (\cc)$ is the full subcategory of~$\zz_G(\cc)$ formed by the   half braidings $(A, \sigma)$ as above  with $A\in \cc_\alpha$. In particular, $\zz_1(\cc)$ is the usual  Drinfeld-Joyal-Street center  of $\cc_1$.  If $\cc$ is additive, then so is its  $G$-center, i.e.,   $\zz_G(\cc)=\oplus_{\alpha\in G} \, \zz_\alpha(\cc)$.

The \emph{forgetful functor} $ \uu\co  \zz_G(\cc)\to \cc$ carries $(A, \sigma)$   to $A$ and
acts in the obvious way on the morphisms.   This functor is  strict
monoidal, strict pivotal, \kt linear, and reflects isomorphisms, meaning that a morphism in $\zz_G(\cc)$ carried to an isomorphism in $\cc$ is itself an isomorphism.

If $\cc$ is a $G$-fusion category over a field, then $\zz_G(\cc)$ has a canonical structure of a $G$-braided  category (see \cite[Theorem 4.1]{TVi3}). Furthermore, if $\cc$ is spherical, then~$\zz_G(\cc)$ is $G$-ribbon (see  \cite[Lemma 5.2]{TVi3}).
For completeness, we recall  in Appendix~\ref{sect-appendix} the construction of the crossing and $G$-braiding of $\zz_G(\cc)$ as well as the computation of the twist of $\zz_G(\cc)$.

By \cite[Theorem 5.1]{TVi3} and \cite[Theorem 5.4]{TVi5}, if $\cc$ is an  additive spherical $G$-fusion  category over an algebraically closed field such that $\dim(\cc_1) \neq 0$,  then $\zz_G(\cc)$ is an anomaly free $G$-modular  category    with canonical rank   $\dim(\cc_1)$.

\section{Three-dimensional HQFTs}

We recall   the definition of a 3-dimensional HQFT with target $\X=K(G,1)$.

\subsection{The category $\Cob^G$}\label{sect-CobG-v0}
For an integer $n\geq 0$, by   an  $n$-manifold, we   mean a smooth  $n$-dimensional   manifold
  with boundary (possibly, empty). By convention, the empty set $\emptyset$  is  considered to be an  $n$-manifold for all~$n$.
The boundary $\partial M$  of an $n$-manifold~$M$   is  an   $(n-1)$-manifold  and $\partial (\partial M)=\emptyset$.
If $M$ is oriented, then $\partial M$ is oriented so that     at any   point of $\partial M$, the orientation of~$M$ is given by a direction away from~$M$ followed by the orientation of $\partial M$.    Given an oriented manifold $M$, we let~$-M$ be the same manifold with opposite orientation.
Clearly, $\partial (-M)=-\partial M$. A  \emph{closed} manifold is a compact manifold with empty boundary.

By a \emph{surface} we  mean  a 2-manifold. A surface~$\Sigma$ is \emph{pointed}   if   every  connected   component  of~$\Sigma$    is endowed with
a base point.  The set of base points of~$\Sigma$ is denoted by $\Sigma_\bullet$. A \emph{$G$-surface} is a pair consisting of a pointed closed oriented surface~$\Sigma$  and a homotopy class $g$ of maps   $(\Sigma, \Sigma_\bullet)\to (\X,\x)$.  Reversing orientation in~$\Sigma $, we obtain the \emph{opposite}  $G$-surface $-\Sigma=(-\Sigma, g)$.   The empty set $\emptyset$  is  considered to be a  $G$-surface with $\emptyset_\bullet=\emptyset$ and $-\emptyset =\emptyset$. The disjoint union of a finite family of  $G$-surfaces is a $G$-surface in the obvious way.  An  \emph{isomorphism}  of
$G$-surfaces   $ (\Sigma,g)\to (\Sigma',g')$ is an orientation-preserving
diffeomorphism $f\colon \Sigma\to \Sigma'$ such that $f(  \Sigma_\bullet)= \Sigma'_\bullet$ and $g =g'f$   as homotopy classes of maps $(\Sigma, \Sigma_\bullet)\to (\X,\x)$.

We define a  category $\Cob^G $ following  \cite{TVi2}. The objects of
$  \Cob^G$    are     $G$-surfaces.   For $G$-surfaces $\Sigma_0, \Sigma_1$, a morphism   $\Sigma_0 \to \Sigma_1$ in $\Cob^G$ is represented by a triple $(M,g,h)$ consisting of a compact oriented $3$-manifold~$M$ with  pointed boundary, a homotopy class $g$ of maps $(M, (\partial M)_\bullet ) \to (\X,\x)$, and an isomorphism of $G$-surfaces
$$
h\co (-\Sigma_0) \sqcup \Sigma_1 \to (\partial M, g \vert_{\partial M}).
$$
Such  triples $(M,g,h)$  are called     \emph{$G$-cobordisms between $\Sigma_0$ and $\Sigma_1$}.
Two    $G$-cobordisms  $(M,g,h ), (M', g', {h}'  )$   between $\Sigma_0$ and $\Sigma_1$ represent the same morphism $\Sigma_0\to \Sigma_1$   if there is an orientation-preserving   diffeomorphism  $f\co  M \to M'$   such that  $f( (\partial M)_\bullet)= (\partial M')_\bullet$,  ${g} =g' f$, and $h'=fh$.   Composition of morphisms in
$\Cob^G$ is defined via the obvious gluing of  cobordisms. The identity morphism of a    $G$-surface $\Sigma$   is
represented by the cylinder   $  \Sigma\times [0,1] $ endowed with the
standard identification of its boundary with  $(-\Sigma) \sqcup \Sigma  $.
The  morphisms $\emptyset \to \emptyset$ in $\Cob^G$ are represented by
 \emph{closed 3-dimensional $G$-manifolds}, that is, by pairs $(M,g)$ consisting of a closed oriented $3$-manifold~$M$
  and a homotopy class of maps $g\co M \to \X$.

The category $\Cob^G$ with monoidal product   induced by   disjoint union and  with unit object~$\emptyset$ is   a symmetric monoidal category.

\subsection{HQFTs}\label{sect-CobG-v1}
Let $\Mod_\kk$ be  the symmetric monoidal category of $\kk$-modules and $\kk$-linear homomorphisms.
A \emph{3-dimensional HQFT   over~$\kk$} with target $\X=K(G,1)$ is a symmetric strong monoidal functor
$$
Z \co \Cob^G \to \Mod_\kk.
$$
Such a functor carries the following data: a
$\kk$-module  $Z(\Sigma)$ for every $G$-surface $\Sigma$, a $\kk$-linear homomorphism $Z(M)\co  Z(\Sigma_0) \to Z(\Sigma_1)$ for every $G$-cobordism $M\co \Sigma_0 \to \Sigma_1$, a $\kk$-linear isomorphism  $Z_0\co \kk\simeq Z(\emptyset)$, and $\kk$-linear isomorphisms
$$
Z_2(\Sigma, \Sigma')\co  Z(\Sigma ) \otimes_\kk Z(\Sigma') \simeq Z(\Sigma \sqcup \Sigma')
$$
for any $G$-surfaces $\Sigma, \Sigma'$. This data should satisfy the usual compatibility axioms of a monoidal functor.
Two 3-dimensional HQFTs are \emph{isomorphic} if they are isomorphic as monoidal functors.

A 3-dimensional HQFT $Z$  can be applied to the morphism $\emptyset \to \emptyset$ in $\Cob^G$ represented by a
closed 3-dimensional $G$-manifold $(M,g)$.  The resulting    endomorphism of  $Z(\emptyset) \simeq \kk $   is multiplication by an element of~$\kk$. This element is  denoted  by  $Z(M,g)$ and  is a homeomorphism invariant of the pair  $(M,g)$.
 It is clear that isomorphic   HQFTs yield the same invariants of  closed 3-dimensional $G$-manifolds.

\section{Colored $G$-graphs}\label{sec-coloredgraphs}

In this  section, $M$ is a compact oriented 3-manifold   with pointed boundary and~$\bb$ is a  $G$-crossed   category
 over $\kk$  (see Section~\ref{sect-crossed-graded-cat}).  We discuss   $G$-graphs   in~$M$ and   their  $\bb$-colorings.
Our    definitions and results   are parallel to those of \cite{TVi4} where we studied $G$-graphs in~$\RR^3$.

\subsection{Pointed   ribbon graphs}\label{Ribbon graphs}
We   recall the   notion   of a ribbon graph   following \cite{Tu0}, \cite{TVi5}.  We   restrict ourselves to  ribbon graphs formed from arcs and coupons and having no circle components. An \emph{arc}  is an oriented segment, i.e., an oriented 1-manifold homeomorphic to  $[0,1]$. A \emph{coupon}  is a rectangle   with a distinguished  side called the \emph{bottom base}; the opposite side being   the \emph{top base}.
A  \emph{ribbon  graph}~$\Omega$ in~$M$  is a union of a finite number of arcs and coupons
embedded in~$M$  and called the \emph{strata}  of~$\Omega$. The strata   must satisfy the following  two  conditions:
\begin{enumerate}
\labeli
\item  they are disjoint except that some endpoints of the arcs may lie on the bases of the coupons;
 all  other endpoints of the arcs    lie in $\partial M  \setminus (\partial M)_\bullet $ and form  $\Omega \cap \partial M  $;
\item $\Omega$ carries a continuous  field of tangent directions in~$M$ (the framing) which is transversal to all strata and
tangent to $\partial M$  on  $\Omega\cap \partial M$.
\end{enumerate}
We  provide all coupons of a ribbon graph $\Omega\subset M$ with the orientation which together with the framing  of~$\Omega$
determine the   orientation of~$M$ opposite to the given one.

A ribbon graph  $\Omega$ in~$ M$  is \emph{pointed} if  every closed connected  component~$C$ of~$M$
such that    $C \cap \Omega \neq \emptyset$ is equipped with a base point  lying in $C \setminus \Omega$.
The set of these base points is denoted by $\Omega_\bullet$.  Clearly, the sets $(\partial M)_\bullet$ and
$\Omega_\bullet$ are disjoint, lie in $ M \setminus \Omega$, and their union  meets every
component of $M$ encountering $\Omega$ in at least one point. For example, the empty ribbon graph~$\emptyset$ in $M$ is pointed with $\emptyset_\bullet=\emptyset$.

\subsection{Tracks and detours}\label{Tracks}
Let $\Omega \subset M$ be a pointed ribbon graph.   Slightly pushing~$\Omega  $ along the  framing, we obtain a parallel copy $\widetilde \Omega \subset M\setminus \Omega$ of $\Omega$ which is also a ribbon  graph,  see the following picture (in this and the next pictures the orientation of~$M$ is  right-handed):
\begin{center}
\psfrag{1}[Bc][Bc]{\scalebox{.7}{$1$}}
\psfrag{2}[Bc][Bc]{\scalebox{.7}{$2$}}
\psfrag{3}[Bc][Bc]{\scalebox{.7}{$3$}}
\psfrag{M}[Br][Br]{\scalebox{.9}{$M$}}
\psfrag{x}[Bl][Bl]{\scalebox{.9}{$x_1$}}
\psfrag{y}[Bl][Bl]{\scalebox{.9}{$x_2$}}
\psfrag{z}[Bl][Bl]{\scalebox{.9}{$x_3$}}
\psfrag{O}[Bc][Bc]{\scalebox{.9}{$\Omega$}}
\psfrag{R}[Bc][Bc]{\scalebox{.9}{$\widetilde \Omega$}}
\rsdraw{.45}{.9}{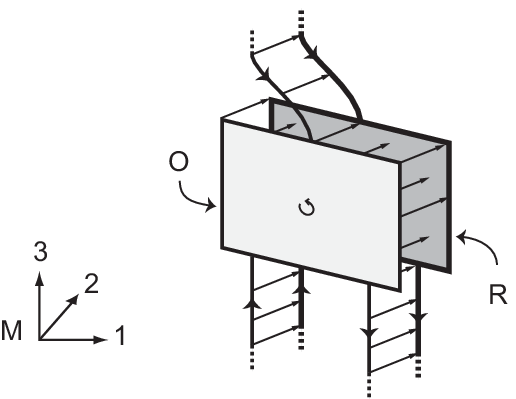}\; .
\end{center}
Each stratum~${e}$   of~$\Omega$ yields a stratum~$\widetilde e$ of $ \widetilde \Omega$ in the obvious way. A \emph{track}  of~$e$ is a  homotopy class   of paths  in $M\setminus \Omega$ leading from  a point of  $(\partial M)_\bullet \cup \Omega_\bullet$  to~$\tilde e$.
For a  track~$\gamma$ of~${e}$, we let $\mu_\gamma\in \pi_1(M\setminus \Omega, \gamma(0))$ be  the homotopy class of the loop $\gamma l_{e} \gamma^{-1}$,  where $l_{e} $ is a small loop in $M\setminus \Omega$ based at the endpoint $\gamma(1)$ of~$\gamma$ and encircling~$e$ as  in the following picture (in particular, if~$e$ is an arc then its linking number with~$l_e$ is -1):
$$
\psfrag{e}[Bc][Bc]{\scalebox{.9}{$e$}}
\psfrag{g}[Bc][Bc]{\scalebox{.9}{$\gamma$}}
\psfrag{m}[Bc][Bc]{\scalebox{.9}{$\ell_e$}}
\rsdraw{.45}{.9}{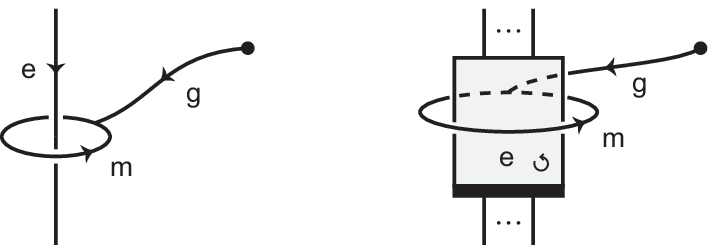} \; .
$$
Here and in the sequel,  the   bottom bases of coupons are drawn boldface.

A \emph{detour} in $M \setminus \Omega$ is a  homotopy class   of paths in $M \setminus \Omega$ with endpoints in   $(\partial M)_\bullet \cup \Omega_\bullet$. In particular, the homotopy class of a constant path in a point of $(\partial M)_\bullet \cup \Omega_\bullet$ is a detour called the \emph{constant detour}. If a detour~$\beta$   in~$M \setminus \Omega$ is composable with a track~$\gamma$    of a stratum of~$\Omega$ (i.e., if~$\beta$ ends in the starting  point   $\gamma(0)$ of~$\gamma$), then $\beta\gamma$  is a  track  of the same stratum and $\mu_{\beta \gamma}=\beta \mu_\gamma \beta^{-1}$.
It follows from our definitions  that every detour in $M \setminus \Omega$ either has both endpoints in $(\partial M)_\bullet $ or is a homotopy class of   loops based  at a point of $\Omega_\bullet$.

\subsection{$G$-graphs}\label{G-graphs}
A \emph{$G$-graph} is a triple $(M, \Omega, g)$  where~$\Omega$   is a pointed ribbon graph in~$M$    and~$g$ is a homotopy class of maps
\begin{equation}\label{chara}
(M \setminus \Omega, (\partial M)_\bullet \cup \Omega_\bullet) \to (\X,\x).
\end{equation}
Clearly,~$g$ carries any detour  in $M \setminus \Omega$   into a homotopy  class  of loops  in $\X $ based at~$\x$. The corresponding element of $G=\pi_1(\X ,\x)$ will be   denoted  by the same letter as the detour itself. Similarly, for any track~$\gamma$    of a stratum of~$\Omega$, the homotopy  class  of loops $\mu_\gamma$ (see Section~\ref{Tracks}) is carried by $g$ into an element of $G$ denoted  again by~$\mu_\gamma$.

For brevity, we will often drop the symbol~$g$ from the notation  for a $G$-graph  $(M, \Omega, g)$ and denote this $G$-graph by $(M,\Omega)$.

\subsection{Precolorings}
A $\bb$-\emph{precoloring} $u$ of  a $G$-graph   $(M,\Omega)$  comprises  two functions. The first function assigns
to every  track~$\gamma $ of an arc of~$\Omega$ an object $u_\gamma\in {{\bb}}_{ \mu_{\gamma} }$  called the \emph{color} of $\gamma$ (see Section~\ref{Tracks} for the definition of $\mu_{\gamma}$).  The second function  assigns  to every
pair $(\beta,\gamma)$, where $\beta$ is a detour in $M \setminus \Omega$ and~$\gamma$ is a track of an arc of~$\Omega$ which is  composable with~$\beta$,   an isomorphism
$$
u_{\beta, \gamma}\colon  u_{\beta \gamma} \to \varphi_{ \beta^{-1} } (u_\gamma).
$$
These functions  must satisfy the following  two  conditions (where $\varphi$ denotes the crossing of $\bb$):
\begin{enumerate}
\labeli
\item If  $\beta $ is  the constant detour  at  the starting  point of a track $\gamma$ of an arc of~$\Omega$, then
$$
u_{\beta,\gamma}=(\varphi_0)_{u_\gamma}\colon u_{ \gamma} \to \varphi_{1}
(u_\gamma).
$$
\item For any  composable detours $\beta, \delta$ in $M \setminus \Omega$   and  any track~$\gamma $ of   an arc of~$\Omega$ which is  composable with~$\delta$,   the following diagram   commutes:
\begin{equation*}
\xymatrix@R=1cm @C=3cm {
u_{ \beta \delta \gamma} \ar[r]^-{u_{  \beta \delta, \gamma}}\ar[d]_{u_{\beta, \delta \gamma}} & \varphi_{\delta^{-1}\beta^{-1}} (u_{  \gamma}) \\
\varphi_{\beta^{-1}} (u_{\delta \gamma}) \ar[r]^-{\varphi_{\beta^{-1}}  (u_{\delta, \gamma}) } & \varphi_{\beta^{-1}} \varphi_{\delta^{-1} } (u_{  \gamma})   \ar[u]_{\varphi_2(\beta^{-1} , \delta^{-1})_{u_\gamma}}.
}
\end{equation*}
\end{enumerate}

A \emph{$\bb$-precolored $G$-graph} is a $G$-graph  endowed with a $\bb$-precoloring. Disjoint unions of $\bb$-precolored $G$-graphs   are $\bb$-precolored $G$-graphs in the obvious way.

\subsection{Colorings}\label{sect-colorings-G-graphs}
A \emph{coupon-coloring} of a $\bb$-precolored $G$-graph $ ((M,\Omega),u)$ is a   function which assigns to each  track~$\gamma$ of  a  coupon~$c$   of~$\Omega$  a     morphism $v_\gamma$ in the category $\bb_{\mu_\gamma}\subset \bb$.
To state our requirements on $v_\gamma$, we need more terminology.  The   \emph{inputs}
(respectively, \emph{outputs}) of~$c$ are the endpoints of the arcs of~$\Omega$ lying on the bottom  (respectively, top)
base of~$c$. The   direction of  the bottom base induced by the orientation of~$c$ determines an order in the set of the inputs.
Let $m \geq 0$   be the number of inputs  of~$c$  and let ${e}_k$  be the arc  of~$\Omega$ incident to the $k$-th input for $k=1, \dots, m$. Set   $\varepsilon_k=+$    if ${e}_k$ is directed out of~$c$ at the $k$-th input and  $\varepsilon_k=-$ otherwise. Similarly,  the direction of the top base of~$c$ induced by the opposite orientation of~$c$ determines an order in the set of the outputs. Let $n\geq 0$   be the number of outputs  of~$c$ and let ${e}^l$ be the arc  of $\Omega$ incident to the $l$-th
output for $l=1, \dots, n$. Set $\varepsilon^l =+$     if ${e}^l$ is directed into~$c$ at the $l$-th output and $\varepsilon^l =-$   otherwise.  Recall  the parallel copy $\widetilde  c \subset \widetilde \Omega$ of~$c$  (see Section~\ref{Tracks}). Let $\gamma$ be a track of~$c$. Composing $\gamma$ with a   path in~$\widetilde  c$ leading   to the $k$-th input, we obtain a track $\rho_k$  of $ {e}_k$.   Composing~$\gamma$ with a   path in~$\widetilde  c$ leading  to the $l$-th output, we obtain  a track $\rho^l$  of $ {e}^l$.   Clearly,
$$
\mu_\gamma=\mu_{\rho_1}^{\varepsilon_1} \cdots  \mu_{\rho_m}^{\varepsilon_m} =
\mu_{\rho^1}^{\varepsilon^1} \cdots  \mu_{\rho^n}^{\varepsilon^n}  \in \pi_1(M \setminus \Omega,\gamma(0))
$$
where  $ \gamma(0)  \in (\partial M)_\bullet \cup \Omega_\bullet$   is the starting point of~$\gamma$.  Set
$$
u_\gamma=\bigotimes_{k=1}^m
u_{\rho_k}^{\varepsilon_k} \quad \text{and} \quad u^\gamma=\bigotimes_{l=1}^n u_{\rho^l}^{\varepsilon^l},
$$
where $X^+=X$ and $X^-=X^*$ for any object $X \in \bb$.
We require that:
\begin{enumerate}
\labeli
\item    $v_\gamma \in \Hom_{{\bb}}(u_\gamma,  u^\gamma)$ for any track $\gamma$ of any coupon  of~$\Omega$;
\item for any  detour~$\beta$ in $M \setminus \Omega$ composable with a track~$\gamma $  of a  coupon  of~$\Omega$, the following diagram  (using the notation above for the inputs and outputs of the coupon) commutes:
\begin{equation}
\begin{split}\label{weakiso---}
\xymatrix@R=1cm @C=1.7cm {
u_{\beta \gamma}=\bigotimes\limits_{k=1}^m
u_{\beta \rho_k}^{\varepsilon_k}  \ar[r]^-{\bigotimes\limits_{k=1}^m
u_{\beta, \rho_k}^{\varepsilon_k}}\ar[d]_{v_{\beta \gamma}  } &
\bigotimes\limits_{k=1}^m \varphi_{ \beta^{-1} }  (u_{ \rho_k}^{\varepsilon_k})    \ar[r]^-{(\varphi_{ \beta^{-1} })_m}
 &\varphi_{\beta^{-1}} (    u_{ \gamma } )  \ar[d]^-{  \varphi_{\beta^{-1}} (v_\gamma) }\\
u^{\beta \gamma}= \bigotimes\limits_{l=1}^n  u_{\beta\rho^l}^{\varepsilon^l}
 \ar[r]^{\bigotimes\limits_{l=1}^n  u_{\beta, \rho^l}^{\varepsilon^l}}
 &\bigotimes\limits_{l=1}^n  \varphi_{\beta^{-1}} (u_{\rho^l}^{\varepsilon^l})  \ar[r]^-{(\varphi_{\beta^{-1}})_n}
& \varphi_{\beta^{-1}}(     u^{\gamma}  ).}
\end{split}
\end{equation}
 Here  $u_{\beta, \gamma}^+= u_{\beta, \gamma} $ while  $u_{\beta,
\gamma}^- \colon  u_{\beta \gamma}^* \to \varphi_{\beta^{-1}}
(u_\gamma^*)$  and $(\varphi_{\beta^{-1}})_m$ are defined in Section~\ref{sect-crossed-graded-cat}.
\end{enumerate}
For   $m=n=1$ and $\varepsilon_1=\varepsilon^1=+$,
the diagram \eqref{weakiso---} simplifies to
 \begin{equation*}\begin{split}\label{weakiso---next}\xymatrix@R=1cm @C=1.7cm {
 u_{\beta \rho_1}  \ar[r]^-{u_{\beta, \rho_1}} \ar[d]_{v_{ \beta \gamma}  } &
  \varphi_{\beta^{-1}} (    u_{ \rho_1} )  \ar[d]^-{  \varphi_{\beta^{-1}} (v_\gamma) }\\
 u_{\beta \rho^1}
 \ar[r]^{u_{\beta, \rho^1} \,\,\,\,\,\,\,\,\,\,\,\,\,}
 &      \varphi_{\beta^{-1}}(    u_{  \rho^1}  ).
}\end{split}\end{equation*}

A \emph{$\bb$-coloring} of a $G$-graph $(M,\Omega)$ is a pair $(u,v)$ where~$u$ is a $\bb$-precoloring  of~$(M,\Omega)$ and~$v$ is  a coupon-coloring of $((M,\Omega),u)$.
A \emph{$\bb$-colored $G$-graph} is a $G$-graph endowed with a $\bb$-coloring. Disjoint unions of $\bb$-colored $G$-graphs   are $\bb$-colored $G$-graphs in the obvious way.

\subsection{Stabilization and conjugation}\label{sect-stabilization}
We define two  operations on a $\bb$-colored   $G$-graphs  called    \emph{stabilization} and \emph{conjugation}.   To stabilize  a $\bb$-colored   $G$-graph $(M, \Omega,g)$, we insert a small coupon~$c$ in the middle of an arc of $\Omega$,   keeping the rest of the data. We take as the bottom base of~$c$ one of its sides incident to the arc. We insert $c$ so that it is transverse to the framing $\nu$ at the middle point of the arc, and provide $c$ with the constant framing equal to $\nu$:
$$
\psfrag{B}[Bl][Bl]{\scalebox{.9}{$\Omega$}}
\psfrag{X}[Bl][Bl]{\scalebox{.9}{$\Omega'$}}
\psfrag{n}[Bl][Bl]{\scalebox{.9}{$\nu$}}
\psfrag{c}[Bl][Bl]{\scalebox{.9}{$c$}}
\rsdraw{.45}{1}{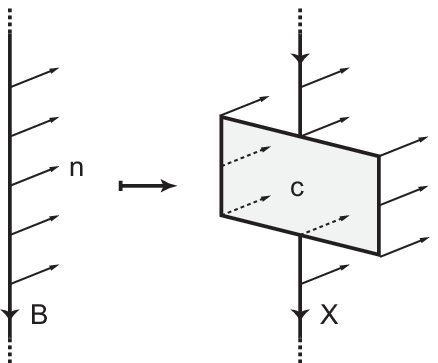}
$$
This gives a   $G$-graph $(M, \Omega',g')$, where  the class of maps~$g'$ is the   restriction of~$g$ via the inclusion $M\setminus \Omega' \subset M\setminus \Omega$. The given $\bb$-precoloring~$u$ of~$\Omega$ restricts to a $\bb$-precoloring~$u'$ of~$\Omega'$  via the same  inclusion.
The given coupon-coloring  of~$\Omega$ similarly restricts to a  coupon-coloring  of all coupons  of~$\Omega'$ except~$c$.
We color any track~$\gamma'$ of~$c$ with  the identity morphism of the object $u'_{\gamma'}=(u')^{\gamma'}$.
Note that this object is computed by $u_\gamma^\varepsilon$, where $\gamma$  is the track of the original arc of~$\Omega$ determined by $\gamma'$ and $\varepsilon=+$ if the arc incident to the  bottom base of~$c$ is directed out of~$c$ while $\varepsilon=-$ otherwise.

The conjugation by any $\kappa\in G$ of a $G$-graph $(M, \Omega,g)$ produces  a $G$-graph $(M, \Omega,g)^\kappa=(M, \Omega,g^\kappa)$, where $g^\kappa$ is the composition of $g$ with the homotopy class of maps
$(\X, x)\to (\X, x)$ representing the endomorphism of  $G=\pi_1(\X,\x)$ carrying each $\alpha\in G$ to $\kappa^{-1} \alpha \kappa\in G$. This transformation lifts to $\bb$-colored $G$-graphs as follows.
The  precoloring~$u$ of $(M, \Omega,g)$ induces a precoloring $u^\kappa$ of $(M, \Omega,g)^\kappa$ which carries every track~$\gamma$ of an arc of~$\Omega$ to
$(u^\kappa)_\gamma=\varphi_\kappa (u_\gamma)$
and carries every pair~$(\beta, \gamma)$, where $\beta$ is a detour in $M\backslash \Omega$ and $\gamma$ is a track of an arc of~$\Omega$ composable with~$\beta$,  to  the isomorphism $(u^\kappa)_{\beta , \gamma}$ defined as the following composition:
$$
\xymatrix@R=1cm @C=3.3cm {
(u^\kappa)_{\beta \gamma}=\varphi_\kappa (u_{ \beta \gamma}) \ar@{.>}[d]_-{(u^\kappa)_{\beta , \gamma}} \ar[r]^-{\varphi_\kappa (u_{\beta, \gamma})} & \varphi_\kappa \varphi_{\beta^{-1}} (u_{  \gamma})  \ar[d]^-{  \varphi_2( \kappa, \beta^{-1}) }\\
\varphi_{\kappa^{-1} \beta^{-1} \kappa}  ((u^\kappa)_\gamma)= \varphi_{\kappa^{-1} \beta^{-1} \kappa}   \varphi_\kappa (
u_\gamma)
& \ar[l]_-{\varphi_2( \kappa^{-1} \beta^{-1} \kappa, \kappa )^{-1}} \varphi_{\beta^{-1} \kappa} (u_{  \gamma}).}
$$
Here, the element of $G=\pi_1(\X,\x)$ represented by the image under $g$ of the detour~$\beta$ is denoted (as above) by the same letter $\beta$, so that the element of $G$ represented by the image under $g^\kappa$ of the detour $\beta$ is indeed computed by $\kappa^{-1} \beta \kappa$.
The coupon-coloring $v^\kappa$ of $((M, \Omega,g)^\kappa, u^\kappa)$ is similarly induced by the coupon-coloring~$v$ of $((M, \Omega,g), u)$ and  the isomorphisms  $(\varphi_\kappa)_n$ and $\varphi^1_\kappa$ associated with~$\varphi$ (see Section~\ref{sect-crossed-graded-cat}). More explicitly, if $\gamma$ is a track of  a  coupon of~$\Omega$ and using the notation of Section~\ref{sect-colorings-G-graphs}, then the morphism $(v^\kappa)_\gamma$ is defined as the following composition:
$$
\xymatrix@R=1cm @C=2.3cm {
(u^\kappa)_\gamma =\bigotimes\limits_{k=1}^m
\varphi_\kappa(u_{\rho_k})^{\varepsilon_k}  \ar@{.>}[d]_-{(v^\kappa)_\gamma} \ar[r]^-{\bigotimes\limits_{k=1}^m
\psi(u_{\rho_k},\varepsilon_k)^{-1}} &
\bigotimes\limits_{k=1}^m \varphi_{ \kappa }  (u_{ \rho_k}^{\varepsilon_k})    \ar[r]^-{(\varphi_{ \kappa })_m}
 &\varphi_{\kappa} (    u_{ \gamma } )  \ar[d]^-{  \varphi_{\kappa} (v_\gamma) }\\
(u^\kappa)^\gamma = \bigotimes\limits_{l=1}^n  \varphi_\kappa (u_{\rho^l})^{\varepsilon^l}
 &  \bigotimes\limits_{l=1}^n  \varphi_{\kappa} (u_{\rho^l}^{\varepsilon^l})
 \ar[l]_-{\bigotimes\limits_{l=1}^n   \psi{(u_{\rho^l},\varepsilon^l)}}
& \ar[l]_-{(\varphi_{\kappa})_n^{-1}} \varphi_{\kappa}(     u^{\gamma}  ).}
$$
Here, for an object $X\in\bb$, we set $\psi(X,+)=\id_{\varphi_\kappa(X)}$ and $\psi(X,-)=\varphi_\kappa^1(X)$.

Two key properties of this transformation will be stated in Section~\ref{sect-Construction-of-precolorings2}.

\section{Isomorphisms and constructions of colorings}\label{sec-coloredgraphs2}

In this  section, $\bb$ is a  $G$-crossed   category  over $\kk$.

\subsection{Isomorphisms of  $G$-graphs}\label{G-graphs2}
An  \emph{isomorphism} between  $G$-graphs $(M, \Omega, g)$ and $(M', \Omega', g')$    is an orientation-preserving diffeomorphism $f\co M  \to M'$ which carries $(\partial M)_\bullet$ onto $(\partial M')_\bullet$, $\Omega_\bullet$ onto~$\Omega'_\bullet$, and~$\Omega$ onto $\Omega'$ (preserving the strata,  the orientation, and the framing)
and satisfies
$$
g = g'f \co (M \setminus \Omega, (\partial M)_\bullet \cup \Omega_\bullet) \to (\X,\x).
$$
Composing~$f$ with a track~$\gamma$ of a stratum  of~$\Omega$, we obtain a track $f \gamma$ of the corresponding stratum of~$\Omega'$. Similarly, composing~$f$ with a detour in $M \setminus \Omega$, we obtain a detour $f \beta$ in $M' \setminus \Omega'$.

An \emph{isomorphism} between  $\bb$-precolored $G$-graphs  $((M,\Omega)  ,    u   )$ and $((M',\Omega')  ,    u'   )$ is a pair $(f,w)$ where $f\co M \to M'$  is an isomorphism  between the $G$-graphs $(M,\Omega) , (M',\Omega' )$  and~$w$ is a function which assigns to every track $\gamma$ of an arc of~$\Omega$    an isomorphism $w_\gamma\co u_\gamma \to u'_{f\gamma}$ in $\bb$ so that for any such~$\gamma$ and any detour~$\beta $ in $M\setminus \Omega$ composable with~$\gamma$,   the following diagram   commutes:
\begin{equation*}
\begin{split}\label{condisom4}\xymatrix@R=1cm @C=3cm {
u_{ \beta  \gamma} \ar[r]^-{u_{  \beta  , \gamma}}\ar[d]_{ w_{\beta\gamma} } & \varphi_{  \beta^{-1} } (u_{  \gamma}) \ar[d]^-{\varphi_{\beta^{-1}}(w_{   \gamma})}\\
u'_{  f(\beta   \gamma)}=u'_{ (f\beta)(f  \gamma)} \ar[r]^-{u'_{  f\beta  , f\gamma}} & \varphi_{\beta^{-1}}  (u'_{ f \gamma}).
}
\end{split}
\end{equation*}

An \emph{isomorphism} between   $\bb$-colored $G$-graphs $(  (M,\Omega) , u,v  )$ and $((M',\Omega') ,   u',v'  )$   is   an isomorphism $(f,w)$ between the underlying $\bb$-precolored $G$-graphs $((M,\Omega)  , u   )$ and $((M',\Omega'),    u'   )$ such that for each  track  $\gamma $ of a coupon of~$\Omega$, the   diagram
\begin{equation*}
\begin{split}\label{condisom5}\xymatrix@R=1cm @C=3cm {
u_{\gamma}  \ar[r]^-{v_{   \gamma}}\ar[d]_{\bigotimes\limits_{i=1}^m w_{\gamma_i}^{\varepsilon_i}} & u^{\gamma} \ar[d]^-{\bigotimes\limits_{j=1}^n w_{\gamma^j}^{\varepsilon^j}}\\
u'_{f\gamma} \ar[r]^-{v'_{    f\gamma}} & (u')^{f\gamma}
} \end{split}
\end{equation*}
commutes. (Here, we use the same notation as in Section~\ref{sect-colorings-G-graphs}.)  Since the vertical arrows are isomorphisms, the commutativity of this diagram implies that~$v'$ is uniquely determined by~$v$ and the isomorphism $(f,w)$.

\subsection{Strong isomorphisms}\label{sect-Construction-of-precolorings2}\label{sect-contruct-colorings2}    A \emph{strong isomorphism} between  two $\bb$-precolorings $u, u'$ of a same $G$-graph $(M,\Omega)$ is
a  function  $w$ which assigns to every track $\gamma$ of an arc of~$\Omega$    an isomorphism $w_\gamma \co u_\gamma \to u'_{\gamma}$ in~$\bb$ such that the pair $(\id_M\co M\to M,w)$  is  an isomorphism  between   the $\bb$-precolored $G$-graphs $(  (M,\Omega) , u  )$ and $((M,\Omega) ,   u'  )$.

A \emph{strong isomorphism} between  two  $\bb$-colorings $(u,v)$ and $(u',v')$ of a $G$-graph $(M,\Omega)$ is a strong isomorphism~$w$ between the  $\bb$-precolorings $u, u'$  of $(M,\Omega)$   such that   the pair $(\id_M\co M \to M,w)$  is  an isomorphism  between   the $\bb$-colored $G$-graphs $(  (M,\Omega) , u,v  )$ and $((M,\Omega) ,   u' ,v' )$.

The notion of a  strong isomorphism allows us  to formulate two  key properties of the conjugation of a $\bb$-colored $G$-graph $((M, \Omega),u,v)$, see Section~\ref{sect-stabilization}. Namely, the $\bb$-colored $G$-graph $((M, \Omega),u,v)^1$  obtained via conjugation by $1 \in G$ is strongly isomorphic to the original $\bb$-colored $G$-graph $((M, \Omega),u,v)$.
Also, for any $\kappa, \kappa' \in G$, the $\bb$-colored $G$-graphs
$$
((M, \Omega), u,v)^{\kappa\kappa'} \quad \text{and} \quad  (((M, \Omega), u,v)^{\kappa})^{\kappa'}
$$
are strongly isomorphic. We leave the proofs of these claims to the reader.

\subsection{Systems of tracks}\label{sect-contruct-colorings}
A \emph{1-system of tracks} for a $G$-graph $(M,\Omega)$
is a family $\{\gamma_e\}_{e}$  where~$e$ runs over the arcs of~$\Omega$ and $\gamma_e$ is a track of~$e$.
To exhibit  such a system we just pick one track for every arc of~$\Omega$
(this is  possible since every  component of~$M$ meeting~$\Omega$ also meets $(\partial M)_\bullet \cup \Omega_\bullet$).
An arbitrary  family of objects  $\{u_e\in \bb_{\mu_{\gamma_e}}\}_{e}$ determines   a $\bb$-precoloring $u$ of $(M,\Omega)$ as follows.
Each track $\gamma$ of an arc $e$ of~$\Omega$  expands uniquely as $\gamma=\alpha \gamma_e$ where~$\alpha$ is a detour   in $M \setminus \Omega$.   Set
$$
u_\gamma =\varphi_{\alpha^{-1} } (u_e)  \in {{\bb}}_{\alpha \mu_{\gamma_e} \alpha^{-1}}={{\bb}}_{\mu_\gamma}.
$$
Note that $u_{\gamma_e}=\varphi_{1 } (u_{ e})$ is canonically isomorphic to $u_e$.
For any detour~$\beta$ in $M \setminus \Omega$  such that~$\gamma $ is  composable with~$\beta$, set
$$
u_{\beta, \gamma}=(\varphi_2(\beta^{-1}, \alpha^{-1})_{u_e})^{-1}  \colon
u_{\beta \gamma}=\varphi_{\alpha^{-1} \beta^{-1}} (u_e) \to \varphi_{\beta^{-1}} \varphi_{\alpha^{-1} } (u_{e})= \varphi_{\beta^{-1}} (u_\gamma).
$$
\begin{lem}\label{lem-precolorings-G-graphs-in}
The  objects and  isomorphisms defined above form a $\bb$-precoloring~$u$ of~$(M,\Omega)$ which only depends (up to strong isomorphism) on the isomorphism classes of the   objects  $\{u_e\}_{e}$.
\end{lem}
\begin{proof}
That $u$ is a $\bb$-precoloring is a direct consequence of the following axioms of the crossing $\varphi$ (see  \cite{TVi4}): $\varphi_2(1,a)_X=(\varphi_0)_{\varphi_a(X)} $ and
$$
\varphi_2(ba,c)_X \, \varphi_2(a,b)_{\varphi_c(X)}= \varphi_2(a,cb)_X \, \varphi_a(\varphi_2(b,c)_X)
$$
for all $a,b,c \in G$ and all object $X$ of~$\bb$.

Next, consider a family of isomorphisms $\{\phi_e \co u_e \to u'_e\}_{e}$ in~$\bb$ and denote by $u'$ the $\bb$-precoloring of~$(M,\Omega)$ derived  as above from the family   $\{u'_e\}_{e}$. For a track $\gamma$ of an arc $e$ of~$\Omega$,  expand $\gamma=\alpha \gamma_e$ as above
and set
$$
w_\gamma= \varphi_{\alpha^{-1} } (\phi_e) \co u_\gamma =\varphi_{\alpha^{-1} } (u_e) \to \varphi_{\alpha^{-1} } (u'_e)=u'_\gamma.
$$
It follows directly from the naturality of $\varphi_2$ that the function $\gamma\mapsto w_\gamma$ is a strong isomorphism   between the $\bb$-precolorings $u, u'$ of $(M,\Omega)$.
\end{proof}

The following lemma shows that this construction  yields all $\bb$-precolorings of $(M,\Omega)$, at least up to strong isomorphism.

\begin{lem}\label{lem-precolorings-G-graphs} Any $\bb$-precoloring $U$  of $(M,\Omega) $ is strongly isomorphic to the $\bb$-precoloring   $u$ of $(M,\Omega) $ determined by the family of objects $\{u_e=U_{\gamma_e}\}_{e}$.
\end{lem}
\begin{proof}
For a  track $\gamma$ of  an arc $e$ of~$\Omega$, expand $\gamma=\alpha \gamma_e$ as above  and set
$$
w_\gamma=U_{\alpha,\gamma_e}\co U_\gamma=U_{\alpha \gamma_e} \to \varphi_{\alpha^{-1}} (U_{\gamma_e})=u_\gamma.
$$
It is easy to check that the function $\gamma\mapsto w_\gamma$ is  a strong isomorphism  between the $\bb$-precolorings $U$ and $u$ of $(M,\Omega)$.
\end{proof}

A \emph{2-system of tracks} for a $G$-graph $(M,\Omega)$  is a family $\{\gamma_c\}_{c}$ where~$c$ runs over all coupons of~$\Omega$   and $\gamma_c$ is a track of~$c$. To exhibit  such a  system we just  fix   one track for every~$c$.
Given a  $\bb$-precoloring $u$  of $(M,\Omega)$, pick a 2-system of tracks
$\{\gamma_c\}_{c}$ for $(M,\Omega)$   and   an arbitrary family of morphisms
$\{v_c\colon u_{\gamma_c} \to u^{\gamma_c}\}_{c}$  in~$\bb$. (Here, the objects $u_{\gamma_c}$ and $u^{\gamma_c}$ are  as in Section~\ref{sect-colorings-G-graphs}). Each track of a coupon~$c$ of~$\Omega$ expands uniquely as $ \beta \gamma_c$ where~$\beta$ is a detour   in $M \setminus \Omega$.   Formula~\eqref{weakiso---} with $\gamma=\gamma_c$ and $v_\gamma=v_c$ defines a morphism
$$
v_{\beta \gamma_c} \colon u_{\beta \gamma_c} \to u^{\beta \gamma_c}
$$ in~$\bb$.
This yields a coupon-coloring $v$ of $((M,\Omega),u)$ and  a $\bb$-coloring $(u,v)$ of $(M,\Omega)$. Lemma~\ref{lem-precolorings-G-graphs} implies that all colorings of~$\Omega$ may be obtained in this way,  at least up to strong isomorphism.

\section{Colored $G$-surfaces}\label{sec-HQFTee}

In this section we introduce   markings and colorings of   $G$-surfaces.
We fix until the end of the section a  $G$-crossed category $\bb$ over $\kk$.

\subsection{Marked $G$-surfaces}\label{sect-marked-surfaces}
A   point of a surface   is \emph{marked} if it is endowed with a tangent direction   and a sign $\pm 1$. A   \emph{marked $G$-surface} is a triple  consisting of a pointed closed  oriented  surface~$\Sigma$, a finite set of marked points  $A \subset \Sigma \setminus \Sigma_\bullet$,  and a homotopy class of  maps
$g \co (\Sigma \setminus A, \Sigma_\bullet)\to (\X,\x)$. For a point  $a\in A$, we let $\varepsilon_a=\pm 1$ be its sign
 and let $\ast_a\in \Sigma \backslash A$ be the base point of the connected component of~$\Sigma$ containing~$a$.
Slightly pushing~$a $ in the given tangent direction we get a point $\tilde a \in \Sigma \setminus  A  $.   A \emph{track} of~$a$ is a homotopy class~$\gamma$ of   paths leading from   $\ast_a$ to $\tilde a$  in $\Sigma\setminus A$.
We let  $\mu_\gamma \in \pi_1(\Sigma \setminus A, \ast_a)$  be the element  represented
the loop $\gamma m_a^{\varepsilon_a} \gamma^{-1}$, where  $m_a$ is a small loop in $\Sigma \setminus A$ based in $\tilde a$ and encircling~$a$   in the direction induced by the orientation of~$\Sigma$ (i.e., $m_a$ is the boundary of disk embedded in $\Sigma \setminus (A\setminus\{a\})$ containing $a$ in its interior and $\tilde a$ in its boundary). Multiplication by loops   defines a left    action of the group $\pi_1(\Sigma \setminus A, \ast_a)$ on the set of tracks of~$a$. Clearly, $\mu_{\beta \gamma}= \beta \mu_\gamma \beta^{-1}$ for  any~$\beta \in \pi_1(\Sigma \setminus A, \ast_a)$. As above,  we  denote by the same letters the  elements of~$\pi_1(\Sigma \setminus A, \ast_a)$ and their images under the group homomorphism $ \pi_1(\Sigma \setminus A, \ast_a)  \to G$ induced by $g$.

An \emph{isomorphism} of  marked $G$-surfaces $(\Sigma, A, g  )$ and $(\Sigma', A', g'  )$ is an orien\-ta\-ti\-on-preserving diffeomorphism   $f\co \Sigma  \to \Sigma'$ which carries $\Sigma_\bullet$ onto $ \Sigma'_\bullet$ and  $A$ onto $A'$  (preserving the tangent directions and the signs of the marked points) and such that  $g = g' f$.

For brevity, we will often denote a marked $G$-surface  $(\Sigma, A, g  )$ by $(\Sigma, A  )$ or even by~$\Sigma$ suppressing~$g$ and~$A$.

\subsection{Colored $G$-surfaces}\label{sect-G-surf}
A \emph{$\bb$-coloring} $u$ of a marked $G$-surface $(\Sigma, A)$ comprises  two functions. The first function   assigns to every
track  $\gamma $ of any  $a\in A$   an object $u_\gamma\in\bb_{\mu_{\gamma}}$ called the \emph{color} of~$\gamma$.   The second function   assigns to  every   track  $\gamma $ of any   $a\in A$   and to every
$\beta\in \pi_1(\Sigma \setminus A ,\ast_a)$   an isomorphism
$$
u_{\beta, \gamma}\colon  u_{\beta \gamma} \to \varphi_{\beta^{-1}} (u_\gamma)
$$
such that:
\begin{enumerate}
\labeli
\item for every   track  $\gamma $,
$$u_{1,\gamma}=(\varphi_0)_{u_\gamma}\colon u_{ \gamma} \to \varphi_{1} (u_\gamma);$$
\item for every   track  $\gamma $ of any   $a\in A$ and for all $ \beta, \delta \in  \pi_1(\Sigma \setminus A ,\ast_a )$,
the following diagram   commutes:
\begin{equation*}
\xymatrix@R=1cm @C=3cm {
u_{ \beta \delta \gamma} \ar[r]^-{u_{  \beta \delta, \gamma}}\ar[d]_{u_{\beta, \delta \gamma}} & \varphi_{\delta^{-1}\beta^{-1}} (u_{  \gamma}) \\
\varphi_{\beta^{-1}} (u_{\delta \gamma}) \ar[r]^-{\varphi_{\beta^{-1}}  (u_{\delta, \gamma}) } & \varphi_{\beta^{-1}} \varphi_{\delta^{-1} } (u_{  \gamma})   \ar[u]_{\varphi_2(\beta^{-1} , \delta^{-1})_{u_\gamma}}.
}
\end{equation*}
\end{enumerate}

  A \emph{$\bb$-colored $G$-surface}  is a marked $G$-surface endowed with a $\bb$-coloring.
Disjoint unions of $\bb$-colored $G$-surfaces   are $\bb$-colored $G$-surfaces   in the obvious way.
Reversing   the orientation of the ambient surface and   the signs of all marked points (while keeping the colors of the tracks),   we transform a $\bb$-colored $G$-surface   $\Sigma $ into  the \emph{opposite}  $\bb$-colored $G$-surface $-\Sigma $.

An  \emph{isomorphism} between
$\bb$-colored $G$-surfaces $(\Sigma, A,  u )$ and $(\Sigma', A',  u' )$ is a pair $(f,F)$ consisting of an isomorphism $f\co\Sigma \to \Sigma'$ of the underlying marked $G$-surfaces   and a  function~$F$ which assigns to every track $\gamma$ of any $a\in A$  an isomorphism $F_\gamma\co u_\gamma \to u'_{f\gamma}$ in~$\bb$ such that, for all  $\beta\in \pi_1(\Sigma \setminus A, \ast_a)$, the following diagram   commutes:
\begin{equation*}\begin{split}\label{condisom2}\xymatrix@R=1cm @C=3cm {
u_{ \beta  \gamma} \ar[r]^-{u_{  \beta  , \gamma}}\ar[d]_{F_{\beta\gamma}} & \varphi_{ \beta^{-1}} (u_{  \gamma}) \ar[d]^-{\varphi_{\beta^{-1}}(F_{   \gamma})}\\
u'_{  f(\beta   \gamma)}=u'_{ (f\beta)(f  \gamma)} \ar[r]^-{u'_{  f\beta  , f\gamma}} & \varphi_{\beta^{-1}}  (u'_{ f \gamma}).
} \end{split}\end{equation*}
Here $f\gamma $ is the track of   $f(a)\in A'$ obtained by composing~$\gamma$ and~$f$.

\subsection{Strong isomorphisms and systems of tracks}\label{{A construction of colorings1}}
Let $(\Sigma, A)$  be a marked $G$-surface. A \emph{strong isomorphism} between two  $\bb$-colorings $u,u'$ of  $(\Sigma, A)$ is a function~$F$ which assigns to every track $\gamma$ of any $a\in A$  an isomorphism $F_\gamma\co u_\gamma \to u'_{ \gamma}$ in~$\bb$ such that the pair $(\id_\Sigma\co\Sigma\to \Sigma, F)$ is an isomorphism between the
$\bb$-colored $G$-surfaces $(\Sigma, A,  u )$ and $(\Sigma, A,  u' )$.

A \emph{0-system of tracks} for $(\Sigma, A)$   is a family $\{\gamma_a\}_{a\in A}$ where $\gamma_a$ is a track of~$a$.
To produce such a system one just picks a track for each   point of~$A$ (this   is always possible since every component of~$\Sigma$ has a base point).

Given  a 0-system of tracks   $\{\gamma_a\}_{a\in A}$  for $(\Sigma, A)$, any    family
$\{u_a\in \bb_{\mu_{\gamma_a}} \}_{a\in A}$     determines   a $\bb$-coloring $u$ of $(\Sigma, A)$ as follows. Every track~$\gamma$ of any $a\in A$ expands uniquely as $\gamma=\alpha\gamma_a$ with $\alpha\in \pi_1(\Sigma \setminus A, \ast_a)$. Set
$$
u_\gamma=  \varphi_{\alpha^{-1} } (u_{ a}) \in {{\bb}}_{\alpha \mu_{\gamma_a} \alpha^{-1}}={{\bb}}_{\mu_\gamma}.
$$
In particular,  $u_{\gamma_a}=\varphi_{1 } (u_{ a})$ is canonically isomorphic to $u_a$.
For $\gamma=\alpha\gamma_a$ as above and any $\beta\in \pi_1(\Sigma \setminus A ,\ast_a)$,  set
$$
u_{\beta, \gamma}=(\varphi_2(\beta^{-1}, \alpha^{-1})_{u_a})^{-1}  \colon
u_{\beta \gamma}=\varphi_{\alpha^{-1} \beta^{-1}} (u_a) \to \varphi_{\beta^{-1}} \varphi_{\alpha^{-1} } (u_{ a})= \varphi_{\beta^{-1}} (u_\gamma).
$$

\begin{lem}\label{lem-coloringsofsurfaces-in}
The  objects and  isomorphisms above    form a $\bb$-coloring of~$(\Sigma, A)$ which only depends (up to strong isomorphism) on the isomorphism  classes of the   objects  $\{u_a\}_{a\in A}$. This construction yields all $\bb$-colorings of $(\Sigma, A)$   up to strong isomorphism: any $\bb$-coloring $U$  of $(\Sigma, A)$ is strongly isomorphic to the $\bb$-coloring  of $(\Sigma, A)$ determined  by the family of objects $\{U_{\gamma_a}\}_{a\in A}$.
\end{lem}

Lemma~\ref{lem-coloringsofsurfaces-in}  is an analogue of Lemmas~\ref{lem-precolorings-G-graphs-in},~\ref{lem-precolorings-G-graphs}  and is  proven similarly.

\subsection{Surfaces versus  graphs}\label{sect-precol-surf-from-graphs}\label{sect-cyl}  Colored  $G$-surfaces and $G$-graphs
are related in two ways:  the boundary of any precolored $G$-graph is  a colored $G$-surface and  the cylinder over any colored $G$-surface  is  a  colored $G$-graph.   Here are more details.
Consider a  $\bb$-precolored $G$-graph $((M, \Omega), u)$.  We endow the pointed oriented  surface $\partial M$ with the set of  marked points $\partial \Omega=\partial M \cap \Omega$ where the distinguished tangent direction  at  any  $  a\in \partial   \Omega$ is induced by the framing of~$\Omega$ at~$a$ and the sign of~$a$  is~$+$ if the adjacent arc   of~$\Omega$  is directed  inside~$M$  and~$-$ otherwise.  The given homotopy class of  maps
$(M \setminus \Omega, (\partial M)_\bullet \cup \Omega_\bullet) \to (\X,\x)$  restricts to  a  homotopy class of maps   $(\partial M \setminus \partial   \Omega, (\partial M)_\bullet) \to (\X,\x)$. This turns $ (\partial M, \partial \Omega)$ into a marked $G$-surface.
We   define its $\bb$-coloring $\partial u$. Each track~$\gamma$ of a point $a\in \partial   \Omega $ in $\partial M$ determines through the inclusion $\partial M \hookrightarrow M$  a track, again denoted by $\gamma$,   of the  arc of $\Omega$ adjacent to~$a$.
Set  $(\partial u)_\gamma =u_{\gamma}\in \bb_{ \mu_\gamma }$ and   $(\partial u)_{\beta, \gamma}=u_{\beta,\gamma}$  for any   $\beta \in \pi_1(\partial M \setminus \partial \Omega, \ast_a)$  where $\ast_a$ is the only point of $(\partial M)_\bullet $ lying in the same   component of $\partial M$ as~$a$. This gives a  $\bb$-colored $G$-surface $(\partial M, \partial \Omega, \partial u)$. It is clear that (strongly) isomorphic $\bb$-precolorings  of    $(M, \Omega)$  yield in this way (strongly) isomorphic $\bb$-colored $G$-surfaces.

The cylinder construction starts from an arbitrary  $\bb$-colored $G$-surface $(\Sigma, A, u)$. Let $I=[0,1]$ be the unit segment  directed from~$0$ to~$1$. Consider the  3-manifold $C=\Sigma \times I$
with the product orientation and pointed  boundary where  $(\partial C)_\bullet=\Sigma_\bullet \times \{0,1\}$.
Then $\Omega=A\times I\subset C$
is a ribbon graph with arcs $\{a\times I\}_{a\in A}$ and no   coupons.  Each arc $a\times I$  carries a framing  which is a lift of the given tangent direction at~$a$. The arc $a\times I$ is directed   towards  $a\times \{0\}$  if the sign of~$a$ is $+$ and   towards  $a\times \{1\}$  otherwise.  Since~$C$ has no closed components,   $\Omega_\bullet=\emptyset$.
Composing the projection $\mathrm{pr} \colon C\setminus \Omega \to \Sigma \setminus A$ with the given homotopy class of maps $(\Sigma \setminus A, \Sigma_\bullet)\to (\X, x)$, we obtain a   homotopy class of maps
$
(C \setminus \Omega, (\partial C)_\bullet )\to (\X,\x)$.
This determines  a $G$-graph $(C, \Omega)$. Every  track~$\gamma$ of the arc $a\times I $ of $ \Omega$  projects to the track $\mathrm{pr}(\gamma)$ of~$a$ in~$\Sigma$, and we set
$
U_\gamma=u_{pr(\gamma)}\in \bb_{\mu_\gamma}
$.
If~$\gamma$   is composable with a detour~$\beta$   in $C\setminus \Omega$, consider   $\mathrm{pr}  (\beta)\in \pi_1(\Sigma \setminus A, \ast_a)$    and   set
$$
U_{\beta, \gamma}=u_{\mathrm{pr}(\beta), \mathrm{pr}( \gamma)}\co U_{\beta \gamma}=u_{\mathrm{pr}(\beta \gamma)} \to \varphi_{\beta^{-1}} (u_{\mathrm{pr}(\gamma)}) = \varphi_{\beta^{-1}} (U_\gamma).
$$
This defines a $\bb$-precoloring $U$ of $(C, \Omega)$. Since the ribbon graph $\Omega$ has no coupons,   the tuple  $ C_\Sigma =((C, \Omega), U)$ is  a $\bb$-colored $G$-graph. It is  called
 the \emph{cylinder} over  $(\Sigma,A,u)$.

\section{The category $\Cob^G_\bb$}\label{sect-cobgb}

For each $G$-crossed category $\bb$ over $\kk$, we define a symmetric monoidal category $  \Cob^G_\bb$  whose objects are  $\bb$-colored   $G$-surfaces  and whose morphisms are certain equivalence classes of  $\bb$-colored $G$-graphs.

\subsection{Objects and morphisms}\label{EQFTs versus HQFTs+++}
The objects of  $\Cob^G_\bb$  are   $\bb$-colored   $G$-surfaces (possibly, empty). For $\bb$-colored $G$-surfaces $\Sigma_0$, $\Sigma_1$, a morphism   $\Sigma_0 \to \Sigma_1$ in $\Cob^G_\bb$ is represented by a triple $(M,\Omega,h)$ consisting of a  $\bb$-colored $G$-graph $(M,\Omega)$  and an isomorphism of $\bb$-colored $G$-surfaces
$$
h\co (-\Sigma_0) \sqcup \Sigma_1 \to \partial M=(\partial M,  {\partial \Omega}).
$$
We   call  such a triple  $(M,\Omega,h)$ a \emph{$\bb$-colored   $G$-cobordism between $\Sigma_0$ and $\Sigma_1$}. Note that the image of $\Sigma_0$ (respectively, $\Sigma_1$) under~$h$  is  a union of several connected components of the  $\bb$-colored $G$-surface $\partial M $. We denote this image by  $\partial_-M$ (respectively,  $\partial_+ M$)  so that $\partial M$ is the disjoint union of the $\bb$-colored  $G$-surfaces $\partial_-M$ and $\partial_+ M$. The isomorphism~$h$ restricts to isomorphisms of  $\bb$-colored  $G$-surfaces $h_- \co -\Sigma_0\to \partial_-M$ and  $h_+ \co \Sigma_1\to \partial_+M$.

Two  $\bb$-colored   $G$-cobordisms between $\Sigma_0$ and $\Sigma_1$ represent the same morphism $\Sigma_0\to \Sigma_1$ in $\Cob^G_\bb$ if they are obtained from each other by a finite sequence of the following operations and their inverses:
\begin{enumerate}
\labeli
\item \emph{Isomorphism:} one picks an isomorphism $f\co  (M, \Omega) \to (M', \Omega')$   of $\bb$-colored $G$-graphs  and replaces   $(M,\Omega,h )$  with
    $$
    (M', \Omega', f h\co (-\Sigma_0) \sqcup \Sigma_1 \to (\partial M',  {\partial \Omega'})).
    $$
\item \emph{Stabilization:}
one  keeps $M, h$ and stabilizes $\Omega$ at an arc as in  Section~\ref{sect-stabilization}.
\item \emph{Conjugation:}
one  picks a closed component~$C$ of~$M$ meeting~$\Omega$ and conjugates $(C, \Omega \cap C)$ by an  element of the group~$G$  keeping the rest of $M , \Omega$ and  the isomorphism~$h$.

\end{enumerate}

\subsection{Composition of morphisms}\label{sect-def-compo-cob-G} Let $\Sigma_0, \Sigma_1, \Sigma_2$ be $\bb$-colored $G$-surfaces and let   $\chi_1\co \Sigma_0\to \Sigma_1$ and $\chi_2\co \Sigma_1\to \Sigma_2$   be   morphisms in $\Cob^G_\bb$ represented respectively by   $\bb$-colored   $G$-cobordisms  $(M_1,\Omega_1,h_1)$ and  $(M_2,\Omega_2,h_2)$.  The morphism  $\chi_2 \circ \chi_1$  is represented by the $\bb$-colored  $G$-cobordism     between  $ \Sigma_0$ and $ \Sigma_2$ defined as follows.  First, gluing  $M_1$ and $M_2$  along  the diffeomorphism
\begin{equation} \label{diff}
(h_2)_-   \circ (h_1)_+^{-1}  \colon  \partial_+(M_1) \to \partial_-(M_2)
\end{equation}
we obtain a 3-manifold $M$. The orientations of $M_1$ and $M_2$ determine an orientation of~$M$ so that   both  natural embeddings $j_1\co M_1 \hookrightarrow M$ and $j_2\co M_2 \hookrightarrow M$    are orientation preserving. Clearly,
$\partial M = \partial_- M  \sqcup  \partial_+ M$ where $\partial_- M =j_1(\partial_-(M_1))$ and $\partial_+ M = j_2(\partial_+(M_2))$. We endow   $\partial M$    with the set of base points
$$
(\partial M)_\bullet = j_1 ((\partial_-(M_1))_\bullet ) \cup j_2  ((\partial_+(M_2))_\bullet ).
$$
Thus,   $M$ is a compact  oriented 3-manifold with pointed boundary. Note  that  the  diffeomorphisms
$(h_1)_+\co \Sigma_1 \to \partial_+(M_1)$ and $(h_2)_-\co \Sigma_1 \to \partial_-(M_2)$ induce an embedding $ h \co\Sigma_1 \hookrightarrow M$ such that
$$
h(\Sigma_1)=j_1 (\partial_+(M_1) ) =j_2 (\partial_-(M_2)\bigr)\subset \Int(M).
$$

Since the map \eqref{diff}  is an isomorphism of $G$-surfaces, the endpoints of the arcs of $\Omega_1$, $\Omega_2$ in $\partial_+(M_1)$, $\partial_-(M_2)$ match under the gluing and so do the orientations and the framings of the adjacent arcs  of $\Omega_1$, $\Omega_2$. At each of these endpoints we insert in the union  $ j_1(\Omega_1) \cup j_2(\Omega_2)$ a small coupon with one input and one output as   in Section~\ref{sect-stabilization}.   The bottom   base of this coupon is its side lying in $j_1(M_1)$. The resulting set $\Omega\subset M$ contains  $ j_1(\Omega_1) \cup j_2(\Omega_2)$ and
expands as a  union of a finite number of arcs and coupons.  The coupons of~$\Omega$ are the images of the coupons of  $\Omega_1$, $\Omega_2$ under $j_1, j_2$ and the coupons added above at the points of the set $h((\Sigma_1)_\bullet) $. The arcs of~$\Omega$ are  the images of the arcs of $\Omega_1, \Omega_2$ under $j_1, j_2$ which are slightly shortened near $h((\Sigma_1)_\bullet)$. The framings of  $\Omega_1, \Omega_2$  extend to a framing of~$\Omega$ in the obvious way.
Clearly,~$\Omega$ is a ribbon graph in~$M$.

Each closed connected  component~$C$  of~$M$ either
lies in $j_i(M_i)$ for $i\in \{1,2\}$ or meets the surface $h(\Sigma_1)$ along several components.
For~$C $   meeting both $h(\Sigma_1)$ and~$\Omega$, pick any point $p_C  \in C \cap h((\Sigma_1)_\bullet)$. We endow~$\Omega$ with the set of base points
$$
\Omega_\bullet= j_1((\Omega_1)_\bullet) \cup j_2((\Omega_2)_\bullet)\cup \{p_C\}_C
$$
where $C$ runs over  closed components of~$M$ meeting both $h(\Sigma_1)$ and~$\Omega$.  This turns~$\Omega$ into
a pointed ribbon graph in~$M$.  Note the inclusion
\begin{equation}\label{eq-inclusion}
(\partial M)_\bullet \cup \Omega_\bullet \subset
j_1((\partial M_1)_\bullet \cup (\Omega_1)_\bullet) \cup j_2((\partial M_2)_\bullet \cup (\Omega_2)_\bullet).
\end{equation}

Since the map \eqref{diff}  is an isomorphism of $G$-surfaces, we can pick  representatives in the given homotopy classes of maps  \begin{equation}\label{reprs}
 \{(M_i \setminus \Omega_i, (\partial M_i)_\bullet \cup (\Omega_i)_\bullet) \to (\X, x)\}_{i=1,2}
\end{equation}
which match  under \eqref{diff}. These representatives  determine a map
\begin{equation}\label{eq-map-compos}
\bigl(M \setminus \Omega, \,\, \, j_1((\partial M_1)_\bullet \cup (\Omega_1)_\bullet) \cup j_2((\partial M_2)_\bullet \cup (\Omega_2)_\bullet) \bigr) \to (\X, x).
\end{equation}
In view of the inclusion \eqref{eq-inclusion},  the latter map   determines a homotopy class of maps
$$
(M \setminus \Omega, (\partial M)_\bullet \cup \Omega_\bullet)\to (\X, x).
$$
Since $\X= K(G,1)$, an elementary obstruction theory shows that this homotopy class does not depend on the choice of representatives of the homotopy classes~\eqref{reprs}.  In this way,   $(M,\Omega)$ becomes a  $G$-graph.

Let $(u^i,v^i)$ be the $\bb$-coloring of $(M_i, \Omega_i)$ for $i=1,2$.
 We use the method  of Section~\ref{sect-contruct-colorings} to derive from these colorings
a $\bb$-coloring $(u,v)$ of $(M, \Omega)$. For $i=1,2$, pick  a 1-system of tracks $\{\gamma^i_e\}_{e}$ for  $(M_i, \Omega_i)$  where~$e$ runs over the arcs of~$\Omega_i$.  Consider any track~$\gamma$  in $M \setminus \Omega$ of an arc of~$\Omega$. This arc lies in  $j_i(e)$ for some   $i=1,2$ and some  arc~$e$ of $\Omega_i$.  Then~$\gamma$ expands uniquely as the product $\gamma=\alpha\,  j_i (\gamma^i_e)$ where~$\alpha$ is a homotopy class of paths   in $M \setminus \Omega$ from   $\gamma(0) \in (\partial M)_\bullet \cup \Omega_\bullet$ to   $$ j_i(\gamma^i_e(0)) \in j_i((\partial M_i)_\bullet \cup (\Omega_i)_\bullet). $$
The  map \eqref{eq-map-compos} carries~$\alpha$   into
a homotopy class of loops  in $(\X ,\x)$   representing an  element of $G=\pi_1(\X ,\x)$  also denoted by~$\alpha$. Set
$$
u_\gamma= \varphi_{\alpha^{-1}}(u^i_{\gamma^i_e})\in {{\bb}}_{\alpha \mu_{\gamma^i_e} \alpha^{-1}}=\bb_{\mu_{\gamma}}
$$
For any detour~$\beta$   in $M \setminus \Omega$    composable with~$\gamma$,  set
$$
u_{\beta, \gamma}=\Bigl(\varphi_2(\beta^{-1}, \alpha^{-1})_{u^i_{\gamma^i_e}}\Bigr)^{-1}  \colon
u_{\beta \gamma}=\varphi_{\alpha^{-1} \beta^{-1}} \bigl(u^i_{\gamma^i_e}\bigr) \to \varphi_{\beta^{-1}} \varphi_{\alpha^{-1} } \bigl(u^i_{\gamma^i_e}\bigr)= \varphi_{\beta^{-1}} (u_\gamma).
$$
This defines a $\bb$-precoloring~$u$ of $(M,\Omega)$.
We now  define a coupon-coloring~$v$  of $((M,\Omega), u)$.  For $i=1,2$ pick  a 2-system of tracks $\{\gamma^i_c\}_c$ of $\Omega_i$ where~$c$ runs over the coupons of $\Omega_i$. Also pick
a 0-system of tracks $\{\gamma_a\}_a$ for $\Sigma_1$ where~$a$ runs over the marked points of $\Sigma_1$.  For each track~$\gamma$  of a coupon of $\Omega$, we  define a morphism $v_\gamma \co u_\gamma\to u^\gamma$ in~$\bb$ as follows. Assume first that the coupon in question is $j_i(c)$  where $i=1,2$ and~$c$ is a coupon of $\Omega_i$. Then  $\gamma$ expands uniquely as the product $\gamma=\alpha \, j_i(\gamma^i_c)$ where $\alpha$ is a homotopy class of paths   in $M \setminus \Omega$ from    $\gamma(0)\in (\partial M)_\bullet \cup \Omega_\bullet$ to
$$
j_i(\gamma^i_c(0)) \in j_i((\partial M_i)_\bullet \cup (\Omega_i)_\bullet).
$$
As above, we use the same letter~$\alpha$   to denote the element of~$G$ represented by the image of~$\alpha$ in $\X$.
We  use the notation of Section~\ref{sect-colorings-G-graphs} for the data associated with the coupon~$c$ and its track  $\gamma^i_c$ in $M_i \setminus \Omega_i$.  Consider the  morphism
$$
v^i_{\gamma^i_c} \colon (u^i)_{\gamma^i_c} \to (u^i)^{\gamma^i_c}
\quad\text{where}\quad
(u^i)_{\gamma^i_c}=\bigotimes_{k=1}^m
(u^i_{\rho_k})^{\varepsilon_k}
\quad\text{and}\quad
(u^i)^{\gamma^i_c}=\bigotimes_{\ell=1}^n
(u^i_{\rho^\ell})^{\varepsilon^\ell}.
$$
(Recall  that $X^+=X$ and $X^-=X^*$ for any object $X \in \bb$.)
The composition of the track~$\gamma$ of the coupon $j_i(c) \subset \Omega$
with a path in the parallel coupon $\widetilde{j_i(c)}\subset \widetilde \Omega$
leading to the $k$-th input is  the product track $\alpha \,  j_i(\rho_k)$. Hence
$$
u_\gamma=\bigotimes_{k=1}^m u_{\alpha\,  j_i(\rho_k)}^{\varepsilon_k}=\bigotimes_{k=1}^m  \bigl(\varphi_{\alpha^{-1}}(u^i_{\rho_k})\bigr)^{\varepsilon_k}.
$$
Similarly,
$$
u^\gamma=\bigotimes_{\ell=1}^n u_{\alpha\,  j_i(\rho^\ell)}^{\varepsilon^\ell}=\bigotimes_{\ell=1}^n \bigl(\varphi_{\alpha^{-1}}(u^i_{\rho^\ell})\bigr)^{\varepsilon^\ell}.
$$
Define the morphism $v_\gamma \co u_\gamma\to u^\gamma$ as the following composition:
$$
\xymatrix@R=1cm @C=2.3cm {
u_\gamma \ar@{.>}[d]_-{v_\gamma} \ar[r]^-{\bigotimes\limits_{k=1}^m
\psi(u^i_{\rho_k},\varepsilon_k)^{-1}} &
\bigotimes\limits_{k=1}^m \varphi_{ {\alpha^{-1}} }  \bigl((u^i_{ \rho_k})^{\varepsilon_k}\bigr)    \ar[r]^-{(\varphi_{ {\alpha^{-1}} })_m}
 &\varphi_{ \alpha^{-1} } \bigl((u^i)_{\gamma^i_c}\bigr)  \ar[d]^-{\varphi_{\alpha^{-1}} \bigl(v^i_{\gamma^i_c}\bigr)}\\
u^\gamma
 &  \bigotimes\limits_{l=1}^n  \varphi_{{\alpha^{-1}}} \bigl((u^i_{\rho^l})^{\varepsilon^l}\bigr)
 \ar[l]_-{\bigotimes\limits_{l=1}^n   \psi{(u^i_{\rho^l},\varepsilon^l)}}
& \ar[l]_-{(\varphi_{{\alpha^{-1}}})_n^{-1}} \varphi_{ \alpha^{-1} } \bigl((u^i)^{\gamma^i_c}\bigr)}
$$
where, for an object $X\in\bb$, we set $\psi(X,+)=\id_{\varphi_{\alpha^{-1}}(X)}$ and $\psi(X,-)=\varphi_{\alpha^{-1}}^1(X)$.
Next, consider a  track~$\gamma$  of the coupon $c_a$ of $\Omega$  at   a marked point~$a$  of $\Sigma_1$.
We     expand (uniquely)  $\gamma=\alpha\,  h(\gamma_a)$ where $ h \co\Sigma_1 \hookrightarrow      \Int(M)$ is the embedding  above  and~$\alpha$ is a homotopy class of paths   in $M \setminus \Omega$ from   $\gamma(0) \in (\partial M)_\bullet \cup \Omega_\bullet$ to
$$
h(\gamma_a(0) )\in h((\Sigma_1)_\bullet)=j_1(\partial_+(M_1)_\bullet)=j_2(\partial_-(M_2)_\bullet)\subset h(\Sigma_1).
$$
For $i=1,2$ the point $h_i(a)\in \partial \Omega_i$ is an   endpoint of a unique   arc $e_i$ of~$\Omega_i$. Then  $h_i ( \gamma_a)$    is a track   of  $e_i$ in $M_i \setminus \Omega_i$, and   $h_i (\gamma_a)=\beta_i \, \gamma^i_{e_i}$  for a unique detour $\beta_i$     in $M_i \setminus \Omega_i$. The isomorphisms of $\bb$-colored $G$-surfaces  \eqref{diff}    induces an isomorphism
$$
\Phi \co u^1_{\beta_1 \gamma^1_{e_1}}=u^1_{h_1(\gamma_a)} \to
u^2_{h_2(\gamma_a)} =u^2_{\beta_2 \gamma^2_{e_2}}.
$$
For $i=1,2$, set
$$
\Psi_i= \varphi_2(\alpha^{-1},\beta_i^{-1})_{u^i_{\gamma^i_{e_i}}} \circ
 \varphi_{ \alpha^{-1} } \bigl(u^i_{\beta_i,\gamma^i_{e_i}}\bigr) \co
 \varphi_{ \alpha^{-1} }\bigl( u^i_{\beta_i \gamma^i_{e_i}}\bigr) \to
  \varphi_{ \beta_i^{-1} \alpha^{-1} }(u^i_{\gamma^i_{e_i}}).
$$
Let  $\varepsilon=\varepsilon_a $  be the sign carried by~$a$.  Clearly, the composition of the track~$\gamma$ with a path in the parallel coupon $\widetilde{c_a}$ leading to its unique input is the track $\alpha \, j_1(\beta_1\gamma^1_{e_1})$ of the arc $j_1(e_1)$ of~$\Omega$. Then
$$
u_\gamma= \bigl ( \varphi_{(\alpha \beta_1)^{-1}}(u^1_{\gamma^1_{e_1}})\bigr )^\varepsilon= \bigl ( \varphi_{ \beta_1^{-1}\alpha^{-1}}(u^1_{\gamma^1_{e_1}}) \bigr )^\varepsilon
$$
Considering the  output of $\widetilde{c_a}$, we similarly obtain that
$$
u^\gamma = \bigl ( \varphi_{(\alpha \beta_2)^{-1}}(u^2_{\gamma^2_{e_2}})\bigr )^\varepsilon =\bigl ( \varphi_{ \beta_2^{-1}\alpha^{-1}}(u^2_{\gamma^2_{e_2}}) \bigr )^\varepsilon.
$$
Define the morphism $v_\gamma \co u_\gamma\to u^\gamma$ as
$$
v_\gamma= \begin{cases}
\Psi_2 \circ \varphi_{\alpha^{-1}} (\Phi) \circ (\Psi_1 )^{-1} & \text{if $\varepsilon=+$,} \\
(\Psi_1 \circ \varphi_{\alpha^{-1}} (\Phi^{-1}) \circ (\Psi_2 )^{-1})^* & \text{if $\varepsilon=-$.}
\end{cases}
$$
This defines a coupon-coloring~$v$ of $((M,\Omega),u)$ and  a $\bb$-coloring $(u,v)$  of $(M,\Omega)$.

It follows from the definition of $u$ that the embedding $j_1 \co M_1 \hookrightarrow M$ induces an isomorphism of  $\bb$-colored $G$-surfaces $\partial_-(M_1) \to \partial_-M$. Composing  with the isomorphism  $(h_1)_- \colon -\Sigma_0 \to \partial_-(M_1)$, we obtain an isomorphism  $-\Sigma_0 \to \partial_-M$. Similarly, $j_2$ and $(h_2)_+$  induce an isomorphism  $\Sigma_2 \to \partial_+M$. This yields  an isomorphism   of $\bb$-colored $G$-surfaces
$$
(-\Sigma_0) \sqcup \Sigma_2 \to \partial M=(\partial M, \partial \Omega, \partial u).
$$
The $\bb$-colored $G$-graph $((M,\Omega), u,v)$   endowed with this  isomorphism is a $\bb$-colored   $G$-cobordism between $\Sigma_0$ and $\Sigma_2$ which represents the composition $\chi_2 \circ \chi_1$ in $\Cob^G_\bb$. This composition is well-defined:
it is independent of  the choice of the points  $\{p_C\}_C$ and the tracks above. In particular, if
$p_C \in C \cap h((\Sigma_1)_\bullet)$ is replaced with $p'_C \in C \cap h((\Sigma_1)_\bullet)$, then there is an  isotopy $\{f_t \co M \to M\}_{t\in I}$ such that $f_0=\id_M$, $f_t \vert_{\Omega}=\id_\Omega$ and $f_t \vert_{M\setminus C}=\id_{M\setminus C}$ for all~$t \in I$, and    $f_1(p_C)=p'_C$. The  $\bb$-colored $G$-cobordisms derived from $p_C$ and $p'_C$  are related via the   diffeomorphism  $f_1$ and  the conjugation   by the element of~$G$ represented by the path $t \in I \mapsto f_t(p_C) \in C\setminus \Omega$.
Also, this composition of morphisms in $\Cob^G_\bb$ is associative.

\subsection{Further structures in $\Cob^G_\bb$} \label{sect-pptes-cob-G}\label{sect-cyl-id}
All objects of $\Cob^G_\bb$  have identity endomorphisms in $\Cob^G_\bb$.
Namely,  for a  $\bb$-colored $G$-surface $\Sigma$, the cylinder over~$\Sigma$ (see Section~\ref{sect-cyl}) endowed with the standard identification of its boundary   with $(-\Sigma) \sqcup \Sigma$  is a $\bb$-colored $G$-cobordism   representing  $\id_\Sigma$ in $\Cob^G_\bb$. These identity endomorphisms are identities for the composition in $\Cob^G_\bb$.  So, $\Cob^G_\bb$ is a category.

A more general construction derives from any isomorphism  $f \co \Sigma \to \Sigma'$ between $\bb$-colored $G$-surfaces
a morphism  $  \Sigma \to \Sigma'$ in $\Cob^G_\bb$  called the \emph{cylinder of $f$} and denoted $\cyl(f) $. This morphism is represented by the $\bb$-colored $G$-cobordism between~$\Sigma$ and~$\Sigma'$ which is formed by  the cylinder $C_{\Sigma'}$ over $\Sigma'$ together with     the isomorphism of  $\bb$-colored $G$-surfaces  $  (-\Sigma) \sqcup \Sigma' \to \partial C_{\Sigma'} $  carrying any $x\in \Sigma$     to
$ (f(x), 0) $    and   any $x'\in \Sigma'$ to  $ (x', 1)$.
It is easy to check that if $f \co \Sigma \to \Sigma'$ and $g \co \Sigma' \to \Sigma''$
 are composable isomorphisms between  $\bb$-colored $G$-surfaces, then $\cyl(gf)=\cyl(g) \circ \cyl(f)$. Consequently, the cylinders  are isomorphisms  in  $\Cob^G_\bb$.

The disjoint union of $\bb$-colored $G$-surfaces and $G$-cobordisms turns   $\Cob^G_\bb$   into a symmetric monoidal category.
Its unit object   is   the
empty set  $\emptyset$ viewed as a $\bb$-colored $G$-surface.  The associativity and   unitality constraints  in $\Cob^G_\bb$ are the cylinders  of the tautological isomorphisms
$$
(\Sigma \sqcup  \Sigma') \sqcup \Sigma'' \simeq \Sigma \sqcup  (\Sigma'  \sqcup \Sigma'')
\quad \text{and} \quad
\emptyset \sqcup \Sigma \simeq \Sigma \simeq \Sigma \sqcup \emptyset,
$$
where $\Sigma , \Sigma' , \Sigma''$ run over all $\bb$-colored $G$-surfaces. The   symmetry  in~$\Cob^G_\bb$ is
determined by the cylinders  of the  obvious  permutation isomorphisms
$$
\Sigma \otimes\Sigma'=\Sigma \sqcup\Sigma' \simeq \Sigma' \sqcup\Sigma=\Sigma' \otimes\Sigma.
$$

The category $\Cob^G_\bb$ has a canonical left duality $\{(-\Sigma,\lev_\Sigma)\}_\Sigma$, where $\Sigma$ runs over all $\bb$-colored $G$-surfaces. Here $-\Sigma$ is the opposite $\bb$-colored $G$-surface (see Section~\ref{sect-G-surf}) and the morphism $\lev_\Sigma\co (-\Sigma)\otimes \Sigma \to   \emptyset$ in $\Cob^G_\bb$ is represented
by the  $\bb$-colored $G$-cobordism
formed by  the cylinder $C_{-\Sigma}$ over $-\Sigma$ together with  the  isomorphism  of $\bb$-colored $G$-surfaces
$$
 -((-\Sigma)\otimes \Sigma)  \sqcup \emptyset = \Sigma \sqcup  (-\Sigma) \simeq   (\Sigma \times \{0\}) \sqcup (-\Sigma \times \{1\})=   \partial C_{-\Sigma} \, .
$$
This left duality  turns the symmetric monoidal category $\Cob^G_\bb$ into a ribbon category with trivial twist (see \cite[Lemma~3.5]{TVi5}). In particular, $\Cob^G_\bb$ is spherical.

The category $\Cob^G $ defined in Section~\ref{sect-CobG-v0} is a symmetric monoidal subcategory of $\Cob^G_\bb$. Indeed, any $G$-surface in the sense of Section~\ref{sect-CobG-v0} is a $\bb$-colored $G$-surface with an empty set of marked points (so that the notion of a $\bb$-coloring for this surface is void). Also, any $G$-cobordism between $G$-surfaces (in the sense of Section~\ref{sect-CobG-v0})  is a $\bb$-colored $G$-cobordism with an empty ribbon graph (so that the notion of a $\bb$-coloring for this cobordism is void). This defines an embedding of categories
$\Cob^G \hookrightarrow \Cob^G_\bb$ which is symmetric strict monoidal.

\section{Graph HQFTs and the main theorem}

\subsection{Graph HQFTs}\label{sect-def-graph-HQFT}
Let $\bb$ be a $G$-crossed category over $\kk$.
A \emph{graph HQFT  over $\bb$ with target $\X=K(G,1)$} is a symmetric   strong   monoidal functor
$
Z\co \Cob^G_\bb \to \Mod_\kk
$
where $\Cob^G_\bb$ is  the symmetric monoidal category   defined in  Section~\ref{sect-cobgb} and $\Mod_\kk$ is the symmetric monoidal category  of \kt modules and \kt linear homomorphisms.
Such a functor includes  \kt linear   isomorphisms
$$
Z_0\co \kk \iso   Z(\emptyset) \quad \text{and} \quad Z_2(\Sigma,\Sigma')\co Z(
\Sigma ) \otimes_\kk Z( \Sigma') \iso  Z( \Sigma  \sqcup \Sigma')
$$
for   any  $\bb$-colored $G$-surfaces $\Sigma, \Sigma'$.

Recall that a morphism $\emptyset \to \emptyset$ in $\Cob^G_\bb$ is represented by a triple $(M,\Omega,\id_\emptyset)$, where~$M$ is a closed oriented 3-manifold and~$\Omega$ is a $\bb$-colored $G$-graph in~$M$. Applying~$Z$ to this morphism we get a \kt linear homomorphism  $Z(\emptyset) \to Z(\emptyset)$. Since $Z(\emptyset) \simeq \kk$, the latter  homomorphism is multiplication by an element of $\kk$. This element is  denoted $Z(M,\Omega)$ and  is  an isomorphism invariant  of  $\bb$-colored $G$-graphs in closed oriented 3-manifolds.

As in  topological quantum field theory (TQFT),  for any  $\bb$-colored $G$-surface~$\Sigma$,
the \kt module  $Z(\Sigma)$ is projective of finite type (see, for example, \cite{TVi5}). Clearly, the isomorphism class of this module
is preserved under isotopy of the set of marked points in~$\Sigma$.
Also, each  connected component~$\Gamma$   of~$\Sigma$ may be treated as a $\bb$-colored $G$-surface whose marked points are the marked points of~$\Sigma$ belonging to~$\Gamma$.  The strong monoidality and the symmetry   of~$Z$ yield  a \kt linear isomorphism
\begin{equation}\label{eq-cas-non-connexe}
Z(\Sigma) \simeq  \bigotimes_{\Gamma} Z(\Gamma)
\end{equation}
where $\Gamma$ runs over connected components of~$\Sigma$ and $\otimes$ is the unordered tensor product of \kt modules.

By Section~\ref{sect-cyl-id},   $\bb$-colored $G$-surfaces and the cylinders of their  isomorphisms  form a (non-full) symmetric monoidal  subcategory of  $\Cob^G_\bb$. It is denoted $\mathrm{Homeo}^G_\bb$.
Restricting a graph HQFT~$Z$ to $\mathrm{Homeo}^G_\bb$, we obtain a symmetric monoidal functor  $ \mathrm{Homeo}^G_\bb \to \Mod_\kk$.  In particular,  $Z$ induces a $\kk$-linear representation of the group of isotopy classes of automorphisms of $\Sigma$.

A graph HQFT $Z$ over $\bb$ is \emph{non-degenerate}  if for any $\bb$-colored $G$-surface $\Sigma$,  the \kt module  $Z(\Sigma)$  is   spanned    by the images of the  homomorphisms
$
 Z(\emptyset) \to Z(\Sigma)
$
induced by the morphisms $  \emptyset \to \Sigma$ in  $\Cob^G_\bb$. Formula~\eqref{eq-cas-non-connexe} and the strong monoidality of~$Z$ imply that if this condition holds for all connected $\bb$-colored $G$-surfaces,  then it holds for all disconnected   $\bb$-colored $G$-surfaces  as well.

An  \emph{isomorphism}   between  graph HQFTs    $Z,Z'\co \Cob^G_\bb \to \Mod_\kk$
is a  monoidal natural isomorphism   $ Z  \to Z'$, i.e.,   a family   of  \kt linear isomorphisms
$\{\rho_\Sigma \co Z ( \Sigma)\to  Z'( \Sigma)\}_{\Sigma}$
where~$\Sigma$ runs over all $\bb$-colored $G$-surfaces. These isomorphisms should commute   with the action of $\bb$-colored $G$-cobordisms,  be multiplicative under disjoint unions of $\bb$-colored $G$-surfaces, and satisfy $\rho_\emptyset=Z'_0 \, Z_0^{-1}$.   It is clear that if~$Z$ and~$Z'$      are isomorphic, then  $Z(M,\Omega)= Z'(M,\Omega)$   for any $\bb$-colored $G$-graph  $\Omega$ in a closed oriented 3-manifold~$M$.

Recall from Section~\ref{sect-CobG-v1} that a 3-dimensional HQFT (over $\kk$) is a symmetric strong monoidal functor
$\Cob^G \to \Mod_\kk$. Let~$J$ be the embedding $ \Cob^G \hookrightarrow \Cob^G_\bb$ defined in Section~\ref{sect-pptes-cob-G}. A \emph{graph extension} along~$\bb$ of a 3-dimensional HQFT $Z \co \Cob^G \to \Mod_\kk$   is a graph HQFT $\widetilde{Z}\co \Cob^G_\bb \to \Mod_\kk$ such that
$
\widetilde{Z} \circ J=Z
$
as symmetric monoidal functors. Clearly, two 3-dimensional HQFTs having isomorphic graph extensions are themselves  isomorphic.

\subsection{The surgery graph HQFT}\label{sect-surg-HQFT}
Let $\bb$ be an anomaly free $G$-modular category over $\kk$ (see Section~\ref{sect-modular-graded-cat}).  In~\cite{TVi4} we  derive from~$\bb$  an HQFT  $ \tau_\bb \co \Cob^G \to \Mod_\kk$ called the surgery HQFT. The methods of \cite{Tu1},  \cite{TVi4} yield a non-degenerate graph HQFT $\Cob^G_\bb \to \Mod_\kk$ extending $\tau_\bb$ and  called the \emph{surgery graph HQFT}. It is also denoted by $\tau_\bb$. We state here two formulas computing $\tau_\bb$ for closed surfaces and for graphs in closed 3-manifolds.

For a connected $\bb$-colored $G$-surface $\Sigma$  of genus $g\geq 0$,  the \kt module   $\tau_{\bb}(\Sigma)$   is computed (up to isomorphism) as follows. The surface~$\Sigma$ carries a base point~$\ast$, a finite set of marked points  $A=\{a_1, \dots,a_m\} \subset \Sigma \setminus \{\ast\}$,  and a homotopy class of  maps
$
(\Sigma \setminus A, \ast)\to (\X,\x).
$
Pick a track $\gamma_i$ of $a_i$ for $i = 1, \dots,  m $ so that these~$m$ tracks can  be represented by paths in $\Sigma \setminus A$   meeting  only in their starting point~$\ast$.  Recall from Section~\ref{sect-marked-surfaces} the homotopy class $\mu_i=\mu_{\gamma_i} \in \pi_1(\Sigma \setminus A,\ast)$ of the loop encircling $a_i$. The group $\pi_1(\Sigma \setminus A,\ast)$ is generated by $\mu_{1}, \dots, \mu_{m}$ and $2g$ elements $\alpha_1,\beta_1, \dots, \alpha_g,\beta_g$ subject to the only relation
\begin{equation*}
(\alpha_1^{-1}\beta_1^{-1}\alpha_1\beta_1)\cdots (\alpha_g^{-1}\beta_g^{-1}\alpha_g\beta_g) \, (\mu_{1})^{\varepsilon_1} \cdots (\mu_{m})^{\varepsilon_m}=1,
\end{equation*}
where  $\varepsilon_i=\pm 1$ is the sign of $a_i$. As usual, denote  the element of~$G$ represented by the image of
$\mu_{i}$   in~$\X$ by the same symbol $\mu_i$, and similarly for $\alpha_j, \beta_j$.  Let $u_i\in \bb_{\mu_i}$ be the color of $\gamma_i$. Set
$$
U_\Sigma=u_1^{\varepsilon_1} \otimes \cdots  \otimes u_m^{\varepsilon_m} \in \bb_{(\mu_{1})^{\varepsilon_1} \cdots (\mu_{m})^{\varepsilon_m}}.
$$
(If $m=0$, then $U_\Sigma=\un \in \bb_1$.) Let $\mathcal{J}=\amalg_{\alpha\in G}\, \mathcal{J}_\alpha$  be  a   representative set of simple objects of~$\bb$.   Then there is an $\kk$-linear isomorphism
\begin{equation}\label{eq-def-tau-B-Sigma}
\tau_{\bb}(  \Sigma  )   \simeq
\!\!\!\!\!\!\!
\bigoplus_{J_1\in \mathcal{J}_{\beta_1}, \ldots,J_g \in \mathcal{J}_{\beta_g}}
\!\!\!\!\!\!\!
\Hom_{\bb}\Bigl (\un_{\bb}, \bigl(\varphi_{\alpha_1}(J_1)^* \otimes J_1 \bigr) \otimes \cdots \otimes \bigl(\varphi_{\alpha_g}(J_g)^* \otimes J_g \bigr) \otimes   U_\Sigma \Bigr).
\end{equation}

We recall      the definition of the  scalar     invariant  $\tau_\bb (M,\Omega)\in \kk$ of a  $\bb$-colored $G$-graph  $(M,\Omega)$ where $M$ is a closed connected oriented 3-manifold.
The Lickorish-Wallace  theorem on  surgery presentations of 3-manifolds   implies that there are
a   framed   link $L=L_1\cup \cdots \cup   L_n    $   in $\R^2 \times (0,1) $ and a ribbon graph~$\Omega'$ in $(\R^2 \times (0,1) ) \setminus L$   such that the   surgery on $S^3=\RR^3\cup\{\infty\}$ along $L$ turns  the pair $(S^3, \Omega')$ into   $(M,\Omega)$, at least  up to an orientation-preserving  homeomorphism. Set
$$
E=(\R^2 \times (0,1)) \setminus (\Omega' \cup L) \subset M \setminus \Omega
$$
and take any point $z \in E$  with big second
coordinate. We can assume that $M_\bullet =\Omega_\bullet =\{z\}$. Restricting the given homotopy class of maps $(M \setminus \Omega,M_\bullet) \to (\X,\x)$   to~$E $ we obtain     a homotopy class of maps $(E,z) \to (\X,\x)$.
Orient the link~$L$ arbitrarily. Stabilize each component $L_i$ of $L$ by inserting into it a small coupon~$c_i$ as in Section~\ref{sect-stabilization}. We choose the bottom base of~$c_i$ so that its unique input is directed out of $c_i$. This turns $\Omega'\cup L$ into a $G$-graph in $\R^2 \times (0,1)$   denoted~$\overline \Omega$.
For each $i \in \{ 1, \dots,  n\}$,  pick a track $\gamma(i)$ of $c_i$ and set $\mu_i=\mu_{\gamma(i)}\in \pi_1(E,z)$.
Let $\mathrm{col}(L)$ be the set of all maps $\lambda\co\{ 1, \dots,  n\} \to \mathcal{J}$ such that
 $\lambda(i) \in  \bb_{\mu_i}$ for all $i \in \{ 1, \dots,  n\}$.
Any such~$\lambda$  together with the $\bb$-coloring $(u,v)$ of $\Omega$ determine a $\bb$-coloring $(\overline{u}, \overline{v})$ of~$\overline{\Omega}$ as follows. Pick a 1-system of tracks $\{\gamma'_e\}_{e}$ for $\Omega'$.
For each $i \in \{ 1, \dots,  n\}$, the composition of the track~$\gamma(i)$ with a path in the parallel coupon~$\widetilde{c_i}$ (obtained by slightly pushing $c_i$ along the  framing) leading to its unique input is a track $\gamma_i$ of the single  arc of $\overline{\Omega}$ contained in~$L_i$. Note that $\mu_{\gamma_i}=\mu_i$. Then the family $\{\gamma'_e\}_{e} \cup \{\gamma_i\}_i$ is a 1-system of tracks for $\overline{\Omega}$. By Lemma~\ref{lem-precolorings-G-graphs-in}, the family of objects $\{u_{\gamma'_e}\}_{e} \cup \{\lambda(i)\}_i$ determines
a $\bb$-precoloring~$\overline{u}$ of~$\overline{\Omega}$ such that $\overline{u}_{\gamma'_e}=\varphi_1(u_{\gamma'_e})$ and
$\overline{u}_{\gamma_i}=\varphi_1(\lambda(i))$. Pick a 2-system of tracks $\{\gamma'_c\}_c$ for $\Omega'$.
Then the family $\{\gamma'_c\}_{c} \cup \{\gamma(i)\}_i$ is a 2-system of tracks for $\overline{\Omega}$.
Consider the $\bb$-coloring $(u^1,v^1)$ of $\Omega'$ obtained from $(u,v)$ by conjugation by $1\in G$ (see Section~\ref{sect-stabilization}).
It follows from the definition that for any coupon $c$ of $\Omega'$,  $(\overline{u})_{\gamma'_c}=(u^1)_{\gamma'_c}$ and $(\overline{u})^{\gamma'_c}=(u^1)^{\gamma'_c}$.
For each $i \in \{ 1, \dots,  n\}$, the composition of the track~$\gamma(i)$ with a path in $\widetilde{c_i}$ leading to its unique output is a track~$\gamma^i$ of the single  arc of $\overline{\Omega}$ contained in $L_i$.
Clearly,    $\mu_{\gamma^i}=\mu_i$ and $\gamma^i=\lambda_i \gamma_i $ where $\lambda_i \in \pi_1(E,z)$ is the longitude  of $L_i$ determined by $\gamma (i)$ and  the  orientation and the framing of~$L_i$. By definition of the surgery, the image of $\lambda_i$ in~$\X$ represents $1 \in G$ and so $\overline{u}_{\gamma^i}=\varphi_1(\lambda(i))$. Then $(\overline{u})_{\gamma(i)}=(\overline{u})^{\gamma(i)}=\varphi_1(\lambda(i))$.
By Section~\ref{sect-contruct-colorings}, the family of morphisms
$
\{(v^1)_{\gamma'_c}\co (\overline{u})_{\gamma'_c} \to (\overline{u})^{\gamma'_c}\}_{c} \cup \{\id_{\varphi_1(\lambda(i))}\co (\overline{u})_{\gamma(i)} \to (\overline{u})^{\gamma(i)}\}_i
$
determines a coupon-coloring~$\overline{v}$ of~$(\overline{\Omega},\overline{u})$.  Hence $(\overline{u},\overline{v})$ is a  $\bb$-coloring of $\overline{\Omega}$. Denote the resulting $\bb$-colored $G$-graph by $\overline{\Omega}^\lambda$. It follows from Section~\ref{sect-contruct-colorings} that the isomorphism class of $\overline{\Omega}^\lambda$ does not depend on the choices of the tracks.
Recall from \cite{TVi4} the  monoidal functor $F_{\bb}$ from the category of isotopy classes of $\bb$-colored $G$-graphs in $\R^2 \times [0,1]$ to the category~$\bb$. In particular, this functor yields   an isotopy invariant
$
F_\bb(\overline \Omega^\lambda)\in \End_\bb(\un)=\kk$.
Then
\begin{equation}\label{eq-def-tau-B-M-Omega}
\tau_{\bb}(M,\Omega)=\Delta^{-n-1}\sum_{\lambda\in\mathrm{col}(L)}
\left(\prod_{i=1}^n \dim\bigl(\lambda( i)\bigr) \right) F_\bb(\overline \Omega^\lambda) \in \kk,
\end{equation}
where $\Delta$ is the canonical rank of~$\bb$ defined in Section~\ref{sect-modular-graded-cat}.

\subsection{The state-sum graph HQFT}\label{sec-statesumgraph}
Assume that $\kk$ is a field and let $\cc$ be a spherical $G$-fusion  category over~$\kk$ with  $\dim(\cc_1)\neq 0$. In \cite{TVi2}, we  derive  from~$\cc$  a 3-dimensional HQFT $\vert \cdot \vert_\cc\co \Cob^G \to \Mod_\kk$ called the \emph{state sum HQFT}.  By Section~\ref{sect-center-graded-cat},
the $G$-center $\zz_G(\cc)$ of~$\cc$ is a $G$-ribbon (and so $G$-crossed) category over~$\kk$.
In Section~\ref{sect-proof}, we prove the following theorem:

\begin{thm}\label{thm-extension-state-sum-graph-HQFT}
The state sum HQFT $\vert \cdot \vert_\cc$ extends to a graph HQFT over $\zz_G(\cc)$.
\end{thm}

The proof of  this theorem goes by extending  the state sum method of \cite{TVi2}  to so-called knotted plexuses in skeletons (which represent ribbon graphs, see Section~\ref{sect-knotted-plexuses-in-skel}) via an invariant of colored knotted $G$-nets (see Section~\ref{sect-prelimoncolorednets}).   The graph HQFT  constructed  in the proof of Theorem~\ref{thm-extension-state-sum-graph-HQFT}  is called the \emph{state sum graph HQFT} and is again denoted  by $\vert \cdot \vert_\cc $.

\subsection{Comparison theorem}\label{sect-main-result}
Assume that $\kk$ is an algebraically closed field and let $\cc=\oplus_{g\in G} \, \cc_g$ be an additive spherical $G$-fusion category over $\kk$ with  $\dim(\cc_1) \neq 0$.  By Section~\ref{sect-center-graded-cat}, the  $G$-center $\zz_G(\cc)$ of~$\cc$ is an additive anomaly free $G$-modular category over $\kk$.
By Sections~\ref{sect-surg-HQFT} and \ref{sec-statesumgraph}, the category $\cc$ gives rise to two graph HQFTs: the surgery graph HQFT $\tau_{\zz_G(\cc)} \co \Cob^G_{\zz_G(\cc)} \to \Mod_\kk$ and the state sum graph HQFT  $\vert \cdot \vert_\cc \co \Cob^G_{\zz_G(\cc)} \to \Mod_\kk$.   Our main result is the following theorem:

\begin{thm}\label{thm-comparison}
The graph HQFTs $\tau_{\zz_G(\cc)}$ and $\vert \cdot \vert_\cc$ are  isomorphic.
\end{thm}

This theorem yields a surgery computation  of $\vert \cdot \vert_\cc$. For $G=\{1\}$,  Theorem~\ref{thm-comparison}    was  first  established in \cite{TVi1} and  independently (in the case   ${\rm char} (\kk)=0$) in \cite{Ba2}. We prove Theorem~\ref{thm-comparison} in Section~\ref{sect-comparison}.

By Section~\ref{sect-def-graph-HQFT}, any graph HQFT yields a scalar invariant of colored $G$-graphs in closed oriented 3-manifolds. The next claim directly follows from Theorem~\ref{thm-comparison}.

\begin{cor}\label{cor-omegaomega}
For any $\zz_G(\cc)$-colored $G$-graph  $\Omega$  in a closed oriented 3-manifold~$M$,
$$
\tau_{\zz_G(\cc)}(M,\Omega) =\vert M,\Omega \vert_\cc.
$$
\end{cor}

The following   corollary of Theorem~\ref{thm-comparison} is Theorem~\ref{thm-comparison-intro} of the introduction:

\begin{cor}\label{iso-additivecase}
The   surgery    HQFT  $\tau_{\zz_G(\cc)} \co \Cob^G  \to \Mod_\kk$ and the state sum   HQFT  $\vert \cdot \vert_\cc \co \Cob^G  \to \Mod_\kk$ are isomorphic.
\end{cor}

\begin{proof}
Both HQFTs extend to graph HQFTs   which are isomorphic by Theorem~\ref{thm-comparison}. Restricting the isomorphism in question   to $\Cob^G$ we obtain an isomorphism of the original    HQFTs.
\end{proof}

Applying Corollary~\ref{cor-omegaomega} to empty $G$-graphs, we get the following:

\begin{cor}
For any closed 3-dimensional  $G$-manifold~$M$,
$$
\tau_{\zz_G(\cc)}(M) =\vert M \vert_\cc.
$$
\end{cor}

Note finally that Theorem~\ref{thm-comparison} and  the  non-degeneracy of  the  surgery graph   HQFT $\tau_{\zz(\cc)}$  imply that the state sum graph HQFT  $\vert \cdot \vert_\cc$ is non-degenerate.

\subsection{Remark}
Using the language of higher categories, one can define extended HQFTs. Any graph HQFT induces a 2-extended 3-dimensional HQFT in the sense of \cite{SW} with values in the 2-category of 2-vector spaces. This observation uses the fact that the complement of a ribbon tangle with no coupons in a 3-manifold is a 3-manifold with corners of codimension~2. Theorem~\ref{thm-comparison} implies that (under the assumptions of this theorem) the 2-extended 3-dimensional HQFTs induced by
the graph HQFTs  $\vert \cdot \vert_\cc$ and $\tau_{\zz_G(\cc)}$ are  isomorphic.

\section{An invariant of colored knotted nets}\label{sect-prelimoncolorednets}

In this section, we define an invariant of colored knotted nets which generalizes $6j$-symbols and which is used below  to construct a state sum graph HQFT.  Until the end of this section, we assume  $\kk$ to be a field and fix a $G$-fusion  category $\cc$  over~$\kk$.
Recall that the $G$-center $\zz_G(\cc)$ of~$\cc$ is then a $G$-braided category over~$\kk$ (see Section~\ref{sect-center-graded-cat}).
Note that all constructions of this section work word for word for a larger class of $G$-graded categories, namely for non-singular $G$-graded categories over arbitrary commutative rings (see Appendix~\ref{sect-nonsing-graded-cat}).

\subsection{Knotted nets}\label{sect-coloredknottednets}
We  recall  the notion of a knotted net in an oriented surface  introduced in    \cite[Section 15.3.2]{TVi5}.
A  \emph{net} $\Gamma$  is a topological space    obtained from a   disjoint
union of a finite number  of   oriented  circles,  oriented  arcs, and oriented coupons (see Section~\ref{Ribbon graphs})
by gluing    some   endpoints of the arcs to the bases of the coupons or to each other.   We require  that different endpoints of  the arcs are never glued to the same  point of a (base of a) coupon.
The images  in~$\Gamma$ of the   arcs,  half-arcs, circles, and coupons are      called  respectively   \emph{edges},  \emph{half-edges},   \emph{circles}, and \emph{coupons}
of~$\Gamma$.  The images  in~$\Gamma$ of the    endpoints of the arcs  that are not glued to coupons (but may be glued to each other) are called  \emph{vertices}  of~$\Gamma$. Every vertex of~$\Gamma$  is incident to a certain number of half-edges of~$\Gamma$, called the \emph{valency} of the vertex. The edges and   circles of~$\Gamma$ are collectively called  \emph{strands}.

A \emph{knotted net}~$\Gamma$   in an oriented surface~$\Sigma$  is a net  immersed in $\Sigma\setminus \partial \Sigma$ such that
\begin{enumerate}
\labeli
\item   all coupons  of~$\Gamma$ are embedded  in $\Sigma$ preserving   orientation;
\item  all multiple points of the immersion  are double transversal
intersections of   the interiors of strands  of~$\Gamma$. At every   double point, one of the two meeting strands is distinguished.
\end{enumerate}
Note that the coupons and the vertices of a knotted net   are pairwise disjoint.
The double points of a knotted net are called \emph{crossing points} or just \emph{crossings}.  They are finite in number and lie   away from the coupons and the vertices. A  crossing of a knotted net  in $\Sigma$ lies in  an open   disk   in $\Sigma$   represented   in our pictures  by a  plane parallel to the page. The orientation of $\Sigma$
is represented by the counterclockwise orientation
of this plane.  The distinguished strand  at the crossing is represented by a red continuous (unbroken) line:
$$
  \rsdraw{.45}{.9}{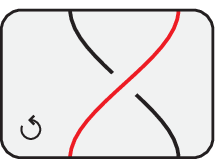} \;\,.
$$

Each crossing~$x$ of a knotted net~$\Gamma$   in~$\Sigma$ gives rise to two points on the strands of~$\Gamma$: the overcrossing  $x_{\rm {ov}}$ lying in the distinguished strand  ($x_{\rm {ov}}$ is  represented in our pictures by a red point) and  the undercrossing  $x_{\rm {un}}$. The overcrossings split the strands of~$\Gamma$ into  consecutive   segments called  {\it underpasses}.  A crossing~$x$ of~$\Gamma$ determines three underpasses: the underpass~$\underline x$ containing the point   $x_{\rm {un}}$ and   two underpasses ${x^{-}}, {x^{+}}$ separated by the  point  $x_{\rm {ov}}$. One of the underpasses ${x^{-}}, {x^{+}}$  is directed towards $x_{\rm {ov}}$ and the other one is directed away from $x_{\rm {ov}}$. We choose notation so that  ${x^{+}}$ is directed towards $x_{\rm {ov}}$ if the crossing~$x$ is   positive  and away from $x_{\rm {ov}}$ if~$x$ is   negative:
$$
\psfrag{u}[Bl][Bc]{\scalebox{.9}{$\underline{x}$}}
\psfrag{x}[Bl][Bl]{\scalebox{.9}{$\underline{x}$}}
\psfrag{c}[Bl][Bl]{\scalebox{.9}{$x^-$}}
\psfrag{r}[Bc][Bc]{\scalebox{.9}{$x^+$}}
\rsdraw{.45}{.9}{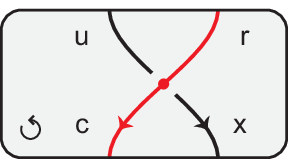} \qquad \quad
\psfrag{c}[Bc][Bc]{\scalebox{.9}{$x^-$}}\rsdraw{.45}{.9}{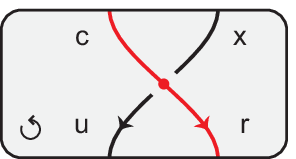} \;\,.
$$

\subsection{Colored knotted nets}\label{sect-coloredknottednets+}
A \emph{$\cc$-coloring} $U$ of a knotted net~$\Gamma$   in an oriented surface~$\Sigma$ comprises three functions. The first function assigns to each     underpass $p$ of $\Gamma$ a homogeneous object $U_p$ of~$\cc$ or of~$\zz_G(\cc)$,  called the \emph{color} of $p$.  We require  that  if~$p$ lies in a  strand of~$\Gamma$
which   is   incident to a coupon or   is the distinguished strand of  at least one  crossing,   then   $U_p\in \zz_G(\cc)$.
For any other~$p$, either $U_p\in \zz_G(\cc)$ or $U_p\in \cc$.
The second function  assigns to every   crossing~$x$ of~$\Gamma$  an isomorphism
\begin{equation}\label{weakiso-}
U_{x}\colon  U_{x^+} \to \varphi_{ \vert U_{\underline{x}} \vert  } (U_{x^-})
\end{equation}
in $\zz_G(\cc)$ called   the \emph{color} of~$x$. Here $\varphi$ is the crossing of the category $\zz_G(\cc)$ and, as usual,   the degree in~$G$ of a  homogenous object~$X$ of a $G$-graded category is denoted by    $ \vert X  \vert  $. If the colors involved  are non-zero objects, then the existence of an isomorphism \eqref{weakiso-}   implies that
$$
\vert U_{x^+} \vert = \vert U_{\underline{x}} \vert^{-1}\,  \vert U_{x^-}\vert \, \vert U_{\underline{x}} \vert \in G .
$$
The third function   assigns to  each coupon~$c$ of   $\Gamma$ a morphism $c_\mathrm{in} \to c^\mathrm{out}$ in $\zz_G(\cc)$, where $c_\mathrm{in}$ (respectively, $c^\mathrm{out}$) is the object  of $\zz_G(\cc)$   determined  in the usual way   by the colors and orientations of the inputs (respectively, outputs) of~$c$. This morphism $c_\mathrm{in} \to c^\mathrm{out}$  is called the \emph{color} of~$c$.

A \emph{knotted $\cc$-net} is a knotted net endowed with a $\cc$-coloring.
Given a knotted $\cc$-net $\Gamma$   in~$\Sigma$,    the \emph{subcolor} of an underpass colored with   an object~$X$   is~$X$ itself  if $X\in \cc$   and is the image $\uu(X) \in \cc$
of~$X$ under the forgetful functor  $\uu \co \zz_G(\cc) \to \cc$ if   $X \in \zz_G(\cc)$.  The \emph{subcolor} of a crossing/coupon is the image of its color  under the   functor  $ \uu$. Note that the above notion of a knotted $\cc$-net   generalizes the one in \cite[Section 15.3.3]{TVi5}
where $G=1$ and the crossings are not colored.

We will use the notions of a cyclic $\cc$\ti set and its associated multiplicity module, see
\cite[Chapter 12]{TVi5}. A vertex~$v$ of a knotted $\cc$-net~${\Gamma}$   in~$\Sigma$    determines a cyclic $\cc$\ti set $ (E_v, c_v, \varepsilon_v)$ as follows:  $E_v$ is the set of half-edges of~$\Gamma$ incident to~$v$ with cyclic order induced by the
opposite orientation of~$\Sigma$, the map $c_v\co E_v \to \Ob(\cc)$ assigns to a half-edge
$e\in E_v$  the subcolor of the edge of~$\Gamma$ containing~$e$, and   the map   $\varepsilon_v \co E_v \to \{+,-\}$ assigns to
$e\in E_v$  the sign~$+$ if~$e$ is oriented towards $v$ and $-$ otherwise. Let $H_v(\Gamma)=H(E_v )$ be the  multiplicity   module  of $E_v$.
The \kt  module $H_v(\Gamma)$ can be described   as follows. Let $n \geq 1$ be the valence of~$v$ and let
$e_1 < e_2 < \cdots < e_n < e_1$ be the half-edges of~$\Gamma$ incident to~$v$ with cyclic order induced by the
opposite orientation of~$\Sigma$. Then  we have the cone isomorphism
$$
\tau^v_{e_1} \co H_v(\Gamma) \stackrel{\simeq} {\longrightarrow}  \Hom_\cc(\un, X_1^{\varepsilon_1} \otimes  \cdots \otimes X_n^{\varepsilon_n} ).
$$
where $X_r=c_v(e_r)$ and $\varepsilon_r=\varepsilon_v(e_r)$ are the subcolor  and sign of~$e_r$ for all~$r$.  Set
$$
H(\Gamma)=\otimes_v \, H_v(\Gamma),
$$
where $v$ runs over all vertices of~$\Gamma$ and   $\otimes$ is the unordered tensor product  of \kt modules.
To emphasize the role of~$\Sigma$, we   sometimes write $H_v(\Gamma;\Sigma)$ for~$H_v(\Gamma)$ and
$H(\Gamma;\Sigma)$ for~$H(\Gamma)$.   If $\Gamma$ has no vertices, then by definition $ H({\Gamma})=\kk$.

An \emph{isotopy} of knotted $\cc$-nets    in~$\Sigma$ is an ambient isotopy  in the class of knotted $\cc$-nets  in  $\Sigma$  preserving  all the data, that is the vertices, the strands, the crossings (with their   distinguished strand),
the coupons (with their distinguished base),  the orientations,   and the colors. An  isotopy
between two  knotted $\cc$-nets $\Gamma$ and $\Gamma'$  in~$\Sigma$ induces   a \kt linear isomorphism   $ H({\Gamma}) \to H({\Gamma'})$ in the obvious way.

Any orientation preserving embedding $f $ of $\Sigma$ into an oriented surface $\Sigma'$ carries a   knotted $\cc$-net $\Gamma$ in $ \Sigma$  into  a knotted $\cc$-net $\Gamma'  =f(\Gamma)   $ in $ \Sigma'$ and
induces a \kt linear isomorphism
$
H(f)\co   H(\Gamma;\Sigma ) \to H(\Gamma';\Sigma')
$
in the obvious way. This applies,   in particular, when   $f$ is an  orientation preserving self-homeomorphism of~$\Sigma$.

\subsection{An invariant of knotted $\cc$-nets in $\R^2$}\label{sect-inv-coloredknottednets}
Let $\Gamma$ be a knotted $\cc$-net in  the plane $\R^2$ (oriented counterclockwise).
Pick a   vector $\alpha_v \in H_v({\Gamma})$ for  every   vertex~$v$   of~$\Gamma$ and  transform $\Gamma$ at its vertices, crossings, and coupons to obtain a Penrose diagram as follows. First, at   each     vertex~$v$ of~$\Gamma$,
pick  a half-edge $e_v\in E_v$, isotop $\Gamma$ near~$ v $  so that   the half-edges incident to~$v$   lie
above~$v$ with respect to the second coordinate on~$\R^2$ and~$e_v$ is the leftmost of them,  and   replace
$v$ by a box colored with $\tau^v_{e_v}(\alpha_v)$, where   $\tau^v$
is the universal cone of $H_v(\Gamma)$:
$$
  \psfrag{R}[Bc][Bc]{\scalebox{.9}{$\R^2$}}
  \psfrag{h}[Bc][Bc]{\scalebox{.9}{$\tau^v_{e_v}(\alpha_v)$}}
  \psfrag{e}[Bl][Bl]{\scalebox{.9}{$e_v$}}
  \psfrag{v}[Br][Br]{\scalebox{.9}{$v$}}
\rsdraw{.45}{.9}{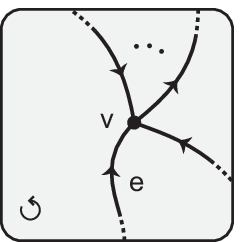} \quad \rsdraw{.45}{.8}{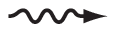} \quad
  \psfrag{e}[Br][Br]{\scalebox{.9}{$e_v$}}
  \psfrag{v}[Bc][Bc]{\scalebox{.9}{$v$}}
\rsdraw{.45}{.9}{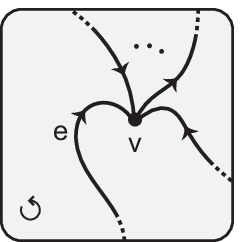} \quad \rsdraw{.45}{.8}{fleche} \quad \rsdraw{.45}{.9}{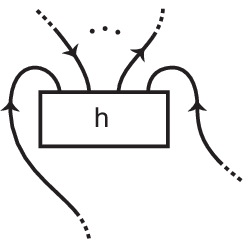} \,.
$$
Next, at each crossing~$x$ of~$\Gamma$,  isotop~$\Gamma$ near~$x$ to make to  ensure that the strands are oriented downward.
Consider the color $(A,\sigma) \in \zz_G(\cc)$ of the underpass $x^+$, the color $(B,\gamma) \in \zz_G(\cc)$ of the underpass $x^-$, the subcolor $X\in \cch$ of the underpass~$\underline{x}$, and the subcolor $\psi\co  A \to \uu\bigl(\varphi_{|X|}(B,\gamma)\bigr)$ of the crossing $x$. If the crossing~$x$ is positive, then replace $x$ as follows:
$$
\psfrag{u}[Bl][Bc]{\scalebox{.9}{$\underline{x}$}}
\psfrag{x}[Bl][Bl]{\scalebox{.9}{$\underline{x}$}}
\psfrag{c}[Bl][Bl]{\scalebox{.9}{$x^-$}}
\psfrag{r}[Bc][Bc]{\scalebox{.9}{$x^+$}}
\psfrag{X}[Bl][Bl]{\scalebox{.9}{$X$}}
\psfrag{Y}[Br][Br]{\scalebox{.9}{$X$}}
\psfrag{n}[Bc][Bc]{\scalebox{.9}{$\tau_{(B,\gamma),X}$}}
\psfrag{v}[Bc][Bc]{\scalebox{.9}{$\psi^{-1}$}}
\psfrag{s}[Bc][Bc]{\scalebox{.9}{$\tau_{(B,\gamma),X}^{-1}$}}
\psfrag{e}[Bc][Bc]{\scalebox{.9}{$\psi$}}
\psfrag{E}[Bl][Bl]{\scalebox{.9}{$A$}}
\psfrag{D}[Bl][Bl]{\scalebox{.9}{$\uu\bigl(\varphi_{|X|}(B,\gamma)\bigr)$}}
\psfrag{C}[Br][Br]{\scalebox{.9}{$B$}}
\rsdraw{.45}{.9}{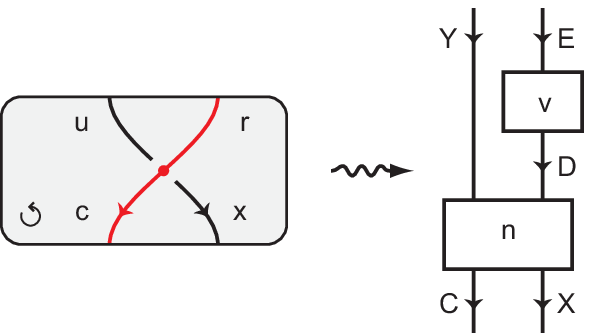}
$$
where $\tau$ is the enhanced $G$-braiding of $\zz_G(\cc)$, see Appendix~\ref{sect-appendix-enhanced-G-braid}. If the crossing~$x$  is negative, then  replace $x$ as follows:
$$
\psfrag{u}[Bl][Bc]{\scalebox{.9}{$\underline{x}$}}
\psfrag{x}[Bl][Bl]{\scalebox{.9}{$\underline{x}$}}
\psfrag{c}[Bc][Bc]{\scalebox{.9}{$x^-$}}
\psfrag{r}[Bc][Bc]{\scalebox{.9}{$x^+$}}
\psfrag{X}[Bl][Bl]{\scalebox{.9}{$X$}}
\psfrag{Y}[Br][Br]{\scalebox{.9}{$X$}}
\psfrag{n}[Bc][Bc]{\scalebox{.9}{$\tau_{(B,\gamma),X}$}}
\psfrag{v}[Bc][Bc]{\scalebox{.9}{$\psi^{-1}$}}
\psfrag{s}[Bc][Bc]{\scalebox{.9}{$\tau_{(B,\gamma),X}^{-1}$}}
\psfrag{e}[Bc][Bc]{\scalebox{.9}{$\psi$}}
\psfrag{E}[Bl][Bl]{\scalebox{.9}{$A$}}
\psfrag{D}[Bl][Bl]{\scalebox{.9}{$\uu\bigl(\varphi_{|X|}(B,\gamma)\bigr)$}}
\psfrag{C}[Br][Br]{\scalebox{.9}{$B$}}
\rsdraw{.45}{.9}{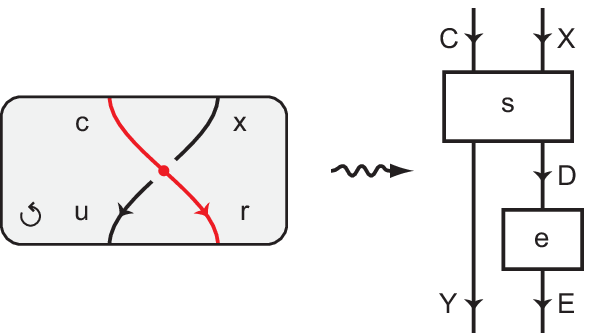} \,.
$$
Finally, at each coupon~$c$  of~${\Gamma}$,  isotop~$\Gamma$ near~$c$ to make the bases of~$c$   horizontal  and to ensure that the distinguished (bottom) base lies below the opposite  base   with respect to the second coordinate on $\RR^2$, replace~$c$ by a box with the same inputs and outputs as~$c$, and label this box with the subcolor of~$c$:
$$
\psfrag{M}[Br][Br]{\scalebox{.9}{\textcolor{red}{$(A_1,\sigma_1)$}}}
\psfrag{W}[Br][Br]{\scalebox{.9}{\textcolor{red}{$(B_1,\gamma_1)$}}}
\psfrag{X}[Bl][Bl]{\scalebox{.9}{\textcolor{red}{$(A_n,\sigma_m)$}}}
\psfrag{Y}[Bl][Bl]{\scalebox{.9}{\textcolor{red}{$(B_n,\gamma_n)$}}}
\psfrag{A}[Br][Br]{\scalebox{.9}{$A_1$}}
\psfrag{B}[Br][Br]{\scalebox{.9}{$B_1$}}
\psfrag{C}[Bl][Bl]{\scalebox{.9}{$A_n$}}
\psfrag{D}[Bl][Bl]{\scalebox{.9}{$B_n$}}
\psfrag{u}[Bc][Bc]{\scalebox{.9}{\textcolor{red}{$f$}}}
\psfrag{v}[Bc][Bc]{\scalebox{.9}{$f$}}
\rsdraw{.45}{.9}{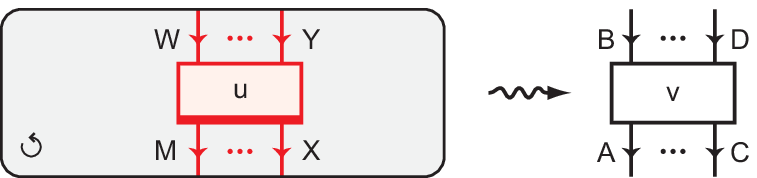} \,.
$$
Also,   the colors of all edges of~$\Gamma$ are traded for  the corresponding  subcolors. This turns~$\Gamma$ in  a $\cc$-colored Penrose diagram without free ends. The Penrose calculus associates with  this diagram an element of $ \End_\cc(\un)$  denoted $\inv_\cc({\Gamma})(\otimes_v \, \alpha_v)$. By linear extension, this procedure defines a \kt linear homomorphism
$$
\inv_\cc(\Gamma)\co H(\Gamma) =\otimes_v \, H_v(\Gamma) \to \End_\cc(\un).
$$
This homomorphism is an isotopy invariant of   the   $\cc$-colored graph $\Gamma$. More precisely, for any isotopy~$\iota$ between  two knotted $\cc$-nets $\Gamma$ and $\Gamma'$ in $\R^2$,
we have  $\inv_\cc(\Gamma')\, H(\iota)=\inv_\cc(\Gamma)$, where $H(\iota)\co H(\Gamma) \to H(\Gamma')$ is the  \kt linear   isomorphism induced by $\iota$.

\subsection{An example}
Let $\Gamma$ be the following knotted $\cc$-net in $\R^2$:
$$
 \psfrag{A}[Br][Br]{\scalebox{.9}{\textcolor{red}{$(A,\sigma)$}}}
 \psfrag{B}[Br][Br]{\scalebox{.9}{\textcolor{red}{$(B,\gamma)$}}}
 \psfrag{C}[Br][Br]{\scalebox{.9}{\textcolor{red}{$(C,\omega)$}}}
 \psfrag{D}[Bl][Bl]{\scalebox{.9}{\textcolor{red}{$(D,\delta)$}}}
 \psfrag{E}[Bl][Bl]{\scalebox{.9}{\textcolor{red}{$(E,\varepsilon)$}}}
 \psfrag{X}[Br][Br]{\scalebox{.9}{$X$}}
 \psfrag{r}[Bc][Bc]{\scalebox{.9}{\textcolor{red}{$f$}}}
 \psfrag{a}[Bl][Bl]{\scalebox{.9}{\textcolor{red}{$\phi$}}}
 \psfrag{c}[Bl][Bl]{\scalebox{.9}{\textcolor{red}{$\psi$}}}
 \psfrag{u}[Bc][Bc]{\scalebox{.9}{$u$}}
 \psfrag{v}[Bc][Bc]{\scalebox{.9}{$v$}}
\Gamma=\,\rsdraw{.45}{.9}{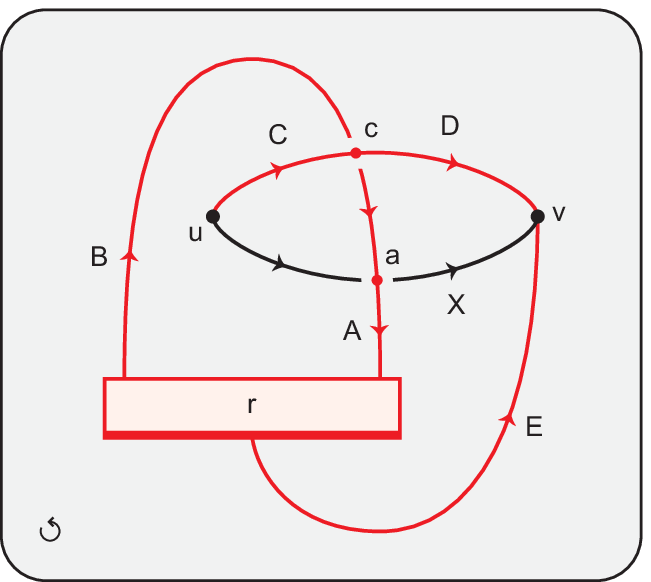}  \,.
$$
The net $\Gamma$ has two vertices $u$ and $v$, six underpasses colored $(A,\sigma)$, $(B,\gamma)$, $(C,\omega)$, $(D,\delta)$,
$(E,\varepsilon)\in \zz_G(\cc)_{\mathrm{hom}}$ and $X \in \cch$, one coupon colored by a morphism in~$\zz_G(\cc)$
$$
f \co (E,\varepsilon) \to (B,\gamma)^* \otimes (A,\sigma),
$$
and two crossings colored by isomorphisms in~$\zz_G(\cc)$
$$
\phi \co (B,\gamma) \to \varphi_{|X|}(A,\sigma) \quad \text{and} \quad \psi \co (D,\delta) \to \varphi_{|B|}(C,\omega).
$$
Clearly, $H(\Gamma)=H_u(\Gamma) \otimes H_v(\Gamma)$. Pick $\alpha \in H_u(  \Gamma  )$ and $\beta \in  H_v(  \Gamma   ) $. Then, by definition,  we have
$$
  \psfrag{D}[Bl][Bl]{\scalebox{.9}{$\uu\bigl(\varphi_{|X|}(A,\sigma)\bigr)$}}
  \psfrag{Q}[Bl][Bl]{\scalebox{.9}{$\uu\bigl(\varphi_{|B|}(C,\omega)\bigr)$}}
  \psfrag{e}[Bc][Bc]{\scalebox{.9}{$\psi$}}
  \psfrag{v}[Bc][Bc]{\scalebox{.9}{$\phi^{-1}$}}
  \psfrag{s}[Bc][Bc]{\scalebox{.9}{$\tau_{(C,\omega),B}^{-1}$}}
  \psfrag{n}[Bc][Bc]{\scalebox{.9}{$\tau_{(A,\sigma),X}$}}
  \psfrag{b}[Bc][Bc]{\scalebox{.9}{$\tilde{\alpha}$}}
  \psfrag{f}[Bc][Bc]{\scalebox{.9}{$\tilde{\beta}$}}
  \psfrag{X}[Bl][Bl]{\scalebox{.9}{$X$}}
  \psfrag{Y}[Bl][Bl]{\scalebox{.9}{$X$}}
  \psfrag{C}[Br][Br]{\scalebox{.9}{$A$}}
  \psfrag{E}[Bl][Bl]{\scalebox{.9}{$B$}}
  \psfrag{V}[Bl][Bl]{\scalebox{.9}{$D$}}
  \psfrag{T}[Br][Br]{\scalebox{.9}{$B$}}
  \psfrag{P}[Bl][Bl]{\scalebox{.9}{$C$}}
  \psfrag{W}[Bl][Bl]{\scalebox{.9}{$E$}}
  \psfrag{k}[Bc][Bc]{\scalebox{.9}{$f$}}
  \inv_\cc (\Gamma)(\alpha \otimes \beta) \, = \,\rsdraw{.45}{.9}{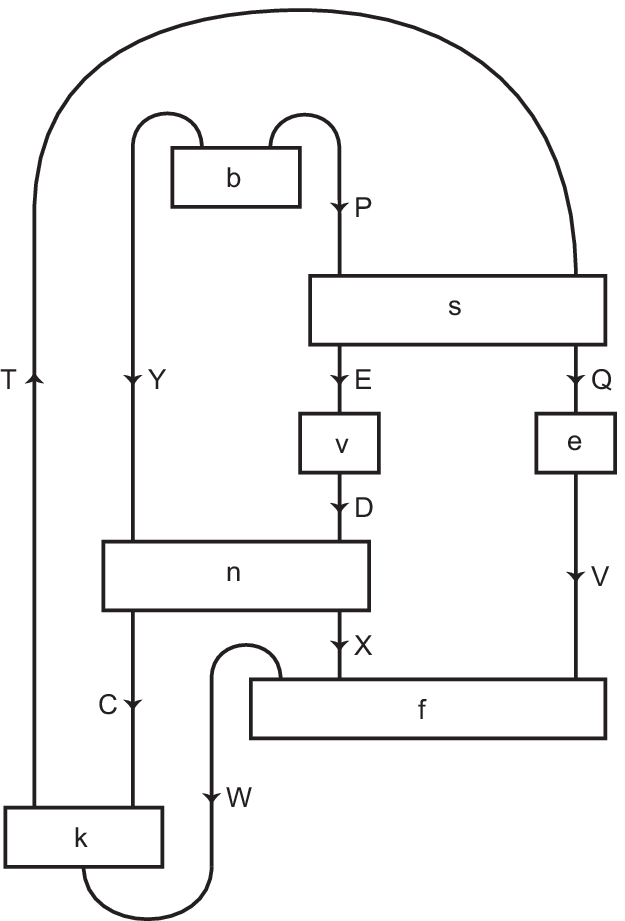}
$$
where
$\tilde{\alpha}$ and $\tilde{\beta}$ are the ismages of $\alpha$ and $\beta$ under the cone isomorphisms
$$
H_u(\Gamma)  \stackrel{\simeq} {\longrightarrow} \Hom_\cc(\un,X^* \otimes C^*) \quad \text{and} \quad
H_v(\Gamma)  \stackrel{\simeq} {\longrightarrow}\Hom_\cc(\un,E \otimes X \otimes D).
$$

\subsection{An invariant of knotted $\cc$-nets in 2-spheres}\label{sect-nets-in-sphere}
Suppose   that  the category~$\cc$ is spherical. Then the invariant    $\inv_\cc$  of knotted $\cc$-nets in $\R^2$ defined in Section~\ref{sect-inv-coloredknottednets}  extends uniquely to an isotopy invariant of knotted $\cc$-nets in  the 2-sphere
 $S^2=\R^2 \cup \{\infty\}$ endowed with   the  orientation   extending the
counterclockwise orientation in $\R^2$.
Indeed,  consider  a knotted $\cc$-net $\Gamma$ in~$S^2$. Pushing $\Gamma$ away from $ \infty $ by an isotopy, we obtain  a $\cc$-colored   graph $\Gamma_0$ in $
\R^2$. The isotopy induces  a \kt linear isomorphism $H (\Gamma;S^2)\simeq H  (\Gamma_0;\R^2)$. Composing    with $\inv_\cc(G_0)\co H (\Gamma_0;\R^2)\to \End_\cc(\un)$ we obtain a \kt linear homomorphism
$$
\inv_\cc(\Gamma) \co  H (\Gamma;S^2)\to \End_\cc(\un).
$$
The sphericity of $\cc$ implies that this homomorphism does not depend on the way we push~$\Gamma$ in~$\R^2$ and is an isotopy invariant of $\Gamma$.

The invariant $\inv_\cc$ further extends to knotted $\cc$-nets in   an arbitrary oriented  surface~$\Sigma$ homeomorphic to $S^2$. Namely, given a knotted $\cc$-net $\Gamma$ in $\Sigma$, pick an orientation preserving homeomorphism $f \co \Sigma \to S^2$, consider the induced \kt linear isomorphism   $H(f) \co H(\Gamma; \Sigma) \to H(f(\Gamma);S^2)$, and set
$$
\inv_\cc (\Gamma)=\inv_\cc (f(\Gamma))\circ H(f) \co  H(\Gamma ;\Sigma)\to  \End_\cc(\un).
$$
Since all orientation preserving homeomorphisms $\Sigma \to   S^2$  are isotopic, this homomorphism  does not depend on the choice of~$f$ and is an isotopy invariant of~$\Gamma$.

We view $\inv_\cc(\Gamma)$ as a generalization of the familiar  $6j$-symbols  which arise when $\Gamma \subset S^2$ is the 1-skeleton of a tetrahedron.

\section{The state sum graph HQFT}\label{sect-construct-stste-sum-HQFT}
In this section we construct state sum graph HQFTs over the $G$-centers of spherical $G$-fusion categories. This construction is based on a presentation of ribbon graphs in 3-manifolds by knotted plexuses in skeletons introduced in  \cite[Chapter 14]{TVi5} where we refer for details.

\subsection{Knotted plexuses in skeletons} \label{sect-knotted-plexuses-in-skel}
Let $M$ be a closed oriented 3-manifold. A \emph{skeleton} of~$M$   is an oriented  compact  2-dimensional  complex $P\subset M$ whose complement in~$M$ is a disjoint union of open 3-balls  called  the    \emph{$P$-balls}. A \emph{vertex} of~$P$ is a point of the 0-skeleton $P^{(0)}$ of~$P$. We let $P^{(1)}$ be the 1-skeleton of~$P$. Here~$P$ oriented means that the surface $\Int(P)=P \backslash P^{(1)}$ is oriented. The set $P^{(1)} \setminus P^{(0)}$ is a finite disjoint union of   open intervals whose closures   are     called   the \emph{edges} of~$P$.

By a \emph{plexus}, we mean a topological space    obtained from a   disjoint union of a finite number  of arcs
and coupons by gluing  the   endpoints of the arcs to the bases of the coupons. We require  that different endpoints of  the arcs are never glued to the same  point of a (base of a) coupon.

A \emph{knotted plexus}~$d$ in~$P$ is  a  plexus drawn (i.e., immersed) in $P \backslash P^{(0)}$    possibly with double  crossings of arcs in $\Int(P)$  so that  at every  crossing, one of   two arcs is  distinguished. All coupons of~$d$ must lie in   $\Int(P) $ while the arcs of~$d$  may meet the edges of $P$ transversely at a finite number of points  called the  \emph{switches} of~$d$. A neighborhood in~$P$ of a switch of $d$  is formed by a finite number (greater than or equal to~$2$) of half-planes adjacent to the edge of $P$ containing the switch so that the plexus~$d$ meets these  half-planes   along a segment  contained in the union of two   of them:
  \begin{center}
  \psfrag{d}[Bc][Bc]{\scalebox{1}{\textcolor{red}{$d$}}}
  \rsdraw{.45}{.9}{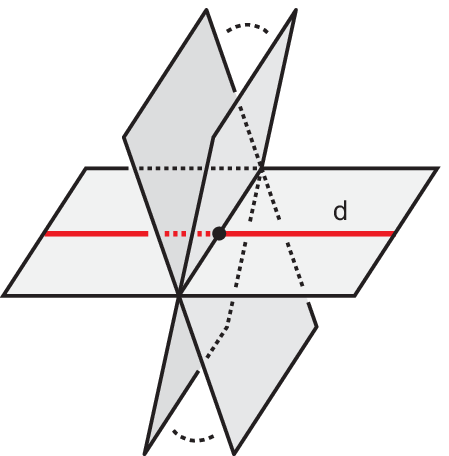} \;.
  \end{center}
We assume that the given orientations of the regions of~$P$  containing these two half-planes are compatible near the switch,  i.e., they are   induced by an orientation of the horizontal plane in the above figure.
(The pair $(P,d)$ is called a positive diagram without circle components in \cite{TVi5}.)

Every knotted plexus~$d$ in~$P$ determines a ribbon graph in~$M$ as follows. Pick a field of normal directions~$n$ on
$\Int(P)$ such that the orientation of $\Int(P)$ followed by~$n$ yields the opposite
orientation of~$M$. Slightly pushing   the undistinguished
strands  at the crossings of~$d$ along~$n$,  we obtain an embedding $ d\hookrightarrow M$ whose image is denoted   $d^n$.  The above orientation condition at the switches implies that
the vector field~$n$ can be chosen to continuously extend to all switches and to determine thus a framing of~$d^n$.
This turns $d^n$ into a ribbon graph  in~$M$. We say that a ribbon graph   in $M$ is   \emph{represented} by $d$ if it is isotopic to~$d^n$.

By \cite[Corollary 14.5]{TVi5}, all ribbon graphs   in $M$ can be represented by knotted plexus in skeletons of $M$, and two knotted plexus in skeletons of $M$ represent  isotopic ribbon graphs if they can be related by certain moves generalizing the Reidemeiter moves to skeletons.

\subsection{An invariant of colored $G$-graphs} \label{sect-the-inv-for-MR}
In the rest  of this section, $\cc$ is a spherical $G$-fusion category over a field  $\kk$ such that $\dim(\cc_1) \neq 0$. Recall from Section~\ref{sect-center-graded-cat} that the $G$-center $\zz_G(\cc)$ of~$\cc$ is then a $G$-ribbon category\footnote{More generally, the constructions of Section~\ref{sect-construct-stste-sum-HQFT} work when $\cc$ is a spherical $G$-fusion category over a commutative ring such that $\dim(\cc_1)$ is invertible and all idempotents in $\cc$ split (see Appendix~\ref{sect-nonsing-graded-cat}). Indeed, in this case,  $\zz_G(\cc)$ is also a $G$-ribbon category.}.  We  define here   a state sum invariant
$|M, \Omega|_\cc  \in \kk$ for any $\zz_G(\cc)$-colored $G$-graph~$\Omega$ in   a closed oriented 3-manifold~$M$.

Let $P$ be a skeleton of $M$ and $d$ be a knotted plexus in  $P$ representing the underlying ribbon graph of~$\Omega$ (see Section~\ref{sect-knotted-plexuses-in-skel}). The vertices   of~$P$ and the   switches, crossings, and   coupons of~$d$ are   called    the \emph{nodes} of~$d$. The complement of the nodes in~$\tilde{d}= d \cup P^{(1)} \subset P$ is a finite disjoint union of   open intervals. Their  closures   are     called   the \emph{rims} of~$d $. Each rim~$e$ lies  in~$P^{(1)}$ or  in~$d$ and connects two nodes  (possibly, equal) called the \emph{endpoints} of~$e$.  By an \emph{oriented rim}  we mean a  rim   endowed with  orientation (which may be compatible or not with  the orientation of the strand of~$d$ containing this rim).
Cutting~$P$   along $\tilde{d}$,   we obtain a compact surface (with interior $P \setminus \tilde{d}$)
whose  components  are called the \emph{faces} of~$d$. We let $\Fac(d)$ be the (finite) set of faces of~$d$.

We  define a map  $  \ell \co \Fac(d) \to G$, called the {\it $G$-labeling} of~$d$.  In every $P$-ball, pick a point, called its \emph{center}.  Pick  in the given homotopy class of maps  $  (M \setminus \Omega, \Omega_\bullet) \to (\X,\x)$ a representative $g \co M \setminus \Omega \to \X$ carrying the centers of all $P$-balls to~$\x$.
Clearly, each face~$r$ of~$d$ is adjacent to two (possibly coinciding) $P$-balls. Pick an oriented arc  in $M$
connecting the centers of these balls and meeting~$P$  transversely in a single point lying in the interior of $r$. We  orient this arc so  that its intersection number with the face~$r$ is equal to $+1$.  Applying~$g$ to the resulting oriented arc we get  a loop in $\X$ representing $\ell(r)\in G=\pi_1(\X,\x)$.

We now use the given $\zz_G(\cc)$-coloring of the $G$-graph $\Omega$ to color the coupons and the crossings of~$d$, as well  as the rims of~$d$ lying in~$d$. To this end, pick a 1-system of tracks $\{\gamma_{\lambda}\}_{\lambda}$ and a 2-system of tracks $\{\gamma_c\}_{c}$ for $\Omega$, where~$\lambda$ runs over the arcs and~$c$ runs over the coupons of~$\Omega$.
First,  consider a rim $e \subset d$.  Slightly pushing~$e$ along the framing of $\Omega$,  we obtain a  parallel copy $\widetilde e$  of~$e$ lying  inside   a   $P$-ball. Let~$\delta_e$ be a path in this $P$-ball leading from its  center  to $\widetilde e$. Let $\lambda=\lambda_e$ be   the arc of $\Omega$ supporting~$e$.  Then, up to   homotopy    in $M\setminus \Omega$, we have $\delta_e=\alpha_e \gamma_{\lambda}$ for a unique homotopy class of paths $\alpha_e$  in $M \setminus \Omega$ leading  from the center of  the $P$-ball above  to the point $\gamma_\lambda(0)\in \Omega_\bullet$.  The map~$g$  carries  $\alpha_e$ to an  element of $G=\pi_1(\X ,\x)$ again denoted~$\alpha_e$.  The given precoloring~$u$ of~$\Omega$ yields an object $u_{\gamma_{\lambda}} \in \zz_G(\cc)$. The \emph{color} of~$e$ is the object
$$
U_e=\varphi_{\alpha_e^{-1}}(u_{\gamma_{\lambda}})\in \zz_G(\cc)
$$
which is homogeneous of degree $\alpha_e\mu_{\gamma_{\lambda}}\alpha_e^{-1} \in G$. Next, consider a crossing~$x$ of~$d$. Let $e^+,e^-$ be  the rims adjacent to~$x$ and  lying in the distinguished strand of~$x$.
We choose notation so that  $e^+$ is directed towards $x$ if the crossing $x$ is   positive  and away from $x$ if $x$ is   negative:
$$
\psfrag{x}[Bl][Bl]{\scalebox{.9}{$x$}}
\psfrag{c}[Br][Br]{\scalebox{.9}{$e^-$}}
\psfrag{r}[Bl][Bl]{\scalebox{.9}{$e^+$}}
\rsdraw{.45}{.9}{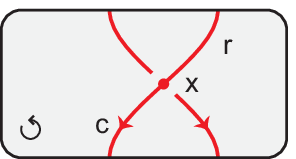}
\qquad \quad
\psfrag{o}[Bl][Bl]{\scalebox{.9}{$e^+$}}
\psfrag{s}[Br][Br]{\scalebox{.9}{$e^-$}}
\rsdraw{.45}{.9}{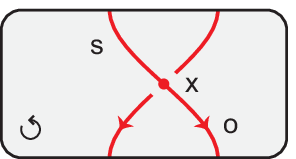} \;\,.
$$
The \emph{color} of~$x$ is the isomorphism
$$
U_x=\bigl(\varphi_2(\alpha_{e^-}\alpha_{e^+}^{-1}, \alpha^{-1}_{e^-})_{u_{\gamma_\lambda}}\bigr)^{-1} \co
U_{e^+} \to \varphi_{\alpha_{e^-}\alpha_{e^+}^{-1}}(U_{e^-}),
$$
where $\lambda$  is the arc of~$\Omega$ supporting $e^+ \cup e^-$.
Finally,  consider a coupon~$c$ of~$d$. Pushing~$c$ along the framing of $\Omega$,  we obtain a parallel copy $\widetilde c$  of~$c$  lying  inside   a   $P$-ball $B_c$.  Let $\delta_c$ be a path in $B_c$ leading from its  center  to $\widetilde c$.
Then  up to   homotopy    in $M\setminus \Omega$, we have $\delta_c=\alpha_c \gamma_{c}$ for a unique    homotopy class of paths $\alpha_c$  in $M \setminus \Omega$ leading  from the center of $B_c$ to the point $\gamma_c(0)\in \Omega_\bullet$.
Let $m \geq 0$   be the number of inputs  of~$c$  and let $e_k$ be the rim of $d$ incident to the $k$-th input for $k=1, \dots, m$. Set   $\varepsilon_k=+$    if $e_k$ is directed out of~$c$ at the $k$-th input and  $\varepsilon_k=-$ otherwise. Let~$\lambda_k$ be   the arc of $\Omega$ supporting~$e_k$ and set $\gamma_k=\gamma_{\lambda_k}$. Composing $\delta_c$ with a   path in~$\widetilde  c$ leading   to the $k$-th input, we obtain a path $\delta_k$  in $B_c$ which expands uniquely as $\delta_k=\alpha_k \gamma_k$  for a unique homotopy class of paths $\alpha_k$  in $M \setminus \Omega$ leading  from the center of  $B_c$ to the point $\gamma_k(0)=\gamma_c(0)\in \Omega_\bullet$. By definition, the color of the rim $e_k$ is
$$
U_{e_k}=\varphi_{\alpha_k^{-1}}(u_{\gamma_k})\in \zz_G(\cc).
$$
Similarly, let $n\geq 0$   be the number of outputs  of~$c$ and let $e^l$ be the arc  of $\Omega$ incident to the $l$-th output for $l=1, \dots, n$. Set $\varepsilon^l =+$ if $e^l$ is directed into~$c$ at the $l$-th output and $\varepsilon^l =-$   otherwise.
Let $\lambda^l$ be   the arc of $\Omega$ supporting~$e^l$ and set $\gamma^l=\gamma_{\lambda^l}$.  Composing $\delta_c$ with a   path in~$\widetilde  c$ leading   to the $l$-th output, we obtain a path $\delta^l$  in~$B_c$ which expands uniquely as $\delta^l=\alpha^l \gamma^l$  for a unique homotopy class of paths $\alpha^l$  in $M \setminus \Omega$ leading  from the center of  $B_c$ to the point $\gamma^l(0)=\gamma_c(0)\in \Omega_\bullet$.
By definition, the color of the rim $e^l$ is
$$
U_{e^l}=\varphi_{(\alpha^l)^{-1}}(u_{\gamma^l})\in \zz_G(\cc).
$$
The given coupon-coloring~$v$ of~$\Omega$  yields  the morphism
$$
v_{\gamma_c} \colon u_{\gamma_c}=\bigotimes_{k=1}^m
u_{\rho_k}^{\varepsilon_k} \longrightarrow u^{\gamma_c}=\bigotimes_{\ell=1}^n
u_{\rho^\ell}^{\varepsilon^\ell}
$$
where $\rho_k$ is the composition of $\gamma_c$ with a   path in~$\widetilde  c$ leading   to the $k$-th input and
$\rho^l$ is the composition of $\gamma_c$ with a   path in~$\widetilde  c$ leading   to the $l$-th output. Note that
$$
\rho_k=\alpha_c^{-1} \delta_k=\alpha_c^{-1}\alpha_k \gamma_k
\quad\text{and}\quad
\rho_l=\alpha_c^{-1} \delta^l=\alpha_c^{-1}\alpha^l \gamma^l.
$$
For a track $\gamma$ of an arc incident to $c$ and a  homotopy class of paths $\alpha$  in $M \setminus \Omega$ leading  from the center of  $B_c$ to the point $\gamma(0)=\gamma_c(0)\in \Omega_\bullet$, the precoloring $u$ yields an isomorphism
$
u_{\alpha_c^{-1}\alpha, \gamma}
\co u_{\alpha_c^{-1}\alpha \gamma} \to \varphi_{\alpha^{-1} \alpha_c}(u_\gamma)$. Consider the isomorphisms
$$
\psi(\alpha,\gamma,+)=\varphi_2(\alpha_c^{-1}, \alpha^{-1}\alpha_c)_{u_\gamma} \circ \varphi_{\alpha_c^{-1}}(u_{\alpha_c^{-1}\alpha, \gamma})
\co \varphi_{\alpha_c^{-1}}(u_{\alpha_c^{-1}\alpha \gamma}) \to \varphi_{\alpha^{-1}}(u_\gamma).
$$
and
$$
\psi(\alpha,\gamma,-)=
\bigl(\psi(\alpha,\gamma,+)^{-1}\bigr)^*\circ \varphi^1_{\alpha_c^{-1}}(u_{\alpha_c^{-1}\alpha \gamma})
\co \varphi_{\alpha_c^{-1}}(u_{\alpha_c^{-1}\alpha \gamma}^*) \to \varphi_{\alpha^{-1}}(u_\gamma)^*.
$$
The \emph{color} of the coupon $c$ is the morphism
$$
U_c \colon \bigotimes_{k=1}^m
U_{e_k}^{\varepsilon_k}\longrightarrow
\bigotimes_{\ell=1}^n
U_{e^l}^{\varepsilon^\ell}
$$
in $\zz_G(\cc)$ defined as the following composition:
$$
\xymatrix@R=1cm @C=2.6cm {
\bigotimes\limits_{k=1}^m
U_{e_k}^{\varepsilon_k} \ar@{.>}[d]_-{U_c} \ar[r]^-{\bigotimes\limits_{k=1}^m
\psi(\alpha_k,\gamma_k,\varepsilon_k)^{-1}} &
\bigotimes\limits_{k=1}^m \varphi_{ {\alpha_c^{-1}} }  (u_{\rho_k}^{\varepsilon_k})    \ar[r]^-{(\varphi_{ {\alpha_c^{-1}} })_m}
 &\varphi_{ \alpha_c^{-1} }(u_{\gamma_c})  \ar[d]^-{\varphi_{\alpha_c^{-1}} (v_{\gamma_c})}\\
\bigotimes\limits_{\ell=1}^n
U_{e^l}^{\varepsilon^\ell}
 &  \bigotimes\limits_{l=1}^n  \varphi_{{\alpha_c^{-1}}} (u_{\rho^l}^{\varepsilon^l})
 \ar[l]_-{\bigotimes\limits_{l=1}^n   \psi(\alpha^l,\gamma^l,\varepsilon^l)}
& \ar[l]_-{(\varphi_{{\alpha_c^{-1}}})_n^{-1}} \varphi_{ \alpha_c^{-1} } (u^{\gamma_c}).}
$$

Fix from now on  a representative set $I=\amalg_{\alpha\in G}\, I_\alpha$
of simple objects of~$\cc$. A \emph{$G$-coloring} of  $d$  is a map $\cl \co \Fac(d) \to I$ such that $\cl (r)\in
I_{\ell (r)}$ for all faces~$r$ of~$d$. The object $\cl(r)$ assigned to a face~$r$ of~$d$ is called   the \emph{$\cl$-color} of~$r$.
For each   $G$-coloring~$\cl$ of $d$, we   define a scalar $\vert \cl  \vert \in \kk$ as follows.

First, for each  oriented rim~$e$ of~$ {d}$,  we define  a   cyclic $\cc$\ti set $P_{e, \cl}$. If $e\subset P^{(1)}$, then
$P_{e, \cl}$ is the set of germs  of  faces of $d$  adjacent to~$e$ turned into a cyclic $\cc$\ti set  as in
\cite[Section~13.1.1]{TVi5} using~$\cl$. Explicitly, the orientations of~$e$ and~$M$ determine a positive
direction on a small loop in~$M$ encircling~$e$. The  resulting oriented loop
determines  a cyclic order on the set $P_{e, \cl}$  of germs  of faces of $d$ adjacent to~$e$.
To each $b\in P_{e, \cl}$ we assign the $\cl$-color of the face of~$d$ containing $b$ and a sign equal to~$+$ if the orientation  of~$b$  induces the one of $e\subset \partial b$ (that is, the orientation of~$b$ is given by the orientation of~$e$ followed by a vector at a point of~$e$ directed inside~$b$)    and equal  to $ -$ otherwise.
In this way, $P_{e, \cl}$ becomes a cyclic $\cc$\ti set.
If $e\subset d$, then~$e$ is adjacent to two faces    $r_+, r_-$ of~$d$
such that the orientation of~$r_+$ (induced by the one of~$P$) induces  the given orientation of~$e$ and the orientation of  $r_-$ induces the  opposite orientation of~$e$:
$$
 \psfrag{M}[Bl][Bl]{\scalebox{.9}{$X$}}
 \psfrag{r}[Bc][Bc]{\scalebox{.9}{$  r_+  $}}
 \psfrag{v}[Bc][Bc]{\scalebox{.9}{$ r_-  $}}
 \psfrag{l}[Bc][Bc]{\scalebox{.9}{\textcolor{red}{$e$}}}
 \rsdraw{.45}{.9}{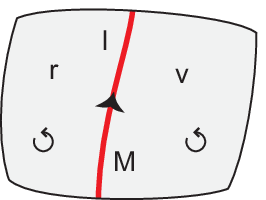}\, .
$$
In this picture the   arrow on~$e$   indicates the   orientation  of~$e$.  Set $\varepsilon=+$ if the   orientation of the strand  of $d$ containing $e$ is compatible with that of~$e$ and  set $\varepsilon=-$ otherwise. Then $P_{e, \cl}=\{r_-  , e , r_+ \}$ where $r_-  < e < r_+ <  r_- $ and   the map   $\{r_-  , e , r_+ \} \to   \Ob(\cc) \times \{+,-\}$ carries
$r_\pm  $ to $(\cl(r_\pm ), \pm)$ and carries~$e$ to  $(X,\varepsilon)$ where $X\in \cc$ is the image of the color of~$e$ under the forgetful functor  $ \zz_G(\cc) \to \cc$.

For any oriented rim~$e$ of~$d$,  let
$
H_\cl (e)=H(P_{e, \cl})
$
be the  multiplicity   module of~$  {P_{e, \cl}}$.
In particular, if $e\subset d$, then in the notation above
$$
H_\cl (e)\simeq \Hom_\cc(\un,c(r_- )^*\otimes {{X}}^\varepsilon \otimes c(r_+ ) ).
$$
We let $$H_\cl  =\bigotimes_e \, H_\cl(e)$$
be the unordered tensor product  of the  \kt modules $H_\cl(e  )$   over all    oriented rims~$e$
of~$d$.
An     unoriented  rim~$E$  of~$ {d}$   gives rise to two  opposite  oriented rims $e_1,e_2  $ whose associated cyclic $\cc$-sets $P_{e_1, \cl}$ and $P_{e_2, \cl}$ are dual to each other. This determines a contraction vector $\ast_E\in H_\cl(e_1) \otimes H_\cl(e_2)$, see \cite[Section~12.3.4]{TVi5}.   Set
$$
\ast_\cl=\bigotimes_E  \, \ast_E \in H_\cl
$$
where  $E$ runs over all unoriented rims of~$d$.

We associate with    each   node $x$ of ${ d}$ a   knotted $\cc$-net $\Gamma_x^\cl$ in
an oriented surface homeomorphic to $S^2$.   We do it as in \cite[Section~15.5.1]{TVi5}
by appropriately incorporating colors to crossings of the involved knotted $\cc$-nets.
More precisely, we consider four cases, as follows.
\begin{enumerate}
\labeli
\item Let~$x$ be a vertex of $P$.  Pick a small closed ball
neighborhood $B_x \subset M$ of~$x$ such that $\Gamma_x=P\cap \partial B_x$
is a   non-empty graph and $ P\cap B_x $ is the cone over~$  \Gamma_x $ with summit~$x$. The vertices of $\Gamma_x$ are the intersection   points of the 2-sphere $\partial B_x$ with the edges   of~$P$ incident to~$x$. The edges   of $\Gamma_x$ are the intersections of  $\partial B_x$ with the faces of~$d$ adjacent to~$x$. We endow $\partial B_x$ with the orientation   induced by that of~$M$ restricted to $M\setminus \Int(B_x)$. Every edge~$a$ of $\Gamma_x $ lies   in a face~$r_a$ of~$d$. We color~$a$ with   $\cl(r_a) \in I$ and endow~$a$ with the orientation induced by that of $r_a \setminus \Int (B_x)$.   In this way, $\Gamma_x$   yields a $\cc$-colored graph  in  $\partial B_x$  denoted $\Gamma_x^\cl$. For example:
\begin{center}
\psfrag{i}[Bc][Bc]{\scalebox{.9}{$i$}}
\psfrag{j}[Bc][Bc]{\scalebox{.9}{$j$}}
\psfrag{k}[Bc][Bc]{\scalebox{.9}{$k$}}
\psfrag{l}[Bc][Bc]{\scalebox{.9}{$l$}}
\psfrag{m}[Bc][Bc]{\scalebox{.9}{$m$}}
\psfrag{n}[Bc][Bc]{\scalebox{.9}{$n$}}
\psfrag{q}[Bc][Bc]{\scalebox{.9}{$q$}}
\psfrag{p}[Bc][Bc]{\scalebox{.9}{$p$}}
\psfrag{s}[Bc][Bc]{\scalebox{.9}{$s$}}
\psfrag{t}[Bc][Bc]{\scalebox{.9}{$t$}}
\psfrag{x}[Bc][Bc]{\scalebox{1}{$x$}}
\psfrag{1}[Bc][Bc]{\scalebox{.7}{$1$}}
\psfrag{2}[Bc][Bc]{\scalebox{.7}{$2$}}
\psfrag{3}[Bc][Bc]{\scalebox{.7}{$3$}}
\psfrag{M}[Br][Br]{\scalebox{.9}{$M$}}
\psfrag{S}[Bl][Bl]{\scalebox{.9}{$S^2$}}
\psfrag{G}[Br][Br]{\scalebox{1.111111}{$\Gamma_x^\cl=$}}
\psfrag{H}[Bl][Bl]{\scalebox{1.111111}{$ \subset \partial B_x \simeq S^2$.}}
\rsdraw{.45}{.9}{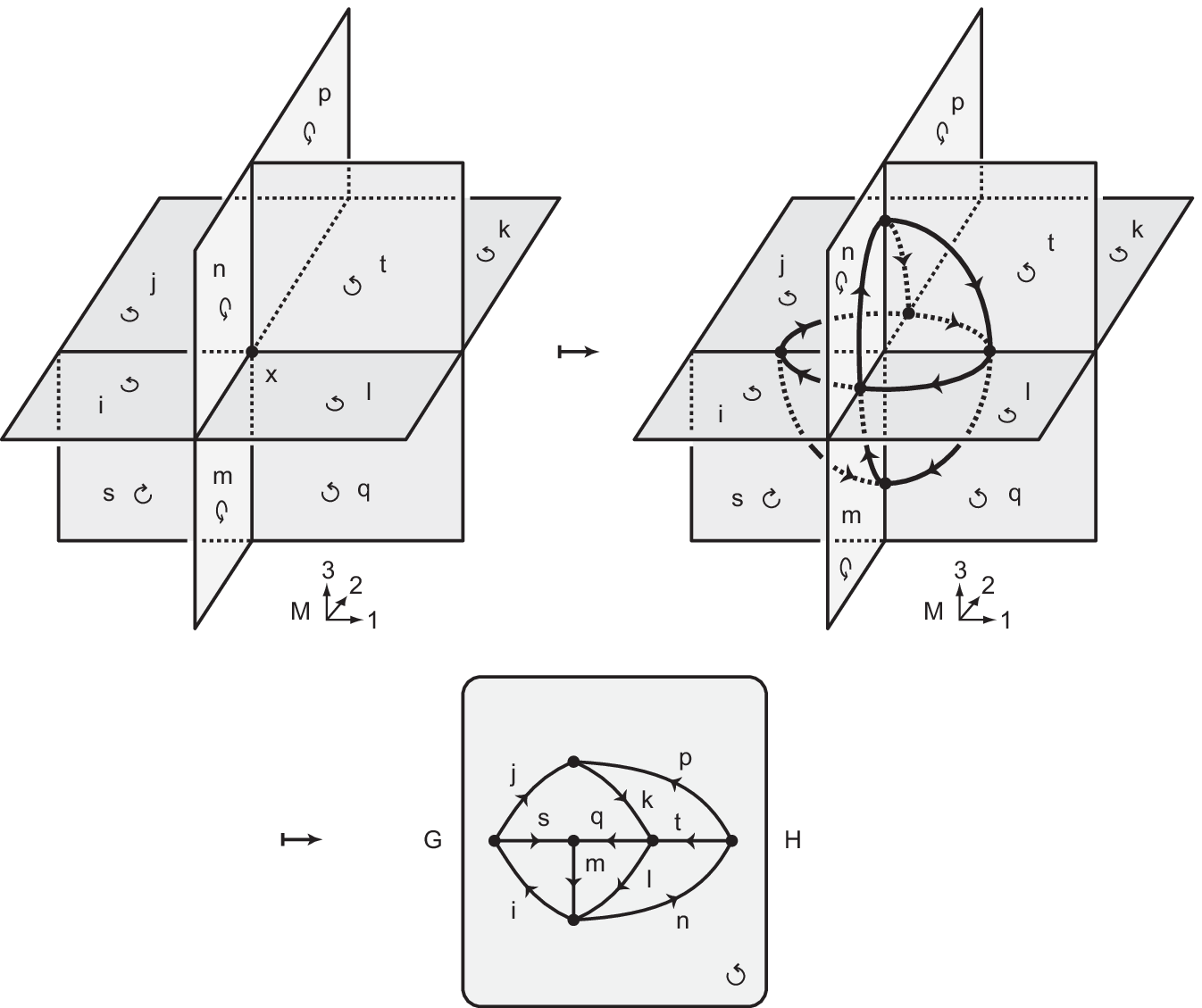}
\end{center}

\item Let~$x$ be a switch of~$d$. A neighborhood of  $x$   in $P$    looks as follows:
\begin{equation*}
 \psfrag{i}[Bc][Bc]{\scalebox{.9}{$j$}}
 \psfrag{j}[Bc][Bc]{\scalebox{.9}{$i$}}
 \psfrag{k}[Bc][Bc]{\scalebox{.9}{$k$}}
 \psfrag{1}[Bc][Bc]{\scalebox{.7}{$1$}}
\psfrag{2}[Bc][Bc]{\scalebox{.7}{$2$}}
\psfrag{3}[Bc][Bc]{\scalebox{.7}{$3$}}
 \psfrag{M}[Br][Br]{\scalebox{.9}{$M$}}
 \psfrag{N}[Bc][Bc]{\scalebox{.9}{\textcolor{red}{$A$}}}
 \psfrag{V}[Bc][Bc]{\scalebox{.9}{\textcolor{red}{$B$}}}
 \psfrag{l}[Bc][Bc]{\scalebox{.9}{$l$}}
 \psfrag{x}[Bc][Bc]{\scalebox{.9}{$x$}}
 \psfrag{m}[Bc][Bc]{\scalebox{.9}{$m_a$}}
 \psfrag{u}[Bc][Bc]{\scalebox{.9}{$m_1$}}
 \psfrag{c}[Bc][Bc]{\scalebox{.9}{$n_b$}}
 \psfrag{e}[Bc][Bc]{\scalebox{.9}{$n_1$}}
 \rsdraw{.476}{.9}{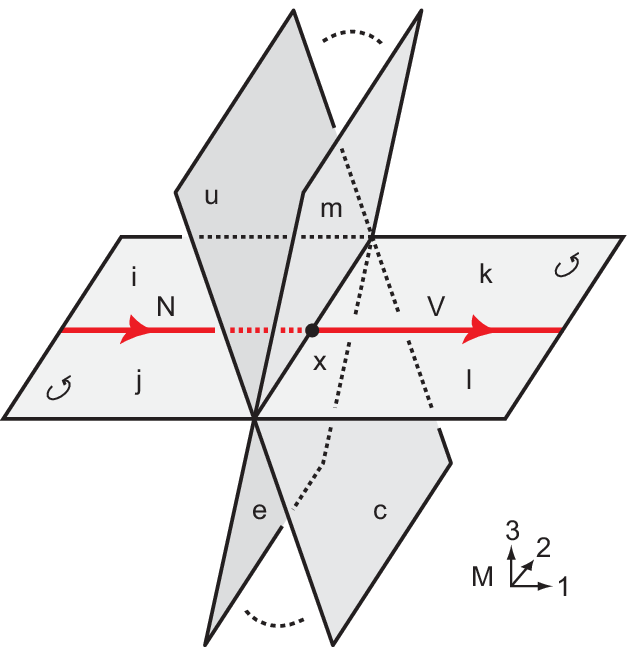}  .
\end{equation*}
Here the orientation of~$M$ is  right-handed (as always in our pictures),   the object  $A=U_e\in \zz_G(\cc)$ is the color of the rim $e$ of $d$ directed to $x$, the object $B=U_o\in \zz_G(\cc)$ is the color of the rim~$o$ of~$d$ directed away from~$x$, and
$i,j,k,l,m_1, \dots , m_a, n_1, \dots , n_b \in I$ are the $\cl$-colors of the faces of
$d$ adjacent to $x$ with $a \geq 0$ and $b \geq 0$. By definition,   $A=\varphi_{\alpha_{e}^{-1}}(X)$ and   $B=\varphi_{\alpha_{o}^{-1}}(X)$, where $X=u_{\gamma_\lambda}\in \zz_G(\cc)$ is the evaluation of the precoloring~$u$ of~$\Omega$ on  the  track $\gamma_\lambda$ of the arc~$\lambda$   of~$\Omega$ supporting the rims~$e$ and~$o$.
Set
\begin{equation*}
 \psfrag{i}[Br][Br]{\scalebox{.9}{$i$}}
 \psfrag{j}[Br][Br]{\scalebox{.9}{$j$}}
 \psfrag{k}[Bl][Bl]{\scalebox{.9}{$k$}}
 \psfrag{l}[Bl][Bl]{\scalebox{.9}{$l$}}
 \psfrag{M}[Bc][Bc]{\scalebox{.9}{\textcolor{red}{$A$}}}
 \psfrag{B}[Bc][Bc]{\scalebox{.9}{\textcolor{red}{$B$}}}
 \psfrag{p}[Br][Br]{\scalebox{.9}{\textcolor{red}{$\psi_1$}}}
 \psfrag{q}[Bl][Bl]{\scalebox{.9}{\textcolor{red}{$\psi_b$}}}
 \psfrag{m}[Bl][Bl]{\scalebox{.9}{$m_a$}}
 \psfrag{n}[Bl][Bl]{\scalebox{.9}{$m_1$}}
 \psfrag{a}[Br][Br]{\scalebox{.9}{$n_1$}}
 \psfrag{c}[Bl][Bl]{\scalebox{.9}{$n_b$}}
 \Gamma_x^\cl=\;\rsdraw{.47}{.9}{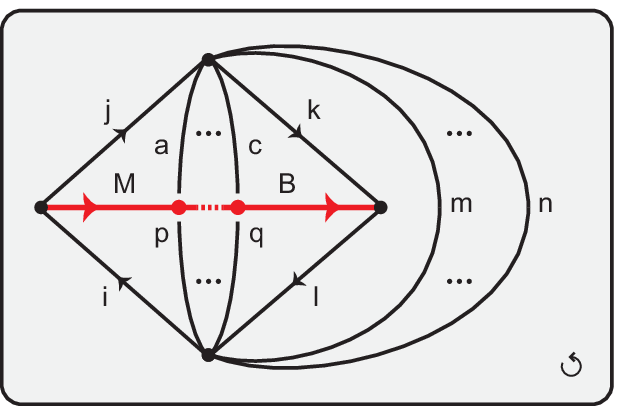}  \; .
\end{equation*}
Here  we draw   $\Gamma_x^\cl$ in the plane $\RR^2$ (oriented counterclockwise) and view $\Gamma_x^\cl$ as a knotted $\cc$-net in $S^2=\RR^2 \cup \{\infty\}$.
We direct the arc  of $\Gamma_x^\cl$ colored by $m_p$ with $1 \leq p \leq a$    upward if the orientation of  the  $m_p$-labeled face of~$d$  followed by  that of~$d$ at~$x$ yields the positive orientation of $M$, and  downward otherwise.
For $1 \leq q \leq b$, let $r_q$ be the $n_q$-labeled face  of~$d$. Set $\varepsilon_q=1$ if the orientation of $r_q$  followed by  that of~$d$ at~$x$ yields the positive orientation of~$M$ and set $\varepsilon_q=-1$ otherwise. We direct the  arc of $\Gamma_x^\cl$ colored by $n_q$   downward if $\varepsilon_q=1$ and upward otherwise.
We numerate the  $b+1$ arcs lying on the  horizontal  segment from left to right by  $q=0,1,...,b$ and  color the $q$-th arc  with
$$
A_q= \varphi_{\beta_q}(X)\in \zz_G(\cc)  \quad\text{where}\quad
\beta_q=\alpha_e^{-1} \ell(r_1)^{\varepsilon_1}\cdots \ell (r_q)^{\varepsilon_q}  \in G.
$$
Note that  $A_0=A$ (since $\beta_0=\alpha_e^{-1}$) and $A_b=B$ (since $\beta_b=\alpha_o^{-1}$).
The knotted net $\Gamma_x^\cl$ has~$b$ crossing points. We  numerate  them  from left to right by  $q=1,\dots,b$ and
color the $q$-th crossing  with  the isomorphism
$$
\psi_q=
\begin{cases}
  \Bigl(\varphi_2\bigl(\ell(r_q),\beta_{q-1}\bigr)_{X}\Bigr)^{-1} \co A_q \to \varphi_{\ell(r_q)}(A_{q-1})
  & \text{if $\varepsilon_q=1$,} \\
      \Bigl(\varphi_2\bigl(\ell (r_q),\beta_{q}\bigr)_{X}\Bigr)^{-1} \co A_{q-1} \to \varphi_{\ell(r_q)}(A_q)
  & \text{if $\varepsilon_q=-1$.}
\end{cases}
$$

\item Let  $x$ be a coupon of $d$  with $m\geq 0$ inputs and $n\geq 0$ outputs. A
neighborhood of~$x$ in~$ P$ looks as follows:
\begin{equation*}
 \psfrag{A}[Bl][Bl]{\scalebox{.9}{\textcolor{red}{$A_1$}}}
 \psfrag{B}[Bl][Bl]{\scalebox{.9}{\textcolor{red}{$A_2$}}}
 \psfrag{C}[Bl][Bl]{\scalebox{.9}{\textcolor{red}{$A_{m-1}$}}}
 \psfrag{D}[Bl][Bl]{\scalebox{.9}{\textcolor{red}{$A_m$}}}
 \psfrag{X}[Bl][Bl]{\scalebox{.9}{\textcolor{red}{$B_1$}}}
 \psfrag{Y}[Bl][Bl]{\scalebox{.9}{\textcolor{red}{$B_2$}}}
 \psfrag{Z}[Bl][Bl]{\scalebox{.9}{\textcolor{red}{$B_{n-1}$}}}
 \psfrag{E}[Bl][Bl]{\scalebox{.9}{\textcolor{red}{$B_n$}}}
 \psfrag{a}[Bc][Bc]{\scalebox{.9}{$k_1$}}
 \psfrag{x}[Bc][Bc]{\scalebox{.9}{$k_{m-1}$}}
 \psfrag{c}[Bc][Bc]{\scalebox{.9}{$l_1$}}
 \psfrag{e}[Bc][Bc]{\scalebox{.9}{$l_{n-1}$}}
 \psfrag{i}[Bc][Bc]{\scalebox{.9}{$i$}}
 \psfrag{j}[Bc][Bc]{\scalebox{.9}{$j$}}
 \psfrag{f}[Bc][Bc]{\scalebox{1.11111111}{\textcolor{red}{$f$}}}
 \rsdraw{.45}{.9}{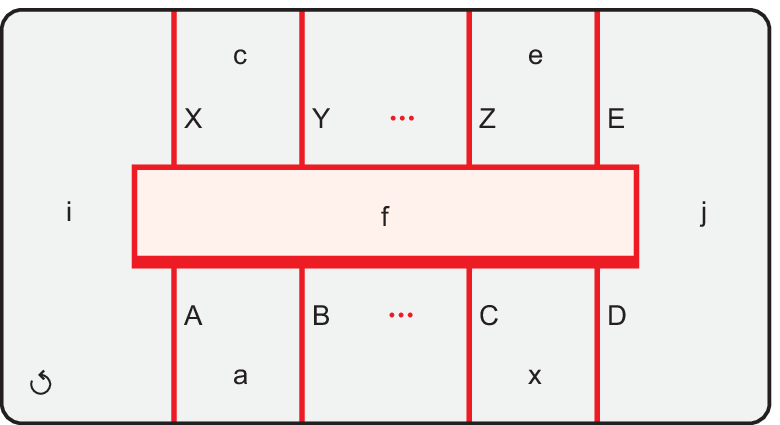} \, .
\end{equation*}
Here $i,j,k_1, \dots, k_{m-1},l_1, \dots, l_{n-1} \in I$ are the $\cl$-colors of the faces of~$d$ adjacent to the coupon,   the objects     $A_1, \dots, A_m\in \zz_G(\cc)$  are the  colors of the rims corresponding to the inputs,    the objects   $B_1,\dots,B_n \in\zz_G(\cc)$  are the  colors of the rims corresponding to the outputs, and the morphism~$f$ in~$\zz_G(\cc)$ is the color of the coupon~$x$.
For $m,n \geq 1$, set
\begin{equation*}
 \psfrag{A}[Bl][Bl]{\scalebox{.9}{\textcolor{red}{$A_1$}}}
 \psfrag{B}[Bl][Bl]{\scalebox{.9}{\textcolor{red}{$A_2$}}}
 \psfrag{C}[Bl][Bl]{\scalebox{.9}{\textcolor{red}{$A_{m-1}$}}}
 \psfrag{D}[Bl][Bl]{\scalebox{.9}{\textcolor{red}{$A_m$}}}
 \psfrag{X}[Bl][Bl]{\scalebox{.9}{\textcolor{red}{$B_1$}}}
 \psfrag{Y}[Bl][Bl]{\scalebox{.9}{\textcolor{red}{$B_2$}}}
 \psfrag{Z}[Bl][Bl]{\scalebox{.9}{\textcolor{red}{$B_{n-1}$}}}
 \psfrag{E}[Bl][Bl]{\scalebox{.9}{\textcolor{red}{$B_n$}}}
 \psfrag{a}[Bc][Bc]{\scalebox{.9}{$k_1$}}
 \psfrag{x}[Bc][Bc]{\scalebox{.9}{$k_{m-1}$}}
 \psfrag{c}[Bc][Bc]{\scalebox{.9}{$l_1$}}
 \psfrag{e}[Bc][Bc]{\scalebox{.9}{$l_{n-1}$}}
 \psfrag{i}[Bc][Bc]{\scalebox{.9}{$i$}}
 \psfrag{j}[Bc][Bc]{\scalebox{.9}{$j$}}
 \psfrag{f}[Bc][Bc]{\scalebox{1.11111111}{\textcolor{red}{$f$}}}
 \psfrag{R}[Br][Br]{\scalebox{1.11111111}{$$}}
 \Gamma_x^\cl=\; \rsdraw{.47}{.9}{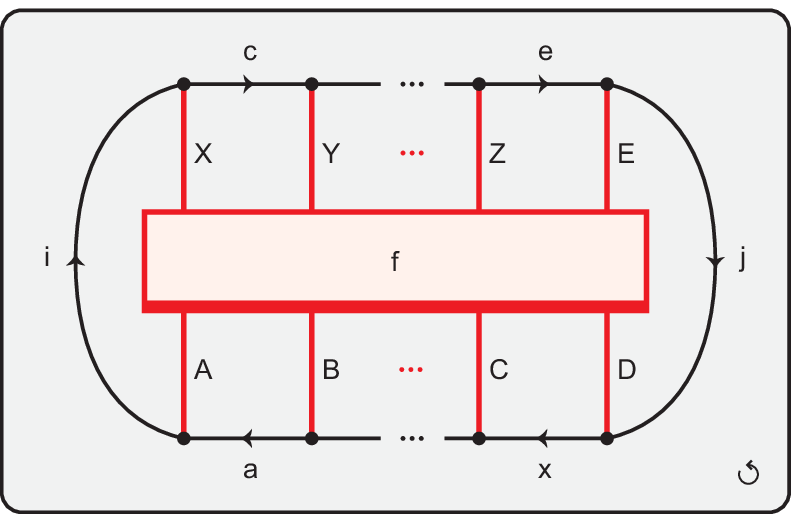}   \, \subset S^2,
\end{equation*}
where the orientations (not shown in the picture) of the vertical arcs  are
induced by the orientations of the corresponding strands of~$d$.   For $m=0$ and $ n\geq 1$, set
\begin{equation*}
 \psfrag{A}[Bl][Bl]{\scalebox{.9}{\textcolor{red}{$A_1$}}}
 \psfrag{B}[Bl][Bl]{\scalebox{.9}{\textcolor{red}{$A_2$}}}
 \psfrag{C}[Bl][Bl]{\scalebox{.9}{\textcolor{red}{$A_{m-1}$}}}
 \psfrag{D}[Bl][Bl]{\scalebox{.9}{\textcolor{red}{$A_m$}}}
 \psfrag{X}[Bl][Bl]{\scalebox{.9}{\textcolor{red}{$B_1$}}}
 \psfrag{Y}[Bl][Bl]{\scalebox{.9}{\textcolor{red}{$B_2$}}}
 \psfrag{Z}[Bl][Bl]{\scalebox{.9}{\textcolor{red}{$B_{n-1}$}}}
 \psfrag{E}[Bl][Bl]{\scalebox{.9}{\textcolor{red}{$B_n$}}}
 \psfrag{a}[Bc][Bc]{\scalebox{.9}{$k_1$}}
 \psfrag{x}[Bc][Bc]{\scalebox{.9}{$k_{m-1}$}}
 \psfrag{c}[Bc][Bc]{\scalebox{.9}{$l_1$}}
 \psfrag{e}[Bc][Bc]{\scalebox{.9}{$l_{n-1}$}}
 \psfrag{i}[Bc][Bc]{\scalebox{.9}{$i$}}
 \psfrag{j}[Bc][Bc]{\scalebox{.9}{$j$}}
 \psfrag{f}[Bc][Bc]{\scalebox{1.11111111}{\textcolor{red}{$f$}}}
 \psfrag{R}[Br][Br]{\scalebox{1.11111111}{$$}}
 \Gamma_x^\cl=\; \rsdraw{.47}{.9}{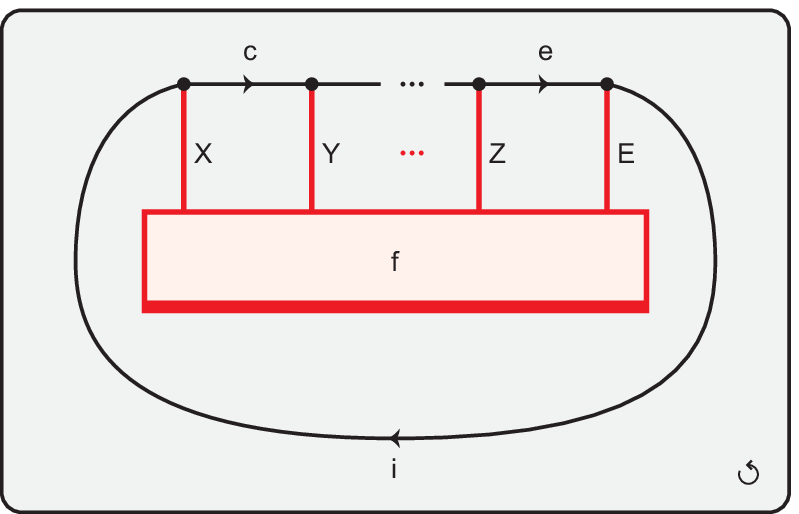}   \, \subset S^2,
\end{equation*}
The case  $ m\geq 1$ and $n=0$ is similar. For $m=n=0$, set
\begin{equation*}
 \psfrag{i}[Bl][Bl]{\scalebox{.9}{$i$}}
 \psfrag{k}[Bc][Bc]{\scalebox{1.11111111}{\textcolor{red}{$f$}}}
 \Gamma_x^\cl=\; \rsdraw{.47}{.9}{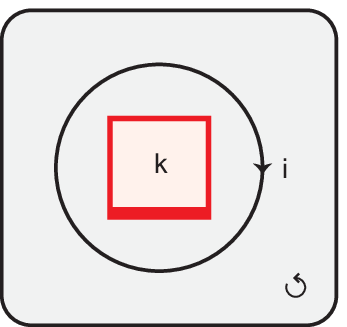}   \, \subset S^2,
\end{equation*}
where here $f \in \End_{\zz_G(\cc)}(\un_{\zz_G(\cc)})  =\kk$.

\item Let $x$ be  a crossing of $d$. A neighborhood of~$x$ in~$P$ looks as follows:
\begin{equation*}
 \psfrag{i}[Bc][Bc]{\scalebox{.9}{$i$}}
 \psfrag{j}[Bc][Bc]{\scalebox{.9}{$j$}}
 \psfrag{k}[Bc][Bc]{\scalebox{.9}{$k$}}
 \psfrag{l}[Bc][Bc]{\scalebox{.9}{$l$}}
 \psfrag{p}[Bl][Bl]{\scalebox{.9}{\textcolor{red}{$\psi$}}}
 \psfrag{u}[Bl][Bc]{\scalebox{.9}{\textcolor{red}{$C$\,}}}
 \psfrag{x}[Bl][Bl]{\scalebox{.9}{\textcolor{red}{$B$}}}
 \psfrag{c}[Bl][Bl]{\scalebox{.9}{\textcolor{red}{$A$}}}
 \psfrag{r}[Bc][Bc]{\scalebox{.9}{\textcolor{red}{$D$}}}
 \rsdraw{.45}{.9}{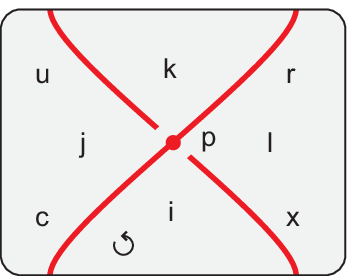}
\end{equation*}
where objects $A,B,C,D \in \zz_G(\cc)$ are the colors of the rims adjacent to $x$, the isomorphism
$\psi$ is the color of $x$, and $i,j,k,l \in I$ are the $\cl$-colors of the faces of~$d$ adjacent to $x$. We associate with~$x$ the  knotted $\cc$-net
\begin{equation*}
 \psfrag{i}[Bc][Bc]{\scalebox{.9}{$i$}}
 \psfrag{j}[Bc][Bc]{\scalebox{.9}{$j$}}
 \psfrag{k}[Bc][Bc]{\scalebox{.9}{$k$}}
 \psfrag{l}[Bc][Bc]{\scalebox{.9}{$l$}}
 \psfrag{p}[Bl][Bl]{\scalebox{.9}{\textcolor{red}{$\psi$}}}
 \psfrag{C}[Br][Br]{\scalebox{.9}{\textcolor{red}{$C$\,}}}
 \psfrag{B}[Bl][Bl]{\scalebox{.9}{\textcolor{red}{$B$}}}
 \psfrag{A}[Br][Br]{\scalebox{.9}{\textcolor{red}{$A$\,}}}
 \psfrag{D}[Bl][Bl]{\scalebox{.9}{\textcolor{red}{$D$}}}
   \Gamma_x^\cl= \;\rsdraw{.45}{.9}{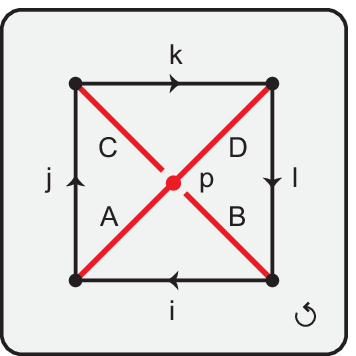}  \, \subset S^2,
\end{equation*}
where the orientations of the diagonals (not shown in the picture)  are
induced by the orientations of the corresponding strands of~$d$.
\end{enumerate}

By Section~\ref{sect-nets-in-sphere}, for any node $x$ of~$d$,  the   knotted $\cc$-net $\Gamma_x^\cl$
yields a   vector
$$
  \inv_\cc (\Gamma_x^\cl) \in \Hom_\kk(H(\Gamma_x^\cl),  \End_\cc(\un))=\Hom_\kk(H(\Gamma_x^\cl),\kk) ={{H(\Gamma_x^\cl)}}^\star.
$$
It results from the definitions   that   we have canonical isomorphisms
$$
H(\Gamma_x^\cl)\simeq \bigotimes_{e_x}\, H_\cl(e_x) \qquad {\rm {and}} \quad H(\Gamma_x^\cl)^\star  \simeq  \bigotimes_{e_x} \, H_\cl(e_x)^\star
$$
where $e_x$ runs over the rims of ${  d}$ incident to $x$ and oriented away from $x$.
The    tensor product of the latter isomorphisms   over all nodes~$x$ of~$d$   yields an isomorphism
$$
\bigotimes_x  H(\Gamma_x^\cl)^\star \simeq \bigotimes_x \bigotimes_{e_x}   H_\cl(e_x)^\star \simeq H_\cl^\star.
$$
The image  under this isomorphism  of the unordered tensor product $\bigotimes_x \inv_\cc (\Gamma_x)$
is a vector $V_\cl \in H_\cl^\star$. Set
$$
\vert \cl \vert = V_\cl(\ast_\cl) \in  \kk \quad \text{and} \quad \dim(\cl)=   \prod_{r \in \Fac(d)} (\dim \cl(r))^{\chi(r)}      \in \kk
$$
where $\chi $ is the Euler characteristic. Finally, set
\begin{equation}\label{eq-state-sum-ribbon-graphs}
|M,  \Omega|_\cc=(\dim (\cc_1))^{-\vert M\setminus P\vert}      \sum_{\cl}  \dim(\cl) \vert \cl \vert
 \in \kk ,
\end{equation}
where  $\cl$ runs over all $G$-colorings of $d$ and the positive integer  $\vert M\setminus P\vert$ is the number of $P$-balls (i.e., the number of connected components of $M\setminus P$).

\begin{thm}\label{thm-state-sum-ribbon-graphs}
The scalar $|M, \Omega|_\cc  $ is a well-defined  isomorphism invariant of the $\zz_G(\cc)$-colored $G$-graph $(M,\Omega)$.  Also, $|M, \Omega|_\cc  $ is invariant under stabilization  and  conjugation of $(M,\Omega)$ (see Section~\ref{sect-stabilization}).
\end{thm}

The proof of Theorem~\ref{thm-state-sum-ribbon-graphs}  consists in verifying that the left-hand side of \eqref{eq-state-sum-ribbon-graphs} remains invariant under the moves on knotted plexuses in skeletons (see the end of Section~\ref{sect-knotted-plexuses-in-skel}). This goes by combining the proofs of \cite[Theorem 7.1]{TVi2} and \cite[Theorem 15.7]{TVi5} and is left to the reader. The   naturality of $\inv_{\cc}$ implies that $|M, \Omega|_\cc  $ is independent of  the choice  of the representative set~$I$   of simple objects of~$\cc$.   For $\Omega=\emptyset$, we obtain the invariant    $|M, \emptyset|_\cc = |M|_\cc  $      of  the $G$-manifold $M$ defined    in  \cite[Theorem 7.1]{TVi2}.

\subsection{Proof of Theorem~\ref{thm-extension-state-sum-graph-HQFT}}\label{sect-proof}
The construction of the state sum graph TQFTs given in \cite[Section 15.7]{TVi5} applies (with colorings replaced by $G$-colorings) to the state sum invariant $\vert \cdot \vert_\cc$ of $\zz_G(\cc)$-colored $G$-graphs (defined in Section~\ref{sect-the-inv-for-MR}) and produces a graph HQFT $\vert \cdot \vert_\cc$ over $\zz_G(\cc)$.
It is clear from the definitions that this graph HQFT extends the state sum HQFT $\vert \cdot \vert_\cc$. This proves Theorem~\ref{thm-extension-state-sum-graph-HQFT}.

\subsection{Properties of $\vert \cdot \vert_\cc$}\label{sect-graph-stsum-hqft-construct}
We state two   properties of the  graph HQFT $\vert \cdot \vert_\cc$. First,  we compute
$|S^3, \Omega|_\cc \in \kk$ for any   $\zz_G(\cc)$-colored $G$-graph $\Omega$ in $S^3$.      By \cite{TVi4},   the $G$-ribbon category $\zz_G(\cc)$ defines a monoidal functor $F_{\zz_G(\cc)}$ from the category of $\zz_G(\cc)$-colored $G$-graphs in $\R^2 \times [0,1]$ to~$\zz_G(\cc)$. In particular,
$$
F_{\zz_G(\cc)}( \Omega) \in \End_{\zz_G(\cc)}(\un_{\zz_G(\cc)})=\kk.
$$
In generalization of \cite[Theorem 16.1]{TVi5} (where   $G=1$), we claim that
\begin{equation}\label{eq-compare-1}
|S^3, \Omega|_\cc  = (\dim (\cc_1))^{-1}  F_{\zz_G(\cc)}( \Omega).
\end{equation}
The proof   repeats  the one of  \cite[Theorem 16.1]{TVi5} with  colorings   replaced by $G$-colorings.

Second, we compute the isomorphism class of the module $\vert \Sigma \vert_\cc$
for any connected $\zz_G(\cc)$-colored $G$-surface $\Sigma$ of genus $g\geq 0$.
Let  $\alpha_1,\beta_1, \dots, $ $ \alpha_g,\beta_g \in G$ and   $U_\Sigma\in \zz_G(\cc)$
 be   as in Section~\ref{sect-surg-HQFT}.
 We claim that if the category~$\cc$ is additive, then
\begin{equation}\label{eq-compare-2}
\vert \Sigma \vert_\cc \simeq \Hom_{\zz_G(\cc)}\Bigl(\un_{\zz_G(\cc)},\widetilde{C}_{\alpha_1,\beta_1} \otimes \cdots \otimes \widetilde{C}_{\alpha_g,\beta_g} \otimes   U_\Sigma  \Bigr),
\end{equation}
where for $\alpha,\beta \in G$, the object $\widetilde{C}_{\alpha,\beta}=(C_{\alpha,\beta},\sigma^{\alpha,\beta})$ of $\zz_G(\cc)$ is defined by
$$
C_{\alpha,\beta}=\bigoplus_{\substack{i \in I_\alpha,  j \in I_\beta}} i^* \otimes j^* \otimes i \otimes j
$$
and, for any $X \in \cc_1$,
$$
\psfrag{i}[Bc][Bc]{\scalebox{.9}{$i$}}
 \psfrag{j}[Bc][Bc]{\scalebox{.9}{$j$}}
 \psfrag{k}[Bc][Bc]{\scalebox{.9}{$k$}}
 \psfrag{l}[Bc][Bc]{\scalebox{.9}{$l$}}
 \psfrag{n}[Bc][Bc]{\scalebox{.9}{$z$}}
 \psfrag{X}[Bc][Bc]{\scalebox{.9}{$X$}}
 \psfrag{a}[cc][cc]{\scalebox{.9}{$p^\beta_{z^* \otimes j \otimes z}$}}
 \psfrag{s}[cc][cc]{\scalebox{.9}{$q^\beta_{z^* \otimes j \otimes z}$}}
 \psfrag{c}[cc][cc]{\scalebox{.9}{$q_{z \otimes k \otimes z^*}^\alpha$}}
 \psfrag{r}[cc][cc]{\scalebox{.9}{$p_{z \otimes k \otimes z^*}^\alpha$}}
  \psfrag{v}[cc][cc]{\scalebox{.9}{$p_X^\gamma$}}
   \psfrag{u}[cc][cc]{\scalebox{.9}{$q_X^\gamma$}}
\sigma^{\alpha,\beta}_X= \!\!\!\sum_{\substack{i,k\in I_\alpha \\ j,l \in I_\beta\\ z \in I_1}} \; \rsdraw{.50}{.9}{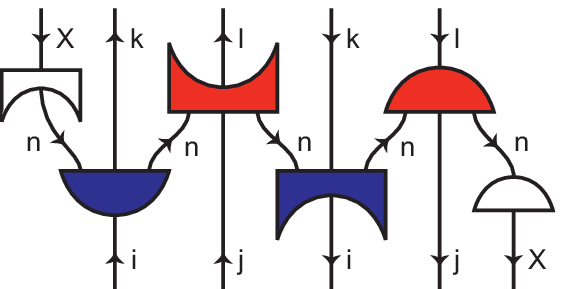} \co C_{\alpha,\beta} \otimes X \to X \otimes C_{\alpha,\beta}.
$$
Here,   complementary curvilinear boxes of the same color represent the projection/inclusion determined by an  isotypic subobject, see \cite[Section 4.6]{TVi5} for   details.
Formula \eqref{eq-compare-2} generalizes \cite[Theorem 16.2]{TVi5} where $G=1$. The proof repeats  the one given there  with  colorings   replaced by $G$-colorings.

\section{Computations in   graph HQFTs}\label{sect-sugery-computation}
In this section,    $Z\co \Cob^G_\bb \to \Mod_\kk$ is an arbitrary   graph HQFT  over a  $G$-crossed category $\bb$ over $\kk$.  We give a surgery formula  for~$Z$ using    so-called torus vectors.

\subsection{Graphs without free ends}\label{sect-free-ends}
A $\bb$-colored $G$-graph   $(M,\Omega)$ \emph{has no free ends}
 if $\Omega \cap \partial M= \emptyset$, that is, if the endpoints of all  arcs of $\Omega$ lie on the bases of the coupons.
Such a graph   represents  a morphism  $(M,\Omega)_+\co\emptyset \to  \partial M$ and a morphism $(M,\Omega)_-\co  -\partial M \to \emptyset$
in the category $\Cob^G_\bb$. Recall the $\kk$-linear isomorphism  $Z_0\co \kk\simeq Z(\emptyset)$ and set
$$
Z_+(M,\Omega) =Z((M,\Omega)_+)\, Z_0\co \kk \to Z(\partial M )
$$
and
$$
Z_-(M,\Omega) =Z_0^{-1}Z((M,\Omega)_-)\co Z(-\partial M )\to  \kk.
$$
If $\partial M=\emptyset$, then any $\bb$-colored $G$-graph $\Omega\subset M$ has no free ends and
  $(M,\Omega)_+=(M,\Omega)_-$. In this case,  the \kt linear homomorphism
$$
Z_0^{-1} Z_+(M,\Omega) = Z_-(M,\Omega) Z_0 \co \kk \to \kk
$$
is multiplication by  the scalar  $Z(M,\Omega)\in \kk$ defined in Section~\ref{sect-def-graph-HQFT}.

\subsection{Torus vectors}\label{sect-torus-vector-G}
We endow the unit disk $D^2=\{ z \in \CC, \vert z\vert \leq 1\}$  and the
unit circle $S^1=\partial D^2$  with the  counterclockwise orientation. Endow  the torus $S^1\times S^1$    with  the product orientation and   the base point  $*=(1,1)$. For
each $\alpha \in G=\pi_1(\X ,\x)$, we let $g_\alpha $  be the unique  homotopy class  of maps $ (S^1\times S^1,*) \to (\X ,\x)$  which carry the loops
$$
t \in [0,1] \mapsto (e^{2\pi i t},1) \in S^1\times S^1 \quad {\rm{and}} \quad t \in [0,1] \mapsto (1, e^{2\pi i t}) \in S^1\times S^1
$$
into     loops  in $(\X ,\x)$    representing respectively~$\alpha$ and $1\in G$. Then  $\T_\alpha=(S^1\times S^1,g_\alpha)$ is a   $\bb$-colored $G$-surface with an empty set of marked points.
Let  ${V}=-( S^1\times D^2)$ be the solid torus with orientation opposite to the product orientation and with   boundary $\partial {V} =S^1\times S^1$ pointed by $(\partial {V})_\bullet=\{\ast\}$.
The homotopy class of maps $g_\alpha$  above extends   uniquely to a homotopy class  of maps
 $\widetilde{g}_\alpha \co  (V, (\partial {V})_\bullet) \to (\X,\x)$.
 The  triple $V_\alpha=(V, \emptyset, \widetilde{g}_\alpha)$ is a $\bb$-colored  $G$-graph with no free ends and with $\partial (V_\alpha) = \T_\alpha$.
The \emph{$\alpha$-torus vector of~$Z$} is the vector
$Z_+(V_\alpha)(1_\kk) \in Z(\T_\alpha)$.

\subsection{The surgery formula}\label{sec-The-surgery-formula}\label{sect-calc-ZLR}
Let  $(M,\Omega)$ be a $\bb$-colored $G$-graph where $M$ is a closed connected oriented $3$-manifold. We give  a surgery formula for   $Z(M,\Omega)\in \kk$.
To this end, present $M$ as the result of surgery on $S^3=\RR^3\cup\{\infty\}$
along a framed  link $L=L_1\cup \cdots \cup L_n\subset \RR^2\times (0,1)$ with     $n\geq 1 $ components.
Pick  a closed regular neighborhood $U\subset S^3$ of~$L$
and  let $E =S^3\setminus \Int (U)$  be     the    exterior  of~$L$ in $S^3$.   We    endow both~$U$ and~$E$ with orientation induced by the right-handed orientation of~$S^3$. By  the definition of  surgery,~$M$ is obtained by gluing~$n$ solid tori to~$E$.
Deforming if necessary $\Omega$ in~$M$, we  assume that $\Omega\subset \Int(E)\subset M$   and that $\Omega_\bullet \subset \partial E$ when $\Omega\neq \emptyset$.

  We  now turn the pair $(E,\Omega)$ into a $\bb$-colored $G$-graph. Let  $ n (S^1 \times S^1   )$   be a disjoint union of~$n$ copies of  $S^1 \times S^1$.
Orient the link~$L$ arbitrarily  and pick an  orientation-preserving  diffeomorphism
$$
 f_L \co  n (S^1 \times S^1)      \longrightarrow \partial U =-\partial E
$$
which, for all   $p \in S^1$   and        $q=1,\dots,n$, carries the $q$-th copy of $S^1\times   \{ p  \} $ to  a positively oriented meridian of $L_q$ in $\partial U$ and carries the $q$-th copy of $   \{p\}   \times S^1$  to  the positively oriented longitude of $L_q$ in $\partial U$ determined by the
framing of~$L$.
Let $(\partial E)_\bullet \subset \partial E$ be the set consisting of the images under $f_L$ of the base points
$*=(1,1)$ of the~$n$ copies of $S^1\times S^1$.  We choose $ f_L $ so that $\Omega_\bullet \subset (\partial E)_\bullet$ when $\Omega\neq \emptyset$.  So, $E$ becomes an  oriented 3-manifold with pointed boundary  and $\Omega$  is a ribbon graph in~$E$ with no free ends.
Pick in the given homotopy class of maps $(M \setminus \Omega, \Omega_\bullet) \to (\X,\x)$ a map carrying
 $(\partial E)_\bullet \subset M \setminus \Omega$ to $\x$ and restrict it  to
$E \setminus \Omega$. This gives a map
$g\co (E\setminus \Omega, (\partial E)_\bullet ) \to  (\X,\x)$  turning  $(E,\Omega)$ into a $G$-graph.
The $\bb$-coloring of $(M,\Omega)$ induces a $\bb$-coloring of $(E,\Omega)$.
Indeed, when $\Omega\neq \emptyset$, pick a 1-system and a 2-system of tracks for $\Omega$ in $E$ that start in the point of
$\Omega_\bullet \subset (\partial E)_\bullet$. These systems together with their evaluation under the $\bb$-coloring of $(M,\Omega)$ define a $\bb$-coloring of  $(E,\Omega)$ as in Section~\ref{sect-contruct-colorings}.
Thus, $(E,\Omega)$ is a $\bb$-colored $G$-graph with no  free ends.

The homotopy class of the  map $g   f_L \co n (S^1 \times S^1)  \to (\X,\x)$ turns $  n (S^1 \times S^1)   $
into a $G$-surface which, in
  the notation of Section~\ref{sect-torus-vector-G},    is a disjoint union $\sqcup_{q=1}^n \T_{\alpha_q} $  where
 $\alpha_q\in G=\pi_1(\X ,\x)$ is represented by the image under $g f_L$ of the loop
  $S^1\to   n (S^1 \times S^1)   $ carrying  any $s \in S^1$ to the point $(s,1)$ in the $q$-th copy of $S^1\times S^1$.
Then the map $f_L$ above is an isomorphism  of $G$-surfaces  $\sqcup_{q=1}^n \T_{\alpha_q}  \to -\partial E$. Its
 cylinder  (see Section~\ref{sect-cyl-id}) induces   a \kt linear isomorphism
$$
Z(f_L)=Z(\cyl(f_L)) \co Z (\sqcup_{q=1}^n \T_{\alpha_q}) \to Z(-\partial E).
$$
The monoidal constraints of the graph HQFT $Z$ induce a \kt linear isomorphism
$$
  \xi_n    \co  Z(\T_{\alpha_1}) \otimes \cdots \otimes  Z(\T_{\alpha_n})
  \to Z(\sqcup_{q=1}^n \T_{\alpha_q}).
$$ Composing these  isomorphisms  with the  homomorphism $
Z_-(E,\Omega) \co   Z (-\partial E )   \to  \kk
$
(see Section~\ref{sect-free-ends}), we  obtain a \kt linear homomorphism
\begin{equation*}
Z^L=Z_-(E,\Omega) \circ  Z(f_L) \circ  \xi_n  \co Z(\T_{\alpha_1}) \otimes \cdots \otimes  Z(\T_{\alpha_n})   \to
\kk.
\end{equation*}

\begin{lem}\label{lem-TQFTsurg-G}
For any $\alpha\in G$, let $w_\alpha  \in Z(\T_{\alpha})$ be the $\alpha$-torus vector of~$Z$. Then
$$
Z(M, \Omega  )= Z^L \, (w_{\alpha_1} \otimes \cdots  \otimes w_{\alpha_n}).
$$
\end{lem}
\begin{proof}
Recall from Section~\ref{sect-torus-vector-G} the  $\bb$-colored $G$-graphs $\{V_\alpha\}_{\alpha \in G}$ with $\partial V_\alpha=\T_\alpha$.
Since $ (M, \Omega ) $ is obtained by gluing $V_{\alpha_1} \sqcup \cdots \sqcup V_{\alpha_n}$ to $(E,\Omega)$  along
$f_L$, we have
$$
(M,\Omega)_+=(E,\Omega)_- \circ \cyl(f_L)  \circ ((V_{\alpha_1})_+ \sqcup \cdots \sqcup (V_{\alpha_n})_+)\co \emptyset \to \emptyset
$$
in the category $\Cob^G_\bb$. Applying the functor~$Z$, we get
$$
Z((M,\Omega)_+)=Z((E,\Omega)_-) \circ Z( f_L)  \circ Z((V_{\alpha_1})_+ \sqcup \cdots \sqcup (V_{\alpha_n})_+) \co Z(\emptyset) \to
Z(\emptyset).
$$
Since $Z(M,  \Omega   ) = Z_0^{-1} Z((M,\Omega)_+)Z_0(1_\kk)$, we have
$$
Z(M,  \Omega   ) = Z_0^{-1} \,  Z((E,\Omega)_-) \, Z(f_L)  \, Z((V_{\alpha_1})_+ \sqcup \cdots \sqcup (V_{\alpha_n})_+)\, Z_0(1_\kk).
$$
Now, by definition,  $Z_-(E,\Omega)=Z_0^{-1}  Z((E,\Omega)_-)$. Also,
 the monoidality of~$Z$ and the definition of the torus vectors imply that
$$
Z((V_{\alpha_1})_+ \sqcup \cdots \sqcup (V_{\alpha_n})_+) \, Z_0(1_\kk)
=\xi_n (w_{\alpha_1} \otimes \cdots  \otimes w_{\alpha_n}).
$$
Therefore,
\begin{gather*}
Z(M,  \Omega   ) = Z_-(E,\Omega) \,  Z(f_L)  \,  \xi_n (w_{\alpha_1} \otimes \cdots  \otimes w_{\alpha_n})
= Z^L \, (w_{\alpha_1} \otimes \cdots  \otimes w_{\alpha_n}). \qedhere
\end{gather*}
\end{proof}

We now evaluate  $Z^L$ on  certain  vectors.
Consider the solid torus  $W=D^2 \times S^1$ with  product orientation and  base point   $*=(1,1)$  in  $\partial {W} =S^1\times S^1$. Consider the knot  $K= \{0\}\times S^1\subset \Int(W)$     with orientation    induced by the opposite
(clockwise) orientation of  $S^1$ and  with framing  $((1,0), 0)$  at all points.  Insert  in~$K$  a  coupon  transversal to the framing and having one input and one output (see Section~\ref{sect-stabilization}).  We choose the  bottom base of the coupon so that the input is directed out of the coupon. In this way, we stabilize~$K$
into  a ribbon graph $K^s \subset W$.
 For  $\alpha \in G$, the homotopy class  of maps $g_\alpha \co (S^1\times S^1,*) \to (\X ,\x)$ from Section~\ref{sect-torus-vector-G}
extends uniquely to a homotopy class  of maps
$$
g^+_\alpha \co (W \setminus K^s, (\partial W)_\bullet=\{\ast \}) \to (\X,\x).
$$
The triple  $W_\alpha=(W,K^s,g^+_\alpha)$ is a $G$-graph with no free ends. Note that for any track $\gamma$ of a stratum of $K^s$, the
associated element  $\mu_\gamma \in \pi_1( W \setminus K^s, \ast)$  is carried by~$g^+_\alpha$  to $\alpha \in \pi_1(\X ,\x)=G$. Each  object $X \in \bb_\alpha$  determines a $\bb$-coloring of $W_\alpha$   as follows.
Pick a track of  the only arc and a track of  the only coupon   of~$K^s$. These tracks forms a 1-system and a 2-system of tracks for $W_\alpha$ and so, when colored by the object $X$ and the morphism $\id_{\varphi_1(X)}$, they determine a $\bb$-coloring of~$W_\alpha$ (see Section~\ref{sect-contruct-colorings}). This yields a $\bb$-colored
$G$-graph $W_\alpha^X$ with no free ends  such that $\partial (W^X_\alpha) = \T_\alpha$. Set
\begin{equation}\label{eq-def-yX}
[X]=Z_+(W^X_\alpha) (1_\kk) \in  Z(\T_\alpha)  .
\end{equation}
For  any $X_1 \in \bb_{\alpha_1}, \dots, X_n \in \bb_{\alpha_n}$, gluing $W_{\alpha_1}^{X_1} \sqcup \cdots \sqcup W_{\alpha_n}^{X_n}$ to $(E,\Omega)$  along  $f_L$ yields a $\bb$-colored $G$-graph $(S^3,T_{X_1, \dots, X_n})$.
Note that the underlying ribbon graph of $T_{X_1, \dots, X_n}\subset S^3$ is the union $L^s \cup \Omega$, where $L^s$ is obtained from~$L$  by stabilizing each of its  components.
\begin{lem}\label{lem-formul-HUL}
For  any $X_1 \in \bb_{\alpha_1}, \dots, X_n \in \bb_{\alpha_n}$,
$$
Z^L \, (  [{X_1}] \otimes\cdots \otimes [{X_n}])= Z(S^3,T_{X_1, \dots, X_n})\in \kk.
$$
\end{lem}
\begin{proof}
The monoidality of the functor~$Z$ implies that
$$
  \xi_n      \, (  [{X_1}] \otimes \cdots \otimes[{X_n}])= Z_+(W_{\alpha_1}^{X_1} \sqcup \cdots \sqcup W_{\alpha_n}^{X_n})(1_\kk).
$$
Then
\begin{gather*}
Z^L \, (  [{X_1}] \otimes \cdots \otimes  [{X_n}])= Z_-(E,\Omega) \, Z(f_L)\,  Z_+(W_{\alpha_1}^{X_1} \sqcup \cdots \sqcup W_{\alpha_n}^{X_n}) (1_\kk)  \\
=Z_0^{-1} Z((E,\Omega)_-) \, Z(f_L)\,  Z((W_{\alpha_1}^{X_1} \sqcup \cdots \sqcup W_{\alpha_n}^{X_n})_+)\,  Z_0(1_\kk)
= Z(S^3, T_{X_1, \dots, X_n} ),
\end{gather*}
where the last equality follows from the functoriality of~$Z$.
\end{proof}

\section{Proof of Theorem~\ref*{thm-comparison}}\label{sect-comparison}
In this section, we prove Theorem~\ref{thm-comparison}. Recall that $\kk$ is an algebraically closed field and  $\cc=\oplus_{\alpha\in G} \, \cc_\alpha$ is an additive spherical $G$-fusion category over $\kk$ such that $\dim(\cc_1) \neq 0$. By Section~\ref{sect-center-graded-cat}, $\zz_G(\cc)$ is an additive anomaly free $G$-modular category whose canonical rank is equal to $\dim(\cc_1)$.
Recall that $I=\amalg_{\alpha\in G}\, I_\alpha$  denotes  a   representative set of simple objects of~$\cc$. Let~$\mathcal{J}=\amalg_{\alpha\in G}\, \mathcal{J}_\alpha$  be  a   representative set of simple objects of~$\zz_G(\cc)$.

The surgery graph HQFT $\tau_{\zz_G(\cc)}$ and the state sum graph HQFT $|\cdot|_\cc$ are (symmetric) monoidal functors $\Cob^G_{\zz_G(\cc)} \to \Mod_\kk$. Now the category $\Cob^G_{\zz_G(\cc)}$ is left rigid  (see Section~\ref{sect-pptes-cob-G}) and $\tau_{\zz_G(\cc)}$ is non-degenerate (see Section~\ref{sect-surg-HQFT}). Thus, by \cite[Lemma 17.2]{TVi5}, we only need to prove that:
\begin{enumerate}
\labela
\item for any $\zz_G(\cc)$-colored   $G$-surface $\Sigma$,
    $$|\Sigma|_\cc \simeq \tau_{\zz_G(\cc)}(\Sigma);$$
\item for any $\zz_G(\cc)$-colored $G$-graph  $\Omega$  in a closed oriented 3-manifold~$M$,
$$
\vert M,\Omega \vert_\cc=\tau_{\zz_G(\cc)}(M,\Omega).
$$
\end{enumerate}

\subsection{Proof of (a)}\label{sect-proof-a}
For $\alpha,\beta \in G$, recall the object $\widetilde{C}_{\alpha,\beta}\in\zz_G(\cc)$ from Section~\ref{sect-graph-stsum-hqft-construct}. By  \cite[Theorem A1]{TVi3}, this object is the coend
$$
\widetilde{C}_{\alpha,\beta}=\int^{X \in \zz_\beta(\cc)} (\varphi_\alpha(X))^* \otimes X.
$$
Now, since $\zz_\beta(\cc)$ is additive and finitely semisimple with $\mathcal{J}_\beta$ as a representative  set of simple objects, we have:
$$
\int^{X \in \zz_\beta(\cc)} (\varphi_\alpha(X))^* \otimes X \, =  \!\! \bigoplus_{J \in \mathcal{J}_\beta} (\varphi_\alpha(J))^* \otimes J.
$$
Consequently, the uniqueness of a coend implies that
\begin{equation}\label{eq-calc-coend}
\widetilde{C}_{\alpha,\beta} \, \simeq  \!\! \bigoplus_{J \in \mathcal{J}_\beta} (\varphi_\alpha(J))^* \otimes J.
\end{equation}

Let $\Sigma$ be a connected $\zz_G(\cc)$-colored $G$-surface of genus $g\geq 0$.
The surface~$\Sigma$ carries a base point~$\ast$, a finite set of marked points  $A=\{a_1, \dots,a_m\}$,  and a homotopy class of  maps
$
(\Sigma \setminus A, \ast)\to (\X,\x).
$
Pick a track $\gamma_i$ of $a_i$ for each $i \in \{1, \dots m \}$. Recall from Section~\ref{sect-marked-surfaces} the homotopy class $\mu_{\gamma_i} \in \pi_1(\Sigma \setminus A,\ast)$ of the loop encircling $a_i$.
The group $\pi_1(\Sigma \setminus A,\ast)$ is generated by $\mu_{\gamma_1}, \dots, \mu_{\gamma_m}$ and $2g$ elements $\alpha_1,\beta_1, \dots, \alpha_g,\beta_g$ subject to the only relation
\begin{equation*}
[\alpha_1,\beta_1] \cdots [\alpha_g,\beta_g](\mu_{\gamma_1})^{\varepsilon_1} \cdots (\mu_{\gamma_m})^{\varepsilon_m}=1,
\end{equation*}
where $[\alpha,\beta]=\alpha^{-1}\beta^{-1}\alpha\beta$ and $\varepsilon_i=\pm 1$ is the sign of $a_i$. Consider the object
$U_\Sigma \in \zz_G(\cc)$ defined as in Section~\ref{sect-surg-HQFT}. By \eqref{eq-compare-2}, we have
$$
\vert \Sigma \vert_\cc \simeq \Hom_{\zz_G(\cc)}\Bigl(\un_{\zz_G(\cc)},\widetilde{C}_{\alpha_1,\beta_1} \otimes \cdots \otimes \widetilde{C}_{\alpha_g,\beta_g} \otimes   U_\Sigma  \Bigr).
$$
Now \eqref{eq-calc-coend} implies that
$$
\widetilde{C}_{\alpha_1,\beta_1} \otimes \cdots \otimes \widetilde{C}_{\alpha_g,\beta_g} \, \simeq
\!\!\!\!\!\!\!\!\!\!
\bigoplus_{J_1\in \mathcal{I}_{\beta_1}, \ldots,J_g \in \mathcal{I}_{\beta_g}}
\!\!\!\!\!\!\!\!\!\!
\bigl(\varphi_{\alpha_1}(J_1)^* \otimes J_1 \bigr) \otimes \cdots \otimes \bigl(\varphi_{\alpha_g}(J_g)^* \otimes J_g \bigr).
$$
The last two formulas together with the additivity of $\zz_G(\cc)$ imply that  $\vert \Sigma \vert_\cc$ is isomorphic to
$$
\bigoplus_{J_1\in \mathcal{I}_{\beta_1}, \ldots,J_g \in \mathcal{I}_{\beta_g}}
\!\!\!\!\!\!\!\!\!\!
\Hom_{\zz_G(\cc)}\Bigl(\un_{\zz_G(\cc)},\bigl(\varphi_{\alpha_1}(J_1)^* \otimes J_1 \bigr) \otimes \cdots \otimes \bigl(\varphi_{\alpha_g}(J_g)^* \otimes J_g \bigr) \otimes   U_\Sigma  \Bigr).
$$
Now \eqref{eq-def-tau-B-Sigma} applied with $\bb=\zz_G(\cc)$ implies that the latter vector space is isomorphic to~$\tau_{\zz_G(\cc)}(  \Sigma  )$.
Consequently $\vert \Sigma \vert_\cc$ is isomorphic to  $\tau_{\zz_G(\cc)}(  \Sigma  )$.

The case of disconnected~$\Sigma$ is deduced from the case of  connected~$\Sigma$ using Formula~\eqref{eq-cas-non-connexe} applied to the graph HQFTs  $  \vert \cdot \vert_\cc$ and $ \tau_{\zz_G(\cc)}$.

\subsection{Proof of (b)}\label{sect-proof-b}
Let $\alpha\in G$. Recall that any object $X \in \zz_\alpha(\cc)$ determines a vector  $[X] \in
\vert \T_\alpha \vert_\cc  $    as in \eqref{eq-def-yX} with $\bb=\zz_G(\cc)$. Proceeding as in \cite[Section 17.4]{TVi5},
we obtain that the family  $ \{[J]\}_{J \in {\mathcal J}_\alpha}$ is  a basis of~$\vert \T_\alpha \vert_\cc$ and that the
 $\alpha$-torus vector  $w_\alpha \in \vert \T_\alpha \vert_\cc$  of $| \cdot |_\cc$
is  computed in this basis by
\begin{equation}\label{expansionofthetorusvector}
w_\alpha= (\dim(\cc_1))^{-1}  \sum_{J\in \mathcal{J}_\alpha} \dim(J) \,[J].
\end{equation}

Pick a $\zz_G(\cc)$-colored $G$-graph $\Omega$ in a closed oriented 3-manifold~$M$.
Since $\vert M,\Omega \vert_\cc$ and $\tau_{\zz_G(\cc)}(M,\Omega)$ are multiplicative under disjoint union, it suffices to consider the case where $M$ is connected. Present $M$ by surgery on $S^3$ along a framed oriented  link $L=L_1\cup \cdots \cup L_n \subset S^3$. We use the notation of Section~\ref{sec-The-surgery-formula} with $Z=|\cdot|_\cc$. By Lemma~\ref{lem-TQFTsurg-G} and Formula~\eqref{expansionofthetorusvector},
\begin{gather*}
\vert M , \Omega\vert_\cc = Z^L\, (w_{\alpha_1} \otimes \cdots  \otimes w_{\alpha_n}) \\
= \sum_{J_1\in \mathcal{J}_{\alpha_1}, \dots , J_n \in \mathcal{J}_{\alpha_n}} \left(\prod_{q=1}^n \frac{\dim(J_q)}{\dim(\cc_1)}  \right)   Z^L ([J_1] \otimes  \cdots \otimes [J_n]).
\end{gather*}
Using the notation of Section~\ref{sect-calc-ZLR}, it follows from
Lemma~\ref{lem-formul-HUL} and Formula~\eqref{eq-compare-1}   that for any $J_1\in \mathcal{J}_{\alpha_1}, \dots , J_n \in \mathcal{J}_{\alpha_n}$,
\begin{gather*}
Z^L ([J_1] \otimes  \cdots \otimes [J_n]))= \vert S^3,T_{J_1, \dots, J_n}\vert_\cc=(\dim (\cc_1))^{-1}  F_{\zz_G(\cc)}(T_{J_1, \dots, J_n}).
\end{gather*}
Therefore
\begin{gather*}
\vert M , \Omega\vert_\cc
= \sum_{J_1\in \mathcal{J}_{\alpha_1}, \dots , J_n \in \mathcal{J}_{\alpha_n}} \left(\prod_{q=1}^n \frac{\dim(J_q)}{\dim(\cc_1)}  \right)   (\dim (\cc_1))^{-1}  F_{\zz_G(\cc)}(T_{J_1, \dots, J_n}) \\
= (\dim(\cc_1))^{-n-1} \!\!\!\!\!\!\!\!\sum_{J_1\in \mathcal{J}_{\alpha_1}, \dots , J_n \in \mathcal{J}_{\alpha_n}} \left(\prod_{q=1}^n \dim(J_q)  \right) F_{\zz_G(\cc)}(T_{J_1, \dots, J_n})=\tau_{\zz_G(\cc)}(M,\Omega),
\end{gather*}
where the last equality is the definition of $\tau_{\zz_G(\cc)}(M,\Omega)$,  see Formula~\eqref{eq-def-tau-B-M-Omega}.

\appendix

\section{The crossing and braiding of the graded center}\label{sect-appendix}

In this appendix, we summarize the construction of the canonical crossing and braiding of the $G$-center of a $G$-graded category,  referring to \cite{TVi3} for details. We formulate these constructions for the class of non-singular $G$-graded categories over a commutative ring which includes the class of $G$-fusion categories over a field.

\subsection{Non-singular graded categories}\label{sect-nonsing-graded-cat}
A monoidal category is   \emph{pure} if $f \otimes \id_X=\id_X \otimes f$  for all object $X$ and all endomorphism $f$ of the monoidal unit $\un$.
In a pure pivotal category, the  left and right traces of endormorphisms are $\otimes$-multiplicative.
By \cite[Remarks 4.2.2]{TVi5}, a \kt linear monoidal category with simple monoidal unit is pure. Note also that the $G$-center of a pure $G$-graded category is pure.

An \emph{idempotent} in a category is an endomorphism $e$ of an object such that $e^2=e$. An
idempotent $e\co X \to X$ \emph{splits} if there is an object~$E$ and   morphisms $p\colon X \to E$ and
$q\colon E \to X$ such that $qp=e$ and $pq=\id_E$. Note that such a {\it splitting triple} $(E,p,q)$ of $e$ is   unique up to   isomorphism.
A \emph{category with split idempotents} is a category in which
all idempotents split.

Following \cite{TVi3}, a $G$-graded category    $\cc$ is
\emph{non-singular} if it is pure, has split idempotents, and for all $\alpha \in G$, the subcategory
$\cc_\alpha$ of $\cc$ has  at least  one object whose left dimension is
invertible in $\End_\cc(\un)$.

Any $G$-fusion category $\cc$ over $\kk$ is pure (since its monoidal unit $\un$ is simple) and both the left and right dimensions of   simple objects of  $\cc$ are invertible in $\End_\cc(\un)=\kk$ (see \cite[Section 4.4.2]{TVi5}). Also, if~$\kk$ is a field, then $\cc$ has   split idempotents. Therefore   any  $G$-fusion category over a field   is non-singular.

\subsection{Notation}
We fix until the end of the appendix a non-singular $G$-graded category $\cc$ over $\kk$. Denote by $\ee$  the class of homogeneous objects of $\cc$ with invertible left dimension. This class decomposes as
$\ee=\amalg_{\alpha \in G}\, \ee_\alpha$, where $\ee_\alpha=\ee\cap \cc_\alpha$. (Note that $\ee_\alpha \neq \emptyset$ since $\cc$ is non-singular.)
For $V\in \ee$, we set $$d_V=\dim_l(V)\in \End_\cc (\un).$$

 By \cite[Theorem 4.1]{TVi3}, the $G$-center $\zz_G(\cc)$ of $\cc$ has a canonical structure of a $G$-braided  category. In the next sections, we
recall the construction of the crossing and of the (enhanced) $G$-braiding of $\zz_G(\cc)$. We also compute the twist of $\zz_G(\cc)$.

\subsection{The crossing}
The crossing in $\zz_G(\cc)$ is constructed  in three steps. At Step~1, we associate a monoidal endofunctor of $\zz_G(\cc)$ to each homogeneous object of $\cc$ with invertible left dimension. At Step 2, we construct a system of   isomorphisms between these endofunctors. At Step 3, we define the crossing as the limit of the resulting projective system of endofunctors and isomorphisms.

\emph{Step 1.} For any $V \in \ee$, we define a  monoidal endofunctor $\varphi_V$ of $\zz_G(\cc)$ as follows.  For  any $(A,\sigma) \in \zz_G(\cc)$, the morphism
$$
\pi^V_{(A,\sigma)}= d_V^{-1}\; \psfrag{A}[Bc][Bc]{\scalebox{.9}{$A$}}
 \psfrag{V}[Bc][Bc]{\scalebox{.9}{$V$}}
 \psfrag{s}[Bc][Bc]{\scalebox{.9}{$\sigma_{V \otimes V^*}$}}
 \rsdraw{.45}{1}{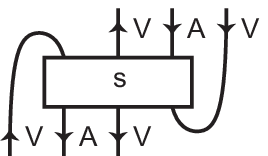} \in \End_\cc(V^* \otimes A \otimes V)
$$
is an idempotent. Since  all  idempotents in $\cc$   split,  there exist an
object $E^V_{(A,\sigma)} \in \cc$ and morphisms $p^V_{(A,\sigma)}\co
V^* \otimes A \otimes V \to E^V_{(A,\sigma)}$ and
$q^V_{(A,\sigma)}\co E^V_{(A,\sigma)} \to V^* \otimes A \otimes V$
such that
$$
\pi^V_{(A,\sigma)}=q^V_{(A,\sigma)}p^V_{(A,\sigma)}\quad \text{and} \quad p^V_{(A,\sigma)}q^V_{(A,\sigma)}=\id_{E^V_{(A,\sigma)}}.
$$
When $A$ is homogeneous, we can and always   choose $E^V_{(A, \sigma)}$  to be homogeneous (of degree   $|V|^{-1} |A| \, |V|$).
We will depict the morphisms $p^V_{(A,\sigma)}$ and $q^V_{(A,\sigma)}$ as
$$
 \psfrag{A}[Bl][Bl]{\scalebox{.8}{$A$}}
 \psfrag{V}[Bl][Bl]{\scalebox{.8}{$V$}}
 \psfrag{E}[Bl][Bl]{\scalebox{.8}{$E_{(A,\sigma)}^V$}}
 p^V_{(A,\sigma)}=\rsdraw{.45}{1}{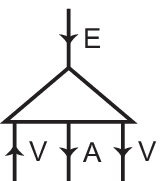}\quad \text{and} \quad  q^V_{(A,\sigma)}=\rsdraw{.45}{1}{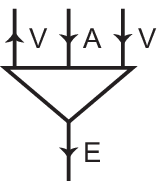} \;.
$$
The family $\gamma^V_{(A,\sigma)}=\{\gamma^V_{(A,\sigma),X}\}_{X \in \cc_1}$, where
$$
\gamma^V_{(A,\sigma),X}= d_V^{-1}\; \psfrag{A}[Bc][Bc]{\scalebox{.8}{$A$}}
 \psfrag{V}[Bc][Bc]{\scalebox{.8}{$V$}}
  \psfrag{X}[Bc][Bc]{\scalebox{.8}{$X$}}
   \psfrag{E}[Bl][Bl]{\scalebox{.8}{$E^V_{(A,\sigma)}$}}
 \psfrag{s}[Bc][Bc]{\scalebox{.9}{$\sigma_{V \otimes X \otimes V^*}$}}
 \rsdraw{.45}{1}{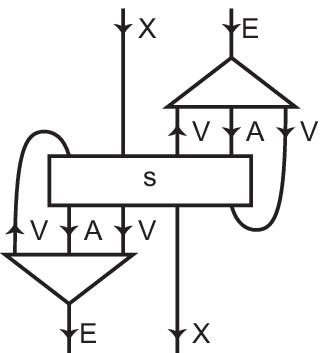} : E^V_{(A,\sigma)} \otimes X\to X\otimes E^V_{(A,\sigma)},
$$
is a  half-braiding of $\cc$ relative to $\cc_1$. Set
$$
\varphi_V(A,\sigma)=(E^V_{(A,\sigma)},\gamma^V_{(A,\sigma)}) \in \zz_G(\cc).
$$
For   a  morphism $f\co (A,\sigma) \to (B,\rho)$ in $\zz_G(\cc)$, set
$$\varphi_V(f)=
 \psfrag{V}[Bc][Bc]{\scalebox{.8}{$V$}}
 \psfrag{A}[Bc][Bc]{\scalebox{.8}{$A$}}
 \psfrag{B}[Bc][Bc]{\scalebox{.8}{$B$}}
 \psfrag{E}[Bl][Bl]{\scalebox{.8}{$E^V_{(A,\sigma)}$}}
 \psfrag{F}[Bl][Bl]{\scalebox{.8}{$E^V_{(B,\rho)}$}}
 \psfrag{f}[Bc][Bc]{\scalebox{.9}{$f$}}
 \rsdraw{.45}{1}{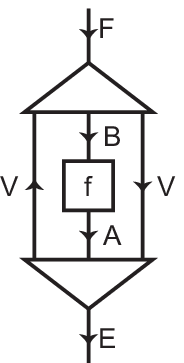} \; \co \varphi_V(A,\sigma) \to \varphi_V(B,\rho).
$$
The monoidal constraints of $\varphi_V$ are defined for any $(A,\sigma),(B,\rho)\in  \zz_G(\cc)$ by
$$(\varphi_V)_2\bigl((A,\sigma),(B,\rho)\bigr)=  \;\psfrag{A}[Bc][Bc]{\scalebox{.8}{$A$}}
 \psfrag{V}[Bc][Bc]{\scalebox{.8}{$V$}}
 \psfrag{B}[Bc][Bc]{\scalebox{.8}{$B$}}
 \psfrag{E}[Bl][Bl]{\scalebox{.8}{$E_{(A,\sigma)}^V$}}
 \psfrag{F}[Bl][Bl]{\scalebox{.8}{$E_{(B,\rho)}^V$}}
 \psfrag{G}[Bl][Bl]{\scalebox{.8}{$E_{(A,\sigma)\otimes (B,\rho)}^V$}}
 \rsdraw{.45}{1}{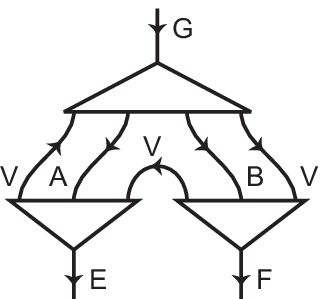} \quad \text{and} \quad
 (\varphi_V)_0= \;
 \psfrag{V}[Bc][Bc]{\scalebox{.8}{$V$}}
 \psfrag{F}[Bl][Bl]{\scalebox{.8}{$E^V_{(\un,\id)}$}}
 \rsdraw{.45}{1}{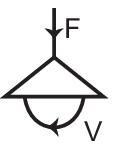}\;.
$$
Then $(\varphi_V,(\varphi_V)_2,(\varphi_V)_0)$ is a pivotal strong  monoidal $\kk$-linear  endofunctor of $\zz_G(\cc)$ such that
$\varphi_V\bigl(\zz_{\beta}(\cc)\bigr) \subset \zz_{|V|^{-1} \beta \, |V|}(\cc)$ for all $\beta \in G$.

The endofunctors $\{ \varphi_V\}_{V \in \ee}$ are related as follows. First, pick any    $U \in \ee_\alpha$, $V \in \ee_\beta$,   $W \in \ee_{\beta \alpha}$ with $\alpha,\beta \in G$. For any $(A,\sigma)\in \zz_G(\cc)$, set
$$
\zeta^{U,V,W}_{(A,\sigma)}=  d_U^{-1}d_V^{-1}\;\psfrag{A}[Bl][Bl]{\scalebox{.7}{$A$}}
 \psfrag{V}[Bl][Bl]{\scalebox{.7}{$U$}}
  \psfrag{R}[Bl][Bl]{\scalebox{.7}{$V$}}
    \psfrag{T}[Bl][Bl]{\scalebox{.7}{$W$}}
   \psfrag{E}[Bl][Bl]{\scalebox{.8}{$E_{\varphi_V(A,\sigma)}^U$}}
      \psfrag{F}[Bl][Bl]{\scalebox{.7}{$E_{(A,\sigma)}^V$}}
         \psfrag{G}[Bl][Bl]{\scalebox{.8}{$E_{(A,\sigma)}^W$}}
 \psfrag{t}[Bc][Bc]{\scalebox{.8}{$p_{(A,\sigma)}^W$}}
  \psfrag{s}[Bc][Bc]{\scalebox{.9}{$\sigma_{V \otimes U \otimes W^*}$}}
 \psfrag{u}[Bc][Bc]{\scalebox{.8}{$q_{(A,\sigma)}^V$}}
 \psfrag{e}[Bc][Bc]{\scalebox{.8}{$q_{\varphi_V(A,\sigma)}^U$}}
 \rsdraw{.45}{1}{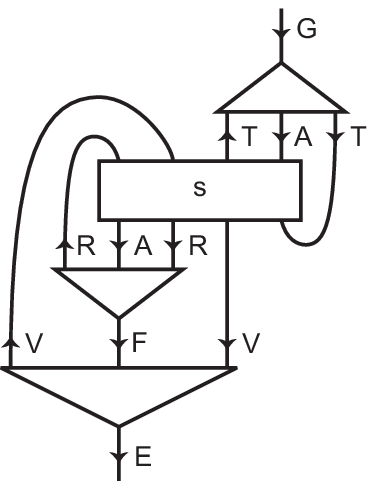} \colon \varphi_U \varphi_V(A,\sigma)  \to \varphi_W (A,\sigma).
$$
Then the family  $\zeta^{U,V,W}=\{\zeta^{U,V,W}_{(A,\sigma)}\}_{(A,\sigma)\in
\zz_G(\cc)}$ is a monoidal natural isomorphism from $\varphi_U\varphi_V$ to $\varphi_W$. Second,
pick any $U \in \ee_1$. For any $(A,\sigma)\in \zz_G(\cc)$, set
$$
\eta^U_{(A,\sigma)}=  \;\psfrag{A}[Bl][Bl]{\scalebox{.8}{$A$}}
 \psfrag{V}[Bl][Bl]{\scalebox{.8}{$U$}}
   \psfrag{E}[Bl][Bl]{\scalebox{.8}{$E_{(A,\sigma)}^U$}}
  \psfrag{s}[Bc][Bc]{\scalebox{.9}{$\sigma_{U}$}}
 \rsdraw{.45}{1}{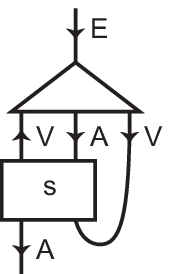} \; \colon (A,\sigma)  \to \varphi_U (A,\sigma).
$$
Then the family $\eta^U=\{\eta^U_{(A,\sigma)}\}_{(A,\sigma)\in \zz_G(\cc)}$ is a monoidal natural isomorphism from the identity endofunctor $1_{\zz_G(\cc)}$ to $\varphi_U$.\\

\emph{Step 2.} For   $U, V \in \ee_\alpha$ with $\alpha\in G$, the family $\delta^{U,V}=\{\delta^{U,V}_{(A,\sigma)}\}_{(A,\sigma)\in \zz_G(\cc)}$, where
$$
\delta^{U,V}_{(A,\sigma)}=d_V^{-1}\;
 \psfrag{A}[Bc][Bc]{\scalebox{.8}{$A$}}
 \psfrag{U}[Bc][Bc]{\scalebox{.8}{$V$}}
 \psfrag{V}[Bc][Bc]{\scalebox{.8}{$U$}}
 \psfrag{s}[Bc][Bc]{\scalebox{.9}{$\sigma_{V \otimes U^*}$}}
 \psfrag{E}[Bl][Bl]{\scalebox{.8}{$E_{(A,\sigma)}^V$}}
 \psfrag{F}[Bl][Bl]{\scalebox{.8}{$E_{(A,\sigma)}^U$}}
 \rsdraw{.45}{1}{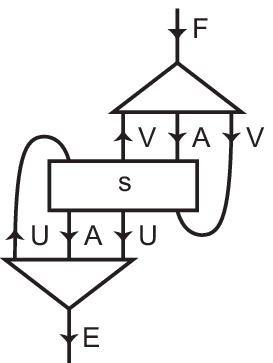} \; \co \varphi_V(A,\sigma) \to \varphi_U(A,\sigma),
$$
is a monoidal natural isomorphism from $\varphi_U$ to~$\varphi_V$. These isomorphisms are related as follows: for any $U,V,W \in \ee_\alpha$,
$$
\delta^{U,V}\delta^{V,W}=\delta^{U,W} \quad \text{and} \quad \delta^{U,U}=\id_{\varphi_U}.
$$

\emph{Step 3.} For $\alpha \in G$, the family  $(\varphi_V, \delta^{U,V})_{U,V \in \ee_\alpha}$ is a
projective system in the category of pivotal strong monoidal $\kk$-linear endofunctors of $\zz_G(\cc)$.
Since all  $\delta^{U,V}$'s are  isomorphisms, this system has a well-defined   projective limit
$$
\varphi_\alpha= \varprojlim (\varphi_V, \delta^{U,V})_{U,V \in \ee_\alpha}
$$
which is  a pivotal strong  monoidal $\kk$-linear endofunctor  of $\zz_G(\cc)$. We can assume that  $\varphi_\alpha\bigl(\zz_\beta(\cc)\bigr) \subset \zz_{\alpha^{-1}\beta\alpha}(\cc)$ for all $ \beta \in G$.
Denote by $\iota^\alpha=\{\iota^\alpha_V\}_{V \in \ee_\alpha}$  the
 universal cone associated with the projective limit above: for $V \in \ee_\alpha$,
$$
\iota^\alpha_V=\{(\iota^\alpha_V)_{(A,\sigma)} \co
\varphi_\alpha(A,\sigma) \to \varphi_V(A,\sigma)\}_{(A,\sigma)\in
\zz_G(\cc)}
$$
is a monoidal natural isomorphism from $\varphi_\alpha$ to $\varphi_V$.

The transformations $\zeta$ and $\eta$ from Step 1 induce monoidal natural isomorphisms $\varphi_2(\alpha,\beta)\co
\varphi_\alpha\varphi_\beta \to \varphi_{\beta\alpha}$   and
$\varphi_0\co 1_{\zz_G(\cc)} \to \varphi_1$, respectively. These
isomorphisms are related to the universal cone as
follows: for $U \in \ee_\alpha$, $V \in \ee_\beta$, $W \in
\ee_{\beta\alpha}$, and $R \in \ee_1$, the following diagrams commute:
$$
    \xymatrix@R=1cm @C=2.5cm {
\varphi_\alpha\varphi_\beta \ar[r]^-{\varphi_2(\alpha,\beta)} \ar[d]_{\varphi_U(\iota^\beta_V)(\iota^\alpha_U)_{\varphi_\beta}} & \varphi_{\beta\alpha}
\ar[d]^{\iota^{\beta\alpha}_W} \\
\varphi_U \varphi_V \ar[r]_-{\zeta^{U,V,W}} & \varphi_W}
\qquad \qquad
\xymatrix@R=1cm @C=.5cm {
& 1_{\zz_G(\cc)} \ar[ld]_-{\varphi_0} \ar[dr]^{\eta^R}   \\
\varphi_1 \ar[rr]_-{\iota^1_R} && \varphi_R.}
$$
Note that $\varphi_2$ and
$\varphi_0$ induce natural isomorphisms  $\varphi_\alpha
\varphi_{\alpha^{-1}} \simeq \varphi_1 \simeq 1_{\zz_G(\cc)}$ and
$\varphi_{\alpha^{-1}} \varphi_\alpha \simeq \varphi_1 \simeq
1_{\zz_G(\cc)}$ for $\alpha \in G$.  Hence, the endofunctor
$\varphi_\alpha$ of $\zz_G(\cc)$ is an equivalence. Therefore
$$\varphi=(\varphi,\varphi_2,\varphi_0) \co \overline{G} \to
\Aut\bigl(\zz_G(\cc)\bigr), \,\, \alpha \mapsto
\varphi_\alpha$$ is a strong monoidal functor such that
$\varphi_\alpha\bigl(\zz_\beta(\cc)\bigr) \subset
\zz_{\alpha^{-1}\beta\alpha}(\cc)$ for all $\alpha,\beta \in G$.
This is the crossing of $\zz_G(\cc)$.

\subsection{The enhanced $G$-braiding}\label{sect-appendix-enhanced-G-braid}
For $V \in \ee$, $(A,\sigma) \in \zz_G(\cc)$,  and $X \in \cch$, set
$$
\Gamma^{V}_{(A,\sigma),X}=  \;\psfrag{A}[Bl][Bl]{\scalebox{.8}{$A$}}
\psfrag{B}[Bl][Bl]{\scalebox{.8}{$X$}}
 \psfrag{V}[Bl][Bl]{\scalebox{.8}{$V$}}
   \psfrag{E}[Bl][Bl]{\scalebox{.8}{$E_{(A,\sigma)}^V$}}
  \psfrag{s}[Bc][Bc]{\scalebox{.9}{$\sigma_{X \otimes V^*}$}}
 \rsdraw{.45}{1}{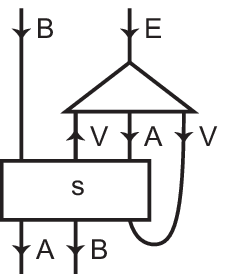}\; \colon A \otimes X  \to X \otimes E_{(A,\sigma)}^V.
$$
This is an isomorphism which is natural in $(A,\sigma) $ and in $X$. Also, for any $U,V \in \ee$ with $|U|=|V|$,  $(A,\sigma) \in \zz_G(\cc)$,  and $X \in \cch$, the following diagram commutes:
$$    \xymatrix@R=1cm @C=.5cm {
& A \otimes X \ar[ld]_-{\Gamma^{V}_{(A,\sigma),X}} \ar[dr]^{\Gamma^{U}_{(A,\sigma),X}}   \\
X \otimes E_{(A,\sigma)}^V \ar[rr]_-{\id_X \otimes \delta^{U,V}} && X \otimes E_{(A,\sigma)}^U.}
$$
Then the transformation $\Gamma$ induces a family of  isomorphisms
$$
 \tau=\{\tau_{(A,\sigma),X} \co A \otimes X \to X \otimes \uu\bigl(\varphi_{|X|}(A,\sigma) \bigr)\}_{(A,\sigma)
  \in \zz_G(\cc), X \in \cch}
$$
which is  natural in $(A,\sigma)  $ and in $X$, where $\uu \co \zz_G(\cc) \to \cc$ is the forgetful functor. The family $\tau$ is related to the universal cones $\{\iota^\alpha\}_{\alpha\in G}$  associated  with $\varphi$ as follows: for any $(A,\sigma) \in \zz_G(\cc)$, $X \in \cch$, and $V \in \ee_{|X|}$,
$$
\bigl(\id_X \otimes (\iota^{|X|}_V)_{(A,\sigma)} \bigr) \tau_{(A,\sigma),X}=\Gamma^V_{(A,\sigma),X}.
$$
We call the family $\tau$ the \emph{enhanced $G$-braiding} in $\zz_G(\cc)$. The enhanced $G$-braiding satisfies properties generalizing that of a $G$-braiding. In particular, it is distributive in each variable with respect to the monoidal product: for all $(A,\sigma), (B,\rho)  \in \zz_G(\cc)$ and all $X,Y \in \cch$,
$$
\tau_{(A,\sigma), X \otimes Y} =(\id_{X \otimes Y} \otimes \varphi_2(|Y|,|X|)_{(A,\sigma)})(\id_X \otimes \tau_{\varphi_{|X|}(A,\sigma),Y})(\tau_{(A,\sigma),X} \otimes \id_Y)
$$
and
\begin{align*}
\tau_{(A,\sigma) \otimes (B,\rho), X} = &
\bigl(\id_X \otimes (\varphi_{|X|})_2((A,\sigma),(B,\rho))
\bigr)\bigl(\tau_{(A,\sigma),X} \otimes
\id_{\varphi_{|X|}(B,\rho)}\bigr) \circ\\
     &\quad \circ \bigl(\id_A \otimes \tau_{(B,\rho),X}\bigr).
\end{align*}

\subsection{The $G$-braiding} The $G$-braiding in $\zz_G(\cc)$ is induced from the enhanced $G$-braiding. More precisely,
for all $(A,\sigma) \in \zz_G(\cc)$ and $(B,\rho) \in \zz_G(\cc)_{\mathrm{hom}}$,
$$
\tau_{(A,\sigma),(B,\rho)}=\tau_{(A,\sigma),B} \co (A,\sigma) \otimes (B,\rho) \to
  (B,\rho) \otimes \varphi_{|(B,\rho)|}(A,\sigma).
$$
is a morphism in $\zz_G(\cc)$. Then the family
$$
\{\tau_{(A,\sigma),(B,\rho)} \}_{(A,\sigma) \in \zz_G(\cc), (B,\rho) \in \zz_G(\cc)_{\mathrm{hom}}}
$$
is a $G$-braiding in  $\zz_G(\cc)$. In particular, $\zz_G(\cc)$ is a $G$-braided category.

\subsection{Ribboness of the $G$-center}
The twist $\theta$ of $\zz_G(\cc)$ is computed as follows: if $(A,\sigma) \in \zz_\alpha(\cc)$ with $\alpha \in G$,
then  for any $U \in \ee_\alpha$,
$$
\theta_{(A,\sigma)}=  \;
 \psfrag{A}[Bl][Bl]{\scalebox{.8}{$A$}}
 \psfrag{B}[Br][Br]{\scalebox{.8}{$A$}}
 \psfrag{R}[Bl][Bl]{\scalebox{.8}{$\varphi_{\alpha}(A,\sigma)$}}
 \psfrag{u}[Bc][Bc]{\scalebox{.8}{$(\iota^\alpha_U)^{-1}_{(A,\sigma)}$}}
 \psfrag{V}[Bl][Bl]{\scalebox{.8}{$U$}}
 \psfrag{E}[Bl][Bl]{\scalebox{.8}{$E_{(A,\sigma)}^U$}}
 \psfrag{t}[Bc][Bc]{\scalebox{.8}{$p_{(A,\sigma)}^U$}}
 \psfrag{s}[Bc][Bc]{\scalebox{1}{$\sigma_{U \otimes U^*}$}}
 \rsdraw{.45}{1}{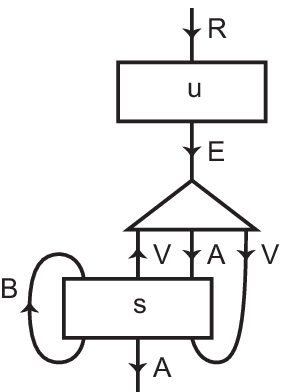}\, : (A,\sigma)\to \varphi_{\alpha}(A,\sigma).
$$
Recall that $\zz_G(\cc)$ is  $G$-ribbon if the twist $\theta$  is self-dual (see Section~\ref{sect-ribbon-graded-cat}). By \cite[Lemma 6.1]{TVi3},
a necessary and sufficient  condition for $\zz_G(\cc)$ to be $G$-ribbon is that
\begin{center}
\psfrag{M}[Br][Br]{\scalebox{.8}{$A$}}
\psfrag{s}[Bc][Bc]{\scalebox{1}{$\sigma_{A \otimes U^*}$}}
\psfrag{A}[Bl][Bl]{\scalebox{.8}{$A$}}
\psfrag{V}[Bl][Bl]{\scalebox{.8}{$U$}}
\rsdraw{.45}{.9}{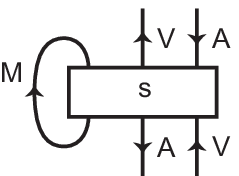} \; $=$ \;
\psfrag{s}[Bc][Bc]{\scalebox{1}{$\sigma_{U^* \otimes A}$}}
\rsdraw{.45}{.9}{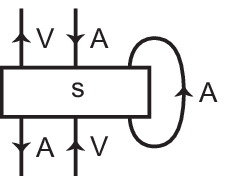}
\end{center}
for all $\alpha \in G$, $(A,\sigma) \in \zz_\alpha(\cc)$, and $U \in \ee_\alpha$.

\end{document}